\documentclass{amsart}
\usepackage{tikz}
\usetikzlibrary{matrix}
\usepackage{hyperref}
\usepackage{amssymb}
\usepackage{amsthm}
\usepackage{enumitem}
\usepackage{tikz-cd}
\usepackage{caption}
\usetikzlibrary{patterns}
\pgfdeclarepatternformonly{mydots}{\pgfqpoint{-1pt}{-1pt}}{\pgfqpoint{1pt}{1pt}}{\pgfqpoint{2pt}{2pt}}% original definition: \pgfqpoint{3pt}{3pt}
{%
  \pgfpathcircle{\pgfqpoint{0pt}{0pt}}{.5pt}%
  \pgfusepath{fill}%
}%

\pgfdeclarepatternformonly{mynewdots}{\pgfqpoint{-1pt}{-1pt}}{\pgfqpoint{1pt}{1pt}}{\pgfqpoint{1.5pt}{1.5pt}}% original definition: \pgfqpoint{3pt}{3pt}
{%
  \pgfpathcircle{\pgfqpoint{0pt}{0pt}}{.5pt}%
  \pgfusepath{fill}%
}%

\pgfdeclarepatternformonly{verynewdots}{\pgfqpoint{-1pt}{-1pt}}{\pgfqpoint{1pt}{1pt}}{\pgfqpoint{2.5pt}{2.5pt}}% original definition: \pgfqpoint{3pt}{3pt}
{%
  \pgfpathcircle{\pgfqpoint{0pt}{0pt}}{.5pt}%
  \pgfusepath{fill}%
}%

\newcommand{\overbar}[1]{\mkern 1.5mu\overline{\mkern-1.5mu#1\mkern-1.5mu}\mkern 1.5mu}

\makeatletter
\newtheorem*{rep@theorem}{\rep@title}
\newcommand{\newreptheorem}[2]{%
\newenvironment{rep#1}[1]{%
 \def\rep@title{#2 \ref{##1}}%
 \begin{rep@theorem}}%
 {\end{rep@theorem}}}
\makeatother

\usepgflibrary{arrows}

\newtheorem*{claim}{Claim}

\theoremstyle{plain}
\newtheorem{theorem}{Theorem}[section]
\newtheorem{lemma}[theorem]{Lemma}
\newtheorem{corollary}[theorem]{Corollary}
\newtheorem{proposition}[theorem]{Proposition}
\newreptheorem{theorem}{Theorem}
\newreptheorem{proposition}{Proposition}
\newreptheorem{corollary}{Corollary}
\newreptheorem{lemma}{Lemma}
\newtheorem{thmx}{Theorem}

\newtheorem{prox}[thmx]{Proposition}

\theoremstyle{definition}
\newtheorem{definition}[theorem]{Definition}
\newtheorem*{question}{Question}
\newtheorem*{convention}{Convention}
\newtheorem{remark}[theorem]{Remark}
\newtheorem{Ex}[theorem]{Example}
\newtheorem*{acknowledgements}{Acknowledgements}
\newtheorem*{background}{Backgrounds and Summary of Results}
\newtheorem*{organization}{Organization of the Paper}
\newtheorem*{RIC}{Reduced Intersection Complex}

\newcommand{\Out}{\mathrm{Out}}

\begin{document}
\title{Quasi-isometry invariants of weakly special square complexes}
\author{Sangrok Oh}
\address{Department of Mathematics, Korea Advanced Institute of Science and Technology, 
		Daejeon, 34141, Korea}
\email{SangrokOh.math@gmail.com}
\begin{abstract}
We define the intersection complex for the universal cover of a compact weakly special square complex and show that it is a quasi-isometry invariant. By using this quasi-isometry invariant, we study the quasi-isometric classification of 2-dimensional right-angled Artin groups and planar graph 2-braid groups. 
Our results cover two well-known cases of 2-dimensional right-angled Artin groups: (1) those whose defining graphs are trees and (2) those whose outer automorphism groups are finite. Finally, we show that there are infinitely many graph 2-braid groups which are quasi-isometric to right-angled Artin groups and infinitely many which are not. 
\end{abstract}
%\date{}                                           % Activate to display a given date or no date

\maketitle
\tableofcontents

%
%
% SECTION 1. Introduction
%
\section{Introduction}\label{1}
\begin{background}
Groups acting geometrically on CAT(0) cube complexes (called \textit{cocompactly cubulated groups}) have a long list of nice properties. 
On the one hand, cocompactly cubulated groups are CAT(0) groups so that they are finitely presented and their word and conjugacy problems are solvable; see Chapter III.$\Gamma$ of \cite{BH} for more properties of CAT(0) groups.
On the other hand, combinatorial properties of CAT(0) cube complexes provide us a more understanding of cocompactly cubulated groups such as the Tits Alternative \cite{SW05}, (bi)automaticity \cite{NR98}, \cite{Swi06}, subgroup separability \cite{HW08}, \cite{HW12}, 3-manifold theory \cite{AGM13}, \cite{Wis09} or rank-rigidity \cite{CS11}.  
We refer to Sageev's notes \cite{Sag12} for a brief story about CAT(0) cube complexes and groups acting on them.

Among cocompactly cubulated groups, the most fundamental and ubiquitous ones are right-angled Artin groups; they are embedded into various groups \cite{CLM10}, \cite{Kob12}, \cite{Tay13}, \cite{KK13}, \cite{BKK14}, \cite{Sab07} and they contain not only well-known groups \cite{CW}, \cite{AGM13}, \cite{HW10} (see \cite{Wis09} for a canonical way to find subgroups in RAAGs and more references) but also complicated group-theoretic subgroups \cite{BB}. Given a simplicial graph $\Lambda$ with vertex set $V(\Lambda)$, the \textit{right-angled Artin group} (RAAG) $A(\Lambda)$ has the following presentation:
$$\langle V(\Lambda)\mid [v_i,v_j]=1~\mathrm{if}~ v_i,v_j\in V(\Lambda)\ \mathrm{are\ joined\ by\ an\ edge\ of\ }\Lambda \rangle.$$
The RAAG $A(\Lambda)$ acts geometrically on a CAT(0) cube complex $X(\Lambda)$ which is the universal cover of a locally CAT(0) cube complex $S(\Lambda)$, called the \textit{Salvetti complex}. The dimension of RAAG is the largest dimension of cubes in $X(\Lambda)$. 
We refer to Charney's article \cite{CH} for an introduction to RAAGs.

Another class of cocompactly cubulated groups is the class of (pure) graph braid groups. Given a graph $\Gamma$, the \textit{graph $n$-braid group} $B_n(\Gamma)$ and the \textit{pure graph $n$-braid group} $PB_n(\Gamma)$ are the $n$-braid group and the pure $n$-braid group over $\Gamma$, respectively. 
By the result of \cite{Abrams00}, if $\Gamma$ is suitably subdivided, then $B_n(\Gamma)$ and $PB_n(\Gamma)$ are the fundamental groups of locally CAT(0) cube complexes $UD_n(\Gamma)$ and $D_n(\Gamma)$, the unordered and ordered discrete configuration spaces of $n$ points on $\Gamma$, respectively (for instance, if $n=2$ and $\Gamma$ is simplicial, then $\Gamma$ is suitably subdivided in this sense). 
In particular, $B_n(\Gamma)$ and $PB_n(\Gamma)$ act geometrically on CAT(0) cube complexes $\overline{UD_n(\Gamma)}$ and $\overline{D_n(\Gamma)}$, the universal covers of $UD_n(\Gamma)$ and $D_n(\Gamma)$, respectively.
By the discrete Morse theory on $D_n(\Gamma)$ and $UD_n(\Gamma)$, presentations and cohomology rings of (pure) graph braid groups are established \cite{FS05}, \cite{FS08}, \cite{KP12}, \cite{KKP12}, \cite{KLP16}.

There were several approaches to clarify algebraic similarities between RAAGs and graph braid groups. 
Besides that both classes of groups act geometrically on CAT(0) cube complexes, they have the following embeddabilities: 
\begin{itemize}
\item
Any graph $n$-braid group $B_n(\Gamma)$ embeds into a RAAG $A(\Lambda)$ \cite{CW}. In this embedding, $\Lambda$ only depends on $\Gamma$.
\item
Any RAAG $A(\Lambda)$ embeds into a pure graph $n$-braid group $PB_n(\Gamma)$ \cite{Sab07}. In this embedding, $\Gamma$ and the index $n$ depend on $\Lambda$.
\end{itemize}
Sometimes, there are isomorphisms between RAAGs and graph braid groups. 
More precisely, if $n\geq4$, then there is a criterion for the graph $\Gamma$ when $B_n(\Gamma)$ is isomorphic to a RAAG \cite{KKP12}, \cite{KLP16}.
Even though criterions for $n=2,3$ are still not known, it is known that when $\Gamma$ is planar, $B_2(\Gamma)$ and $PB_2(\Gamma)$ admit presentations whose relators are commutators \cite{KP12} as RAAGs do.

Based on common properties of planar graph 2-braid groups and RAAGs, one would expect that for any planar graph 2-braid group, there exists a RAAG which is quasi-isometric to the planar graph 2-braid group and vice-versa.
However, it turns out that these classes of groups are not geometrically similar.  

\begin{thmx}\label{MainThm}
There are infinitely many planar graph 2-braid groups which are quasi-isometric to 2-dimensional RAAGs. There are also infinitely many planar graph 2-braid groups which are NOT quasi-isometric to any RAAGs.
\end{thmx}

\textbf{Unless otherwise stated, complexes are assumed to be connected.}

\vspace{0.5mm}

In order to prove Theorem \ref{MainThm}, let us see the large scale geometry of RAAGs and graph 2-braid groups.
In the study about quasi-isometric rigidity of CAT(0) cube complexes or cocompactly cubulated groups, the first step is to see the behavior of top dimensional (quasi-)flats via quasi-isometries.   
In the case of RAAGs and graph 2-braid groups, the following list shows some successes of obtaining rigidity results by this approach.

\begin{remark}\label{123}
\begin{enumerate}
\item 
When two simplicial graphs $\Lambda_1$ and $\Lambda_2$ are atomic (no vertices of valency 1, no cycles of length $< 5$ and no separating closed stars), $A(\Lambda_1)$ and $A(\Lambda_2)$ are quasi-isometric if and only if they are isomorphic \cite{BKS(a)}.
\item 
For two simplicial graphs $\Lambda_1$ and $\Lambda_2$, if the outer automorphism groups of $A(\Lambda_1)$ and $A(\Lambda_2)$ are finite, then $A(\Lambda_1)$ and $A(\Lambda_2)$ are quasi-isometric if and only if they are isomorphic \cite{Hua(a)}.
\item
For two simplicial graphs $\Gamma_1$ and $\Gamma_2$, if $\overline{D_2(\Gamma_1)}$ and $\overline{D_2(\Gamma_2)}$ are quasi-isometric, then the underlying complexes of the intersection complexes of $\overline{D_2(\Gamma_1)}$ and $\overline{D_2(\Gamma_2)}$ are isometric \cite{Fer12}. The definition of an intersection complex is given in Section \ref{3.1}.
\end{enumerate}
\end{remark}

In \cite{Hua(b)}, Huang introduced \textit{weakly special cube complexes}, a slightly larger class than the class of special cube complexes (introduced by Haglund and Wise in \cite{HW08}), which contains both Salvetti complexes and discrete configuration spaces of $n$ points on (possibly non-simplicial) graphs. He showed that any quasi-isometry between the universal covers of two compact weakly special cube complexes preserves top-dimensional flats up to finite Hausdorff distance.
It means that for two simplicial graphs $\Gamma$ and $\Lambda$, if $\overline{D_2(\Gamma)}$ is quasi-isometric to $X(\Lambda)$, then $\Lambda$ has no triangles but at least one edge.
However, there are no more relations between $\Gamma$ and $\Lambda$ directly obtained from results in \cite{Hua(b)} due to the existence of a plenty of arbitrarily constructed flats in $\overline{D_2(\Gamma)}$ and $X(\Lambda)$. We thus define a quasi-isometrically rigid class of subcomplexes (larger than flats) in the universal cover of a compact weakly special square complex.

Given a compact weakly special square complex $Y$, let $\overline{Y}$ be the universal cover of $Y$ with the universal covering map $p_Y:\overline{Y}\rightarrow Y$. 
A \textit{standard product subcomplex} $K$ of $Y$ is defined as the image of a local isometry $\bold{P}_1\times\bold{P}_2\rightarrow Y$ for two non-simply connected graphs $\bold{P}_1$, $\bold{P}_2$ without vertices of valency 1. 
Among subcomplexes in $p^{-1}_{Y}(K)$ which are isometric to the product of two infinite trees, a maximal one $\overbar{K}$ (called a \textit{p-lift} of $K$) is defined as a \textit{standard product subcomplex} $\overbar{K}$ of $\overline{Y}$. Note that $\overbar{K}$ projects to $K$ under $p_Y$.
A \textit{maximal\ product\ subcomplex} of $Y$ or $\overline{Y}$ is the standard product subcomplex which is maximal under the set inclusion. For the details, see Section \ref{2.2}.
Then the pattern of maximal product subcomplexes of $\overline{Y}$ is preserved by a quasi-isometry up to finite Hausdorff distance. 

\begin{thmx}\label{MaxcontainingFlat}
Let $\overline{Y}$ and $\overline{Y'}$ be the universal covers of two compact weakly special square complexes. 
If there is a $(\lambda,\varepsilon)$-quasi-isometry $\phi:\overline{Y}\rightarrow\overline{Y'}$, then there exists a constant $C=C(\lambda,\varepsilon)>0$ such that the following hold: 
\begin{enumerate}
\item 
For any maximal product subcomplex $\overbar{M}\subset\overline{Y}$, there exists a unique maximal product subcomplex of $\overline{Y'}$ denoted by $\phi_H(\overbar{M})$ such that $d_H(\phi(\overbar{M}),\phi_H(\overbar{M}))<C$ where $d_H$ denotes the Hausdorff distance.
\item 
Let $\overbar{M}_i$ be a finite collection of maximal product subcomplexes of $\overline{Y}$.
If the intersection of $\overbar{M}_i$'s contains a standard product subcomplex of $\overline{Y}$, then the intersection of $\phi_H(\overbar{M}_i)$'s contains a standard product subcomplex of $\overline{Y'}$.
\end{enumerate}
\end{thmx}

Let $|I(\overline{Y})|$ be a (possibly disconnected) simplicial complex whose vertex set is the collection of maximal product subcomplexes of $\overline{Y}$ and $(k+1)$ vertices span a $k$-simplex $\triangle$ whenever the $(k+1)$ maximal product subcomplexes corresponding to these vertices share a standard product subcomplex. 
Lemma \ref{IntofStd} implies that each simplex $\triangle\subset I(\overline{Y})$ corresponds to a unique standard product subcomplex $\overbar{K}_\triangle\subset\overline{Y}$ which is at finite Hausdorff distance from the intersection of the $(k+1)$ maximal product subcomplexes.
Under the action $G:=\pi_1(Y)\curvearrowright\overline{Y}$, let $G_\triangle$ be the stabilizer of $\overbar{K}_\triangle$ be the \textit{assigned group} of $\triangle$.
The simplicial complex $|I(\overline{Y})|$ with these assigned groups is called the \textit{intersection complex} of $\overline{Y}$, denoted by $I(\overline{Y})$, and $|I(\overline{Y})|$ is called the \textit{underlying complex} of $I(\overline{Y})$.

Given two compact weakly special square complexes $Y$ and $Y'$ with their universal covers $\overline{Y}$ and $\overline{Y'}$, respectively, Theorem \ref{MaxcontainingFlat} implies that the quasi-isometry $\phi:\overline{Y}\rightarrow\overline{Y'}$ induces a \textit{semi-isomorphism} $\Phi:I(\overline{Y})\rightarrow I(\overline{Y'})$, an isometry preserving a certain relation among assigned groups.
It means that an intersection complex and its underlying complex are quasi-isometry invariants for the universal cover of a compact weakly special square complex but it turns out that the intersection complex is a more sensitive invariant. 

\begin{thmx}\label{QIimpliesIso1}
Let $\overline{Y}$ and $\overline{Y'}$ be the universal covers of two compact weakly special square complexes. 
If there is a quasi-isometry $\phi:\overline{Y}\rightarrow\overline{Y'}$, then $\phi$ induces a semi-isomorphism $\Phi:I(\overline{Y})\rightarrow I(\overline{Y'})$ via Theorem \ref{MaxcontainingFlat}. 
\end{thmx}

\begin{prox}\label{1.5'}
There are two compact weakly special square complexes such that for their universal covers $\overline{Y}$ and $\overline{Y'}$, $|I(\overline{Y})|$ and $|I(\overline{Y'})|$ are isometric but $I(\overline{Y})$ and $I(\overline{Y'})$ are not semi-isomorphic, i.e. there is no semi-isomorphism between $I(\overline{Y})$ and $I(\overline{Y'})$. In particular, $\overline{Y}$ and $\overline{Y'}$ are not quasi-isometric.
\end{prox}

Theorem \ref{QIimpliesIso1} is our starting point to know whether there exist planar graph 2-braid groups which are quasi-isometric to 2-dimensional RAAGs. 
After investigating topological properties of $I(\overline{D_2(\Gamma)})$ for planar simplicial graphs $\Gamma$, we find an infinite collection $\mathcal{G}$ of planar simplicial graphs such that if $\Gamma\in\mathcal{G}$, then every component of $I(\overline{D_2(\Gamma)})$ is semi-isomorphic to $I(X(T))$ where $T$ is a tree of diameter $\geq 3$. 
By algebraic properties of planar graph 2-braid groups and the theory of complexes of groups, we show that this semi-isomorphism induces a quasi-isometry between $PB_2(\Gamma)$ (or $B_2(\Gamma)$) and $A(T)*\mathbb{Z}$ where $\mathbb{Z}$ is an infinite cyclic group.
We also find an infinite collection $\mathcal{G}'$ of planar simplicial graphs such that if $\Gamma'\in\mathcal{G}'$, then there is no $I(X(\Lambda))$ which is semi-isomorphic to any component of $I(\overline{D_2(\Gamma')})$. 
This is the sketch of the proof of Theorem \ref{MainThm}.
\vspace{1mm}

Theorem \ref{QIimpliesIso1} covers not only the result about graph 2-braid groups but also the results about 2-dimensional RAAGs in Remark \ref{123}. See Section \ref{5.1} for the case of 2-dimensional RAAGs. A result in \cite{BN} can also be covered using the concept of intersection complexes.

\begin{reptheorem}{Tree}[Theorem 5.3 in \cite{BN}]
Let $T$ be a tree of diameter $\geq 3$. If a RAAG $A(\Lambda)$ is quasi-isometric to $A(T)$, then $\Lambda$ is a tree of diameter $\geq 3$.
\end{reptheorem} 

This perspective can also give a new insight on quasi-isometric rigidity of 2-dimensional RAAGs. 
Suppose that a triangle-free simplicial graph $\Lambda$ is the union of subgraphs $\Lambda_i$ such that
\begin{enumerate}
\item for each $i$, the outer autormophism group of $A(\Lambda_i)$ is finite, and 
\item the intersection of any two of these subgraphs is either a vertex or an empty set, and the intersection of any three is empty.
\end{enumerate}
In this case, let us call such a subgraph $\Lambda_i$ a \textit{block}.
We show that the collection of isometry classes of blocks is a quasi-isometry invariant. 

\begin{thmx}\label{1.7'}
Let $\Lambda$ ($\Lambda'$, resp.) be a simplicial graph which is the union of blocks $\Lambda_i$ ($\Lambda'_j$, resp.). 
Suppose that $\Lambda$ and $\Lambda'$ have no finite sequence of blocks such that the first and last ones are the same and the intersection of two consecutive ones is a vertex.
Let $\mathcal{I}$ ($\mathcal{I}'$, resp.) be the collection of isometry classes of $\Lambda_i$'s ($\Lambda'_j$'s, resp.). If $\mathcal{I}$ and $\mathcal{I}'$ are different, then $A(\Lambda)$ and $A(\Lambda')$ are not quasi-isometric. 
\end{thmx}
\end{background}

\begin{RIC}
As $|I(\overline{Y})|$ is obtained from $\overline{Y}$, a (possibly disconnected) complex of simplices $|RI(Y)|$ is similarly obtained from $Y$. First, the vertex set of $|RI(Y)|$ is the collection of maximal product subcomplexes of $Y$. 
Suppose that the intersection $W$ of $(k+1)$ maximal product subcomplexes of $Y$ contains a finite collection of standard product subcomplexes $K_i$ such that each $K_i$ is maximal under the set inclusion among standard product subcomplexes contained in $W$. 
Then the vertices in $|RI(Y)|$ corresponding to these $(k+1)$ maximal product subcomplexes span exactly $m$ $k$-simplices in $|RI(Y)|$ each of which corresponds to one of $K_i$'s.  
For the standard product subcomplex $K_{\triangle'}\subset Y$ corresponding to each simplex $\triangle'\subset|RI(Y)|$, there exists a simplex $\triangle\subset |I(\overline{Y})|$ which corresponds to a p-lift of $K_{\triangle'}$ in $\overline{Y}$.
Then the canonical action $G=\pi_1(Y)\curvearrowright\overline{Y}$ induces the action $G\curvearrowright|I(\overline{Y})|$ by isometries (Theorem \ref{TPBCM}).
To each simplex $\triangle'\subset |RI(Y)|$, we assign the conjugacy class of $G_\triangle$ and the complex $|RI(Y)|$ with these assigned groups is called a \textit{reduced intersection complex} of $Y$, denoted by $RI(Y)$.

Both $RI(Y)$ and $I(\overline{Y})$ are defined as complexes of groups, the assigned group of any simplex of which is isomorphic to the direct product of two finitely generated free groups (especially called a \textit{join group}). 
To emphasize this fact, we define a class of complexes of groups, called complexes of join groups. A \textit{complex of join groups} is a complex of groups whose cells are simplices and the assigned groups are join groups with specific properties. 
A \textit{semi-morphism} between complexes of join groups is a map such that (1) it becomes a combinatorial map if we ignore the assigned groups and (2) it preserves a certain inclusion structure of assigned groups. 
The precise definitions and properties related to complexes of join groups and semi-(iso)morphisms are established in Section \ref{3.2}.

\begin{thmx}\label{TPBCM3}
The action $\pi_1(Y)\curvearrowright I(\overline{Y})$ induced from the action $\pi_1(Y)\curvearrowright \overline{Y}$ is by semi-isomorphisms and the quotient map $\rho_Y:I(\overline{Y})\rightarrow RI(Y)$ is a semi-morphism.
\end{thmx}

Based on Theorem \ref{TPBCM3} and algebraic properties of $\pi_1(Y)$, $I(\overline{Y})$ can be constructed from $RI(Y)$.
Let us see the construction case by case when $Y$ is either a 2-dimensional Salvetti complex $S(\Lambda)$ or $D_2(\Gamma)$ for a special kind of graph $\Gamma$.

\vspace{0.5mm}

In the case of 2-dimensional RAAGs $A(\Lambda)$, every standard product subcomplex of $S(\Lambda)$ is a standard subcomplex whose defining graph has a join decomposition. 
Since $S(\Lambda)$ has a unique vertex, the assigned group of each simplex $\triangle'\subset RI(S(\Lambda))$ is the canonical subgroup of $A(\Lambda)$ generated by the defining graph of the standard product subcomplex $K_{\triangle'}$ corresponding to $\triangle'$. This fact induces that $I(X(\Lambda))$ can be constructed from the copies of $RI(S(\Lambda))$ (Remark \ref{ConstructionOfX}).

By the construction, $I(X(\Lambda))$ can be considered as a realization of $A(\Lambda)$ endowed with the metric induced from the join length. Let $\mathcal{J}(\Lambda)$ be the collection of all the subgraphs of $\Lambda$ which have join decompositions. 
The \textit{join length} of $g\in A(\Lambda)$ is the minimum $l$ such that $g=g_1\cdots g_l$ where $g_i$ is in the subgroup $A(\Lambda')\leq A(\Lambda)$ for some $\Lambda'\in\mathcal{J}(\Lambda)$.
We then obtain a similar result in \cite{Ha11}, \cite{KK14}.

\begin{thmx}\label{Quasi-tree2}
Let $\Lambda$ be a triangle-free graph. Then $I(X(\Lambda))$ is a quasi-tree. In particular, $A(\Lambda)$ is weakly hyperbolic relative to $\{A(\Lambda')\ |\ \Lambda'\in \mathcal{J}(\Lambda)\}$.  
\end{thmx}

In the case of pure 2-braid groups over simplicial graphs $\Gamma$, every standard product subcomplex of $D_2(\Gamma)$ can be represented by the product of disjoint two subgraphs of $\Gamma$.
A \textit{cactus} is a graph in which any two cycles have at most one vertex in common; if any two cycles do not intersect and each cycle $a$ intersects the closure of the complement of $a$, viewed as topological spaces, at a single vertex, then the cactus will be said to be \textit{simplest}.
If $\Gamma$ is a simplest cactus, then every standard product subcomplex of $D_2(\Gamma)$ can be represented by the product of two disjoint subcollections of the collection of all cycles in $\Gamma$. Let us see an example.

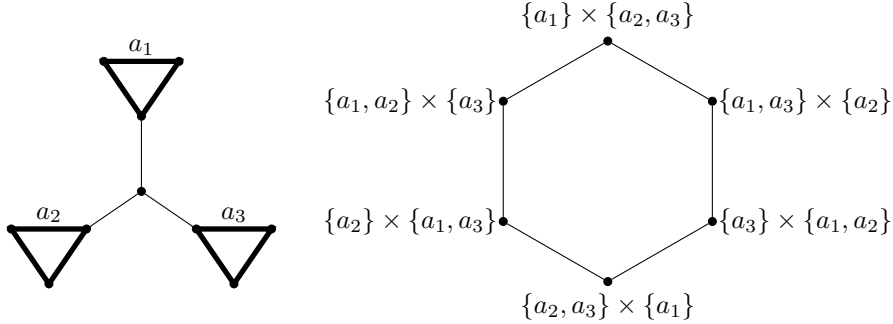
\begin{figure}[h]
\begin{tikzpicture}
\tikzstyle{every node}=[draw,circle,fill=black,minimum size=3pt,inner sep=0pt]
  \node[anchor=center] (x) at (0,0){};
  \node[anchor=center] (a1) at (0,1)[label={[shift={(0,0.65)}]:$a_1$}]{};
  \node[anchor=center] (a2) at (0.5,1.73){};  
  \node[anchor=center] (a3) at (-0.5,1.73){};  
  \node[anchor=center] (b1) at (0.73,-0.5){};
  \node[anchor=center] (b2) at (1.73,-0.5){};  
  \node[anchor=center] (b3) at (1.23,-1.23)[label={[shift={(0,0.65)}]:$a_3$}]{};  
  \node[anchor=center] (c1) at (-0.73,-0.5){};
  \node[anchor=center] (c2) at (-1.73,-0.5){};  
  \node[anchor=center] (c3) at (-1.23,-1.23)[label={[shift={(0,0.65)}]:$a_2$}]{};  
  \foreach \from/\to in {x/a1,a1/a2,a2/a3,a3/a1,x/b1,b1/b2,b2/b3,b3/b1,x/c1,c1/c2,c2/c3,c3/c1}
  \draw (\from) -- (\to);  
   \draw[line width=2pt] (a1.center) -- (a2.center) -- (a3.center) -- cycle;
   \draw[line width=2pt] (b1.center) -- (b2.center) -- (b3.center) -- cycle;
   \draw[line width=2pt] (c1.center) -- (c2.center) -- (c3.center) -- cycle;
  
  \draw (6.2,2) node (p1) [label={[label distance=-0.9cm]above:$\{a_1\}\times\{a_2,a_3\}$}]{};
  \draw (7.59,1.2) node (p2) [label={[label distance=0cm]right:$\{a_1,a_3\}\times\{a_2\}$}]{};
  \draw (7.59,-0.4) node (p3) [label={[label distance=0cm]right:$\{a_3\}\times\{a_1,a_2\}$}]{};
  \draw (6.2,-1.2) node (p4) [label={[label distance=-0.9cm]below:$\{a_2,a_3\}\times\{a_1\}$}]{};
  \draw (4.81,-0.4) node (p5) [label={[label distance=0cm]left:$\{a_2\}\times\{a_1,a_3\}$}]{};
  \draw (4.81,1.2) node (p6) [label={[label distance=0cm]left:$\{a_1,a_2\}\times\{a_3\}$}]{};
  \foreach \from/\to in {p1/p2,p2/p3,p3/p4,p4/p5,p5/p6,p6/p1}
  \draw (\from) -- (\to);
\end{tikzpicture}
\caption{$\mathcal{O}_{3}$ and $RI(D_2(\mathcal{O}_{3}))$.}
\label{Ex1}
\end{figure}

\begin{Ex}\label{T_3Ex}
The simplest cactus $\mathcal{O}_{3}$ in the left of Figure \ref{Ex1} has three cycles $a_1$, $a_2$, $a_3$. There are six maximal product subcomplexes of $D_2(\mathcal{O}_3)$ represented by  
\begin{center}
$\{a_2,a_3\}\times\{a_1\},\ \{a_2\}\times\{a_1,a_3\},\ \{a_1,a_2\}\times\{a_3\}$,\\
$\{a_1\}\times\{a_2,a_3\},\ \{a_1,a_3\}\times\{a_2\},\ \{a_3\}\times\{a_1,a_2\}$.
\end{center}
Let $\mathbf{u}_0$, $\mathbf{u}_1$ and $\mathbf{u}_2$ be the vertices in $RI(D_2(\mathcal{O}_{3}))$ which correspond to $\{a_2,a_3\}\times\{a_1\}$, $\{a_2\}\times\{a_1,a_3\}$ and $\{a_1,a_2\}\times\{a_3\}$, respectively, and let $\mathbf{u}_3$, $\mathbf{u}_4$ and $\mathbf{u}_5$ be the vertices which are obtained from $\mathbf{u}_0$, $\mathbf{u}_1$ and $\mathbf{u}_2$ by switching the coordinates, respectively. 
Then $RI(D_2(\mathcal{O}_{3}))$ is a hexagon with vertices $\mathbf{u}_0,\cdots,\mathbf{u}_5$ in this order. See the right in Figure \ref{Ex1}. 

Let $I_0$ be a component of $I(\overline{D_2(\mathcal{O}_3)})$, and let $\rho_0:I_0\rightarrow RI(D_2(\mathcal{O}_{3}))$ be the restriction of the quotient map $\rho_{D_2(\mathcal{O}_3)}:I(\overline{D_2(\mathcal{O}_3)})\rightarrow RI(D_2(\mathcal{O}_{3}))$ obtained in Theorem \ref{TPBCM3}.
Let $\mathbf{v}\in I_0$ be a vertex whose image under $\rho_0$ is $\mathbf{u}_4$.
Then $\mathbf{v}$ corresponds to a maximal product subcomplex $\overbar{M}_{\mathbf{v}}\subset\overline{D_2(\mathcal{O}_{3})}$ which is a p-lift of $\{a_1,a_3\}\times\{a_2\}$.
Along every standard flat in $\overbar{M}_{\mathbf{v}}$ which is a p-lift of $\{a_1\}\times \{a_2\}$, a maximal product subcomplex which is a p-lift of $\{a_1\}\times\{a_2,a_3\}$ is attached to $\overbar{M}_{\mathbf{v}}$. 
Along every standard flat in $\overbar{M}_{\mathbf{v}}$ which is a p-lift of $\{a_3\}\times \{a_2\}$, a maximal product subcomplex which is a p-lift of $\{a_3\}\times\{a_1,a_2\}$ is attached to $\overbar{M}_{\mathbf{v}}$. 
Thus, there are infinitely many edges of $I_0$ attached to $\mathbf{v}$ such that the other endpoint $\mathbf{v}'$ of each edge corresponds to a maximal product subcomplex which is a p-lift of either $\{a_1\}\times\{a_2,a_3\}$ or $\{a_3\}\times\{a_1,a_2\}$ as the left in Figure \ref{Ex1RI}, i.e. $\rho_0(\mathbf{v}')$ is either $\mathbf{u}_3$ or $\mathbf{u}_5$.

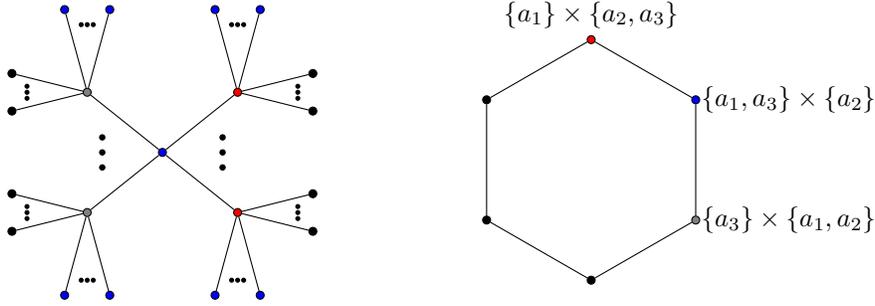
\begin{figure}[h]
\begin{tikzpicture}
\tikzstyle{every node}=[draw,circle,fill=black,minimum size=3pt,inner sep=0pt]
\tikzstyle{selected edge} = [draw,line width=5pt,-,red!50]
\tikzstyle{selected vertex} = [vertex, fill=red!24]
  \node[fill=blue] (x) at (0,0){};
  \node[fill=red] (a1) at (1,0.8){};
  \node[fill=red] (a2) at (1,-0.8){};  
  \node[minimum size=2pt] (a3) at (0.8,0.2){};
  \node[minimum size=2pt] (a4) at (0.8,0){};
  \node[minimum size=2pt] (a5) at (0.8,-0.2){};  
  \node[fill=gray] (b1) at (-1,0.8){};
  \node[fill=gray] (b2) at (-1,-0.8){};  
  \node[minimum size=2pt] (b3) at (-0.8,-0.2){};
  \node[minimum size=2pt] (b4) at (-0.8,0){};
  \node[minimum size=2pt] (b5) at (-0.8,0.2){};  
  \node[anchor=center] (c1) at (-2,-0.55){};
  \node[anchor=center] (c2) at (-2,-1.05){};
  \node[minimum size=1.5pt] (c8) at (-1.8,-0.72){};
  \node[minimum size=1.5pt] (c9) at (-1.8,-0.8){};  
  \node[minimum size=1.5pt] (c10) at (-1.8,-0.88){};
  \node[anchor=center] (c11) at (-2,0.55){};
  \node[anchor=center] (c12) at (-2,1.05){};
  \node[minimum size=1.5pt] (c13) at (-1.8,0.72){};
  \node[minimum size=1.5pt] (c14) at (-1.8,0.8){};  
  \node[minimum size=1.5pt] (c15) at (-1.8,0.88){};
  \node[fill=blue] (c3) at (-1.3,-1.9){};
  \node[fill=blue] (c4) at (-0.7,-1.9){};
  \node[minimum size=1.5pt] (c5) at (-1.08,-1.7){};
  \node[minimum size=1.5pt] (c6) at (-1,-1.7){};
  \node[minimum size=1.5pt] (c7) at (-0.92,-1.7){};
  \node[anchor=center] (d1) at (2,-0.55){};
  \node[anchor=center] (d2) at (2,-1.05){}; 
  \node[minimum size=1.5pt] (g1) at (1.8,-0.72){};
  \node[minimum size=1.5pt] (g2) at (1.8,-0.8){};  
  \node[minimum size=1.5pt] (g3) at (1.8,-0.88){};
  \node[fill=blue] (d3) at (1.3,-1.9){};
  \node[fill=blue] (d4) at (0.7,-1.9){};
  \node[minimum size=1.5pt] (d5) at (1.08,-1.7){};
  \node[minimum size=1.5pt] (d6) at (1,-1.7){};
  \node[minimum size=1.5pt] (d7) at (0.92,-1.7){};
  \node[anchor=center] (d8) at (2,0.55){};
  \node[anchor=center] (d9) at (2,1.05){};
  \node[minimum size=1.5pt] (d10) at (1.8,0.72){};
  \node[minimum size=1.5pt] (d11) at (1.8,0.8){};  
  \node[minimum size=1.5pt] (d12) at (1.8,0.88){};
  \node[fill=blue] (e1) at (1.3,1.9){};
  \node[fill=blue] (e2) at (0.7,1.9){};
  \node[minimum size=1.5pt] (e3) at (1.08,1.7){};
  \node[minimum size=1.5pt] (e4) at (1,1.7){};
  \node[minimum size=1.5pt] (e5) at (0.92,1.7){};
  \node[fill=blue] (f1) at (-1.3,1.9){};
  \node[fill=blue] (f2) at (-0.7,1.9){};
  \node[minimum size=1.5pt] (f3) at (-1.08,1.7){};
  \node[minimum size=1.5pt] (f4) at (-1,1.7){};
  \node[minimum size=1.5pt] (f5) at (-0.92,1.7){};
  \foreach \from/\to in {x/a1,x/a2,x/b1,x/b2,b2/c1,b2/c2,b2/c3,b2/c4,a2/d1,a2/d2,c11/b1,b1/c12,a1/d8,a1/d9,a1/e1,a1/e2,b1/f1,b1/f2,a2/d3,a2/d4}
  \draw (\from) -- (\to);
  
  \node[fill=red] (p1) at (5.7,1.5) [label={[label distance=-0.9cm]above:$\{a_1\}\times\{a_2,a_3\}$}]{};
  \node[fill=blue] (p2) at (7.09,0.7) [label={[label distance=0cm]right:$\{a_1,a_3\}\times\{a_2\}$}]{};
  \node[fill=gray] (p3) at (7.09,-0.9) [label={[label distance=0cm]right:$\{a_3\}\times\{a_1,a_2\}$}]{};
  \draw (5.7,-1.7) node (p4) {};
  \draw (4.31,-0.9) node (p5) {};
  \draw (4.31,0.7) node (p6) {};
  \foreach \from/\to in {p1/p2,p2/p3,p3/p4,p4/p5,p5/p6,p6/p1}
  \draw (\from) -- (\to);

\end{tikzpicture}
\caption{A component of $I(\overline{D_2(\mathcal{O}_3)})$ and $RI(D_2(\mathcal{O}_3))$}
\label{Ex1RI}
\end{figure}
\end{Ex}

In Example \ref{T_3Ex}, $RI(D_2(\mathcal{O}_3))$ can be considered as a graph of groups and $I_0$ is the Bass-Serre tree associated to $RI(D_2(\mathcal{O}_3))$.
In general, we can apply the theory of complexes of groups to $RI(D_2(\Gamma))$ in order to obtain the following fact.

\begin{thmx}\label{pi_1-injective2}
For a simplest cactus $\Gamma$, let $I_0(\overline{D_2(\Gamma)})$ be a component of $I(\overline{D_2(\Gamma)})$ and $SPB_2(\Gamma)$ the stabilizer of $I_0(\overline{D_2(\Gamma)})$ under the induced action $PB_2(\Gamma)\curvearrowright I(\overline{D_2(\Gamma)})$. Then $RI(D_2(\Gamma))$ can be considered as the developable complex of groups decomposition of $SPB_2(\Gamma)$ and $I_0(\overline{D_2(\Gamma)})$ is the development of $RI(D_2(\Gamma))$.   
\end{thmx}

The \textit{development} of a developable complex of groups is analogous to the Bass-Serre tree associated to a graph of groups. We refer to \cite{BH} or \cite{LM} for the theory of (developable) complexes of groups. How the theory is applied to a reduced intersection complex and how Theorem \ref{pi_1-injective2} is obtained are written in Section \ref{3.2} and Section \ref{4.2}, respectively.

\vspace{1mm}

Based on their properties, (reduced) intersection complexes can be used to study quasi-isometric rigidity of 2-dimensional RAAGs and planar graph 2-braid groups. In particular, we have the following question:

\begin{question}
For planar simplicial graphs $\Gamma_1$ and $\Gamma_2$, suppose that $|I(\overline{D_2(\Gamma_1)})|$ and $|I(\overline{D_2(\Gamma_2)})|$ are simply connected. 
Then, are $\overline{D_2(\Gamma_1)}$ and $\overline{D_2(\Gamma_2)}$ quasi-isometric when $I(\overline{D_2(\Gamma_1)})$ and $I(\overline{D_2(\Gamma_2)})$ are semi-isomorphic?
\end{question}
\end{RIC}

\begin{organization}
Section \ref{2} contains preliminary material, especially about weakly special square complexes.
Section \ref{3} consists of two subsections: the first one contains the definitions and topological properties of (reduced) intersection complexes, and the second one talks about complexes of join groups and semi-morphisms between them.
In Section \ref{4}, we see the construction of $RI(Y)$ from $Y$ and the construction of $I(\overline{Y})$ from $RI(Y)$ when $Y$ is either $S(\Lambda)$ for a triangle-free graph $\Lambda$ or $D_2(\Gamma)$ for a simplest cactus $\Gamma$.
Finally, in Section \ref{5}, we see quasi-isometric rigidity results about 2-dimensional RAAGs and planar graph 2-braid groups which are deduced from semi-morphisms between (reduced) intersection complexes.
\end{organization}

\begin{acknowledgements}
This paper which is part of the author's PhD thesis was started under the supervision of Kihyoung Ko and finished under the supervision of Hyungryul Baik. The author thanks Kihyoung Ko and Hyungryul Baik for helpful suggestions and stimulating discussions. The author thanks Sam Shephard for pointing out mistakes and discussing how to fix them. The author thanks Huang Jingyin, Hyowon Park, Philippe Tranchida for numerous useful conversations. The author also thanks the referee for carefully reading the paper and providing helpful comments.
\end{acknowledgements}

%%%%%%%%%%%%%%%%%%%%%%%%%%
%
%
%
%              CHAPTER 2. Preliminaries
%
%
%
%%%%%%%%%%%%%%%%%%%%%%%%%%

\section{Preliminaries}\label{2}

In this paper, we deal with two types of finite dimensional polyhedral complexes: (1) complexes of simplices and (2) complexes of cubes (or cube complexes). 
Unless otherwise stated, polyhedral complexes (or their subcomplexes) are assumed to be connected and by each cell, we mean its closure. 
Among 1-dimensional polyhedral complexes, `\textit{graphs}' denote finite ones and `\textit{trees}' denote (possibly infinite) simply connected ones; a \textit{star tree} is a tree of diameter 2. 
For a (closed) subset $S$ of a polyhedral complex $X$, the \textit{full subcomplex of $X$ spanned by $S$} is the union of cells of $X$ whose vertices are contained in $S$.
A \textit{combinatorial map} $f:X\rightarrow Y$ between two polyhedral complexes $X$, $Y$ of the same type is a (continuous) map which takes the interior of each cell of $X$ homeomorphically onto the interior of a cell of $Y$. If $f$ is injective (locally injective, resp.), then $f$ is said to be a \textit{combinatorial embedding (combinatorial immersion, resp.)}.
\textbf{Unless otherwise stated, any map between polyhedral complexes is assumed to be combinatorial.}

Let us introduce several graph-related terminologies which we will frequently use: `leaf', `standard (sub)graph', `standard infinite tree', `boundary cycle' and `(simplest) cactus'. A \textit{leaf} is a vertex of valency 1. 
A \textit{standard graph} denotes a non-simply connected graph without leaves and a \textit{cycle} is a standard graph whose fundamental group is isomorphic to an infinite cyclic group. If a standard graph has the fundamental group isomorphic to a non-abelian free group, then its universal cover is said to be a \textit{standard infinite tree} (we use this term to denote the universal cover of a standard subgraph which is not a cycle).
If a cycle in a graph is a full subgraph, then it is said to be an \textit{induced cycle}.
If a graph $\Gamma$ is planar, then $\Gamma$ can be considered as the image of an embedding $\Gamma\hookrightarrow \mathbb{R}^2$: For the closures $A_1, \cdots, A_r$ of connected components of $\mathbb{R}^2-\Gamma$, the boundary $\partial{A_i}$ of each $A_i$ is called a \textit{boundary cycle}. 
A \textit{cactus} is a planar graph in which the intersection of any two (boundary) cycles is an empty set or a vertex; if a cactus is obtained from a tree by attaching at most one cycle to each vertex of the tree, then the cactus will be said to be \textit{simplest}.

The definition of the link of a vertex in a complex of simplices is a little bit different from the definition of the link of a vertex in a cube complex. 
\begin{enumerate}
\item For a complex of simplices $X$, the link of a vertex $v\in X$, denoted by $lk_{X}(v)$, is the full subcomplex spanned by all the vertices adjacent to $v$, and the star of $v$, denoted by $st_{X}(v)$, is the full subcomplex spanned by $v$ itself and vertices in $lk_X(v)$ (we will often omit $X$ when the complex $X$ is clear). Note that $st_{X}(v)$ is always connected but $lk_{X}(v)$ is possibly disconnected.
\item For a cube complex $Y$, the link of a vertex $y\in Y$, denoted by $Lk(y,Y)$, is the $\epsilon$-sphere about $y$ which is considered as a complex of simplices; an $n$-simplex of $Lk(y,Y)$ corresponds to the corner of an $(n+1)$-cube of $Y$ containing $y$ and simplices of $Lk(y,Y)$ are attached along their faces according to the attachment of the corresponding corners of cubes.
\end{enumerate}

The terms `maximal' and `minimal' are used as usual: an object satisfying a property is said to be $\it{maximal}$ ($\it{minimal}$, resp.) if it is maximal (minimal, resp.) under the set inclusion among all the objects satisfying the given property.

Throughout this paper, for the universal cover $\overline{Y}$ of a cube complex $Y$, the universal covering map is denoted by $p_Y:\overline{Y}\rightarrow Y$.
A free group of rank $n$ is denoted by $\mathbb{F}_n$; especially, an infinite cyclic group ($n=1$) and a non-abelian finitely generated free group ($n>1$) will be denoted by $\mathbb{Z}$ and $\mathbb{F}$, respectively.

%
% 2.1 Preliminaries
%
\subsection{CAT(0) Cube Complexes}\label{2.1}
Let $Y$ be a cube complex. The dimension of $Y$ is the largest dimension of cubes in $Y$; if the dimension is 2, then $Y$ is especially called a \textit{square complex}. 
A cube complex $Y$ is said to be non-positively curved (NPC) if, for each vertex $y\in Y$, its link $Lk(y,Y)$ is a flag complex. 
If $Y$ is simply connected, then $Y$ is said to be a CAT(0) \textit{cube complex}. We refer to \cite{HW08} and \cite{Wis12} for more about (CAT(0)) cube complexes.

If each $n$-cube is endowed with the standard metric of the unit cube $I^n$ in Euclidean $n$-space $\mathbb{E}^n$, then a cube complex $Y$ becomes a length space with the path metric $d$. 
Every CAT(0) cube complex endowed with this path metric is actually a CAT(0) space; Gromov showed that a finite dimensional cube complex is NPC if and only if it is locally CAT(0) with this path metric \cite{Grom} and Leary extended this result to the infinite dimensional case \cite{Lea}. 
As a consequence, an NPC cube complex with the path metric is a CAT(0) space if it is simply connected by the Cartan-Hadamard theorem.
We refer to \cite{BH} for more about CAT(0) geometry and the relation between the combinatorial and geometric definitions of (locally) CAT(0) cube complexes. 

Let $\phi:X\rightarrow Y$ be a combinatorial map between two cube complexes $X$ and $Y$.
Then $\phi$ is said to be a \textit{combinatorial local isometry} if, for each vertex $x\in X$, the combinatorial link map $Lk(x,X)\rightarrow Lk(\phi(x),Y)$ induced by $\phi$ is injective and its image is a full subcomplex of $Lk(\phi(x),Y)$. Additionally, if $\phi$ is a combinatorial embedding, then $\phi$ is said to be a \textit{combinatorial locally isometric embedding} and the image of $\phi$ is said to be \textit{locally convex} in $Y$. 
The following theorem tells that if $Y$ is NPC, in these combinatorial definitions, the term `combinatorial' can be dropped.

\begin{theorem}[\cite{CW}, \cite{HW08}]\label{CW}
Let $\phi:X\rightarrow Y$ be a combinatorial map between two finite dimensional cube complexes $X$ and $Y$. If $Y$ is NPC, then the map $\phi$ is a combinatorial local isometry if and only if $\phi$ is a (geodesical) local isometry. If $Y$ is NPC and $\phi$ is a combinatorial local isometry, then $X$ is also NPC.
\end{theorem}

Two main properties of the path metric on an NPC cube complex are the following:
First, for two NPC cube complexes $X$, $Y$ with their universal covers $\overline{X}$, $\overline{Y}$, respectively, a local isometry $\phi:X\rightarrow Y$ induces an injective homomorphism $\phi_*:\pi_1(X,x)\rightarrow\pi_1(Y,\phi(x))$ and an \textit{elevation} $\overline{\phi}:\overline{X}\hookrightarrow \overline{Y}$, an isometric embedding which covers $\phi$ as in Figure \ref{Elevation} (see Chapter II.4 in \cite{BH}).
Second, the path metric on a CAT(0) cube complex is convex so that the nearest point projection onto a convex subcomplex is well-understood.

\vspace{-1mm}

\begin{figure}

\begin{tikzpicture}
  \matrix (m) [matrix of math nodes,row sep=3em,column sep=6em,minimum width=3em]
  { \overline{X} & \overline{Y} \\
    X & Y \\};
  \path[-stealth]
    (m-1-1) edge node [left] {$p_{X}$} (m-2-1)
            edge node [above] {$\overline{\phi}$} (m-1-2)
    (m-2-1.east|-m-2-2) edge node [above] {$\phi$} (m-2-2)
 %           node [above] {$\exists$} (m-2-2)
    (m-1-2) edge node [right] {$p_{Y}$} (m-2-2);
\end{tikzpicture}
	\caption{An elevation of a local isometry}
    	\label{Elevation}
\end{figure}

\begin{lemma}[\cite{BKS(a)}, \cite{Hua(b)}]\label{Projection}
Let $Y$ be an $n$-dimensional $\rm{CAT(0)}$ cube complex with the path metric $d$. Let $U_i$ be a convex subcomplex of $Y$ with the nearest point projection $\pi_{U_i}$ of $Y$ onto $U_i$ for $i=1,2$.
Suppose that $D=d(U_1,U_2)$, $V_1=\{y\in U_1\ |\ d(y,U_2)=D \}$ and $V_2=\{y\in U_2\ |\ d(y,U_1)=D \}$. Then,
\begin{enumerate}
\item $V_1$ and $V_2$ are non-empty convex subcomplexes.
\item 
$\pi_{U_1}$ maps $V_2$ isometrically onto $V_1$ and $\pi_{U_2}$ maps $V_1$ isometrically onto $V_2$. Moreover, the convex hull of $V_1\cup V_2$ is isometric to $V_1\times[0,D]$.
In other words, $V_1$ and $V_2$ are $\mathrm{parallel}$.
\item 
There exists a constant $A=A(D,n,\epsilon)\leq 1$ such that if $p_1\in U_1$, $p_2\in U_2$ and $d(p_i,V_i)\geq \epsilon>0$, then
$$d(p_1, U_2)\geq D + Ad(p_1,V_1),\ d(p_2, U_1)\geq D + Ad(p_2,V_2).$$
For $r\geq D$, let $r'=D+(2r-D)/A$. Then, $N_{r}(U_1)\cap N_r(U_2)\subset N_{r'}(V_i)$.
\end{enumerate}
\end{lemma}

Let us come back to a cube complex $Y$ and define a combinatorial object, called a hyperplane. 
A $\it{midcube}$ $\sigma$ in a unit cube $I^n$ is an $(n-1)$-dimensional unit cube which contains the barycenter of $I^n$ and is parallel to one of $(n-1)$-dimensional faces of $I^n$. A midcube $\sigma$ is said to be \textit{dual} to an edge $e$ of $I^n$ if $e$ is not parallel to $\sigma$.
For a cube complex $Y$, each cube $c$ in $Y$ is the image of a combinatorial map $q:I^n\rightarrow Y$. 
For an edge $e$ of $I^n$ and the midcube $\sigma\subset I^n$ dual to $e$, $q(\sigma)$ is said to be a $\it{midcube}$ of $c$ which is \textit{dual} to $q(e)$.
Let $\square$ denote the equivalence relation on edges of $Y$ generated by $e\square f$ if and only if $e$ and $f$ are two opposite edges of some square in $Y$. A $\it{hyperplane}$ is the union of midcubes of $Y$ which are dual to edges in an equivalence class $[e]$ generated by $\square$.
The $\it{carrier}$ $\kappa(h)$ of a hyperplane $h$ is the union of all the cubes in $Y$ intersecting $h$. If $Y$ is a CAT(0) cube complex, then hyperplanes have the following properties.

\begin{theorem}[\cite{Sag95}]
Let $Y$ be a $\rm{CAT(0)}$ cube complex. Then,
\begin{enumerate}
\item every hyperplane is embedded,
\item every hyperplane separates $Y$ into precisely two components,
\item every hyperplane is itself a CAT(0) cube complex, and
\item every collection of pairwise intersecting hyperplanes intersects.
\end{enumerate}
\end{theorem}

\begin{remark}
The convexity of subcomplexes in a CAT(0) cube complex $Y$ can be stated with respect to hyperplanes. A subcomplex $P\subset Y$ is convex if and only if, for any two hyperplanes $h_1$, $h_2$ in $Y$ intersecting $P$ and intersecting each other, $h_1\cap h_2$ is contained in $P$.
\end{remark}

In the remaining part of this subsection, let $Y$ be a finite dimensional CAT(0) cube complex and let us investigate some kinds of convex subcomplexes of $Y$. Consider the real line $\mathbb{R}$ and the non-negative real line $\mathbb{R}_{\geq 0}$ as 1-dimensional cube complexes whose vertices correspond to integers.
A \textit{singular geodesic} (a \textit{singular ray}, resp.) is the image of an isometric embedding $\mathbb{R}\hookrightarrow Y$ ($\mathbb{R}_{\geq 0}\hookrightarrow Y$, resp.). 
The \textit{parallel set} $\mathbb{P}(\gamma)\subset Y$ of a singular geodesic $\gamma$ is the convex hull of the union of all the singular geodesics parallel to $\gamma$. 
By Lemma \ref{Projection} (or Theorem 2.14 in Chapter II.2 of \cite{BH}), $\mathbb{P}(\gamma)$ has a product structure; $\mathbb{P}(\gamma)$ is isometric to the product of $\mathbb{R}$ and a CAT(0) cube complex.

For a subcomplex $K\subset Y$, let $\mathcal{H}$ be the collection of hyperplanes in $Y$ intersecting $K$. Suppose that $\mathcal{H}$ is a disjoint union of non-empty subcollections $\mathcal{H}_1,\cdots,\mathcal{H}_n$ such that for $h\in\mathcal{H}_i$ and $h'\in\mathcal{H}_j$, if $i\neq j$, then $h\cap h'\cap K$ is non-empty, and otherwise, either $h=h'$ or $h\cap h'$ is empty. By Proposition 2.6 in \cite{CS11}, then, $K$ is convex and has a product structure, i.e. $K$ is the image of an isometric embedding $\bold{P}_1\times\cdots\times \bold{P}_n\hookrightarrow Y$ for CAT(0) cube complexes $\bold{P}_i$. 
If each $\bold{P}_i$ has dimension 1, then $K$ is said to be an $n$-dimensional \textit{product\ subcomplex}.
Let $P_i\subset Y$ be the image of $v_1\times\cdots\times v_{i-1}\times \bold{P}_i\times v_{i+1}\times\cdots\times v_n$ for some (indeed any) vertex $v_j\in \bold{P}_j$ under the isometric embedding. For convenience, $K$ is denoted by $P_1\times\cdots\times P_n$. Under this convention, the inclusion relation between product subcomplexes can be interpreted coordinate-wisely: If a product subcomplex $K=P_1\times\cdots\times P_n$ is contained in a product subcomplex $K'=P'_1\times\cdots\times P'_n$, then we can assume that $P_i \subset P'_i$ by permuting the coordinates as necessary.

\begin{remark}\label{EmbeddingImmersion}
If $Y$ is an $n$-dimensional NPC cube complex, then every immersion $\iota:\bold{P}_1\times\cdots\times \bold{P}_n\hookrightarrow Y$ for 1-dimensional cube complexes $\bold{P}_i$ is locally isometric. It can be seen by looking at the induced link map and its image.
For instance, when $n=2$, the link of any vertex in $\bold{P}_1\times\bold{P}_2$ is a bipartite complete graph and the link of any vertex in $Y$ has no triangles; the image of the induced link map must be a full subcomplex. If $Y$ is additionally simply-connected, then the immersion $\iota$ is an isometric embedding.
\end{remark}

\begin{lemma}\label{NS}
Suppose that $\overbar{K}_1$ and $\overbar{K}_2$ are $n$-dimensional product subcomplexes of an $n$-dimensional CAT(0) cube complex $Y$ each of which is the product of infinite trees without leaves. If $\overbar{K}_1\subset N_r(\overbar{K}_2)$, then $\overbar{K}_1\subset \overbar{K}_2$.
\end{lemma}
\begin{proof}
By the CAT(0) geometry and the assumption, the distance function $d_{\overbar{K}_2}:=d(\cdot,\overbar{K}_2)$ is convex and bounded on $\overbar{K}_1$.
Since any geodesic segments in $\overbar{K}_1$ are extendible, $d_{\overbar{K}_2}$ is constant on $\overbar{K}_1$; let $D=d_{\overbar{K}_2}(\overbar{K}_1)$.
Let $\pi_{\overbar{K}_2}:Y\rightarrow \overbar{K}_2$ be the nearest point projection of $Y$ onto $\overbar{K}_2$. By Lemma \ref{Projection}, the convex hull of $\overbar{K}_1$ and $\pi_{\overbar{K}_2}(\overbar{K}_1)$ is isometric to $\overbar{K}_1\times[0,D]$. Since $Y$ and $\overbar{K}_1$ have the same dimension, $D$ must be $0$ and therefore, we have $\overbar{K}_1\subset \overbar{K}_2$.
\end{proof}

The parallel set of a singular geodesic is an example of product subcomplexes, and so is a top dimensional geodesical flat. More precisely, if the dimension of $Y$ is $n$, then a geodesical $n$-flat $F\subset Y$ (briefly, \textit{flat}) is a subcomplex since $F$ has the geodesic extension property. 
There is an isometric embedding $\mathbb{R}^n\hookrightarrow Y$ whose image is $F$ such that the isometric embedding induces the product structure of $F$.
In dimension 2, if the coarse intersection of flats is quasi-isometric to a geodesic, then these flats are actually contained in the parallel set of a singular geodesic. 

\begin{lemma}\label{parallel}
Let $Y$ be a $\rm{CAT(0)}$ square complex and $F_1,\cdots,F_n$ flats in $Y$ ($n\geq 2$). Suppose that there exists a constant $r>0$ and a $(\lambda,\varepsilon)$-quasi-geodesic $\gamma\subset Y$ such that $N_r(F_i)\cap N_r(F_j)$ is quasi-isometric to $\gamma$ for any distinct elements $i$, $j$ in $\{1,\cdots,n\}$. Then there exist a singular geodesic $\gamma'\subset Y$ and a constant $C_1=C_1(\lambda,\varepsilon,r)>0$ such that $F_i$ is contained in the parallel set $\mathbb{P}(\gamma')$ and $d_H(\gamma',\gamma)<C_1$.
\end{lemma}
\begin{proof}
By Lemma \ref{Projection}, there are subcomplexes $V_i\subset F_i$ such that $V_i=\{y\in F_i\ |\ d(y,F_j)=d(F_1,F_2)\}$ for $\{i,j\}=\{1,2\}$ such that $V_i$ is quasi-isometric to $N_r(F_1)\cap N_r(F_2)$. 
By the assumption, $V_1$ is quasi-isometric to $\gamma$ so that $\mathbb{P}(V_1)$ contains $F_1$ and $F_2$.

From $F_1$ and $F_3$, we also obtain two parallel singular geodesics $V'_1\subset F_1$ and $V'_3\subset F_3$ which are quasi-isometric to $\gamma$ as in the previous paragraph. Since both $V_1$ and $V'_1$ are quasi-isometric to $\gamma$ and contained in $F_1$, they are parallel.
Therefore, $F_3$ is also contained in $\mathbb{P}(V_1)$ and by induction, we prove the lemma.
\end{proof}

%
% 2.2 Weakly special square complex
%
\subsection{Weakly Special Square Complexes}\label{2.2}
An NPC cube complex $Y$ is \textit{weakly\ special} if there is no hyperplane which \textit{self-osculates} or \textit{self-intersects}; the terms `self-osculate' and `self-intersect' were introduced by Huglund and Wise in \cite{HW08} and using these terms, the definition of `weakly special' was introduced by Huang in \cite{Hua(b)}. 
If $Y$ is compact, by Proposition 3.10 in \cite{HW08}, there exists a finite cover $Y'$ of $Y$ such that every hyperplane in $Y'$ is two-sided. \textbf{In this subsection, we assume that $Y$ is a compact weakly special square complex and every hyperplane in $Y$ is two-sided. The universal cover of $Y$ is denoted by $\overline{Y}$ and the universal covering map is denoted by $p_Y:\overline{Y}\rightarrow Y$.}

By definition, $Y$ has the following properties:
\begin{enumerate}
\item If two edges $e_1$ and $e_2$ are dual to the same hyperplane, then $e_1$ is embedded if and only if $e_2$ is embedded.
\item If two distinct edges $e_1$ and $e_2$ intersect at a vertex, then the hyperplanes dual to $e_1$ and $e_2$ are distinct.
\item 
Suppose that $I=(\overline{y}_1,\cdots,\overline{y}_n)$ and $I'=(\overline{y}'_1,\cdots,\overline{y}'_n)$ are sequences of consecutive vertices such that $p_Y(\overline{y}_1)=p_Y(\overline{y}'_1)$, and the images of hyperplanes dual to the edges $(\overline{y}_i,\overline{y}_{i+1})$ and $(\overline{y}'_i,\overline{y}'_{i+1})$ under $p_Y$ are the same. Then there is a unique deck transformation sending $\overline{y}_i$ to $\overline{y}'_i$ for all $i$.
\end{enumerate} 
Under these properties, we can define a subcomplex of $Y$ which has a kind of product structure by mimicking the way to define a product subcomplex of a CAT(0) square complex.

\begin{lemma}\label{ProjofPS}
Let $\overbar{K}\subset\overline{Y}$ be a 2-dimensional product subcomplex denoted by $\overbar{P}_1\times\overbar{P}_2$. Then $p_Y(\overbar{K})$ is equal to the image of a local isometry $\iota:\mathbf{P}_1\times\mathbf{P}_2\rightarrow Y$ where $\bold{P}_i$ is a graph isometric to $p_Y(\overbar{P}_i)$ for $i=1,2$. Moreover, if $\overbar{P}_i$ has no leaves, then $\bold{P}_i$ is a standard graph, a graph without leaves.
\end{lemma}
\begin{proof}
For an edge $e\subset\overline{Y}$, let $h_e$ be the hyperplane in $\overline{Y}$ dual to $e$ and $\kappa(h_e)$ the carrier of $h_e$.
Consider two parallel 1-dimensional subcomplexes $\alpha=\overbar{P}_1\times \overline{p}_2$ and $\alpha'=\overbar{P}_1\times \overline{p}_2'$ of $\overline{Y}$ for two distinct vertices $\overline{p}_2,\overline{p}_2'\in \overbar{P}_2$ and let $\pi_{\alpha'}:\alpha\rightarrow \alpha'$ be the restriction of the nearest point projection onto $\alpha'$.
First, we will show that $p_Y(\alpha)$ is isometric to $p_Y(\alpha')$ via the following two steps.

$\mathbf{Step\ 1.}$ For an edge $e\subset\alpha$, the edge $\pi_{\alpha'}(e)\subset\alpha'$ is dual to $h_e$. Since $p_Y(e)$ and $p_Y(\pi_{\alpha'}(e))$ are dual to the same hyperplane, by the property (1) of weakly special square complexes, $p_Y(e)$ is embedded if and only if $p_Y(\pi_{\alpha'}(e))$ is embedded.

$\mathbf{Step\ 2.}$ For an edge $e\subset\alpha$, let $e_1\subset\alpha$ be an edge such that $p_Y(e_1)=p_Y(e)$, i.e. there is a deck transformation $g\in\pi_1(Y)$ which sends $e$ to $e_1$. Let $\pi_{\alpha'}(e_1)\subset\alpha'$ be the edge dual to $h_{e_1}$. 
For a vertex $\overline{y}\in e$, the vertex $\pi_{\alpha'}(\overline{y})$ is in $\pi_{\alpha'}(e)$ and there exists a unique geodesic path $\gamma$ joining $\overline{y}$ to $\pi_{\alpha'}(\overline{y})$ such that $\gamma\subset\kappa(h_e)$.
Let $\overline{y}_1$ be the vertex in $e_1$ such that $p_Y(\overline{y})=p_Y(\overline{y}_1)$.
Consider a geodesic path $\gamma_1$ joining $\overline{y}_1$ to $\pi_{\alpha'}(\overline{y}_1)$.
Since $\gamma$ and $\gamma_1$ are parallel and $p_Y(\overline{y})=p_Y(\overline{y}_1)$, by the property (3) of weakly special cube complexes, $g$ is the exact deck transformation sending $\overline{y}$ and $\gamma$ to $\overline{y}_1$ and $\gamma_1$, respectively. In particular, $\gamma_1\subset\kappa(h_{e_1})$ so that $g.\pi_{\alpha'}(\overline{y})$ is in $\pi_{\alpha'}(e_1)$.
Since the images of $h_e$ and $h_{e_1}$ under $p_Y$ are the same, by the property (2) of weakly special cube complexes, we have $p_Y(\pi_{\alpha'}(e_1))=p_Y(e')$.

By the above two steps, we can construct a well-defined combinatorial map $\iota_1:p_Y(\alpha)\rightarrow p_Y(\alpha')$ which sends $p_Y(e)$ for each edge $e\subset\alpha$ to $p_Y(\pi_{\alpha'}(e))$. Indeed, $\iota_1$ is a bijection. 
Similarly, a bijection $\iota_2:p_Y(\beta)\rightarrow p_Y(\beta')$ is defined.
Let $\bold{P}_1$ and $\bold{P}_2$ be graphs isometric to $p_Y(\overbar{P}_1)$ and $p_Y(\overbar{P}_2)$, respectively. 
Then $\iota_1$ and $\iota_2$ induces a combinatorial map $\iota=\iota_1\times\iota_2:\bold{P}_1\times\bold{P}_2\rightarrow Y$ whose image is $p_Y(\overline{K})$ and the map $\iota$ is obviously an immersion.
By Remark \ref{EmbeddingImmersion}, the immersion $\iota:\bold{P}_1\times\bold{P}_2\rightarrow Y$ is the desired local isometry.
Since the action of $\pi_1(Y)$ on $\overline{Y}$ has no fixed points, if $\overbar{P}_i$ has no leaves, then $\bold{P}_i$ is a standard graph.
\end{proof}

The converse of the above lemma also holds.

\begin{lemma}\label{Converse}
Let $\iota:\bold{P}_1\times \bold{P}_2\rightarrow Y$ be a local isometry for graphs $\bold{P}_i$ ($i=1,2$), and let $K$ be the image of $\iota$. 
Then there exists a 2-dimensional product subcomplex $\overbar{K}$ of $\overline{Y}$ such that the image of $\overbar{K}$ under $p_Y$ is equal to $K$. 
\end{lemma}
\begin{proof}
For a vertex $\overline{y}\in p_Y^{-1}(K)$, let $p_Y(\overline{y})$ be the image of a vertex $\bold{p}_1\times \bold{p}_2\in\bf{P}_1\times\bf{P}_2$ under $\iota$, and let $P_1=\iota({\bf{P}_1 } \times \bold{p}_2)$ and $P_2=\iota(\bold{p}_1\times\bf{P}_2)$.
Let $\overline{\bold{P}_i}$ be the universal cover of $\bold{P}_i$ for $i=1,2$. Then there is an elevation $\overline{\iota}:\overline{\bold{P}_1}\times\overline{\bold{P}_2}\rightarrow \overline{Y}$ such that the image of a vertex $\overline{\bold{p}}_1\times\overline{\bold{p}}_2\in\overline{\bold{P}_1}\times\overline{\bold{P}_2}$ is $\overline{y}$.
Consider the component $\overbar{P}_i$ of the preimage of $P_i$ under $p_Y$ which contains $\overline{y}$ for $i=1,2$. 
Then $\overline{\iota}(\overline{\bold{P}_1}\times \overline{p}_2)$ is contained in $\overbar{P}_1$. 
Suppose that $e_1$ is an edge of $\overbar{P}_1$ which is not contained in, but meets $\overline{\iota}(\overline{\bold{P}_1}\times \overline{p}_2)$.
Since $p_Y(e_1)$ is contained in $P_1$, by the property (3) of weakly special cube complex, there is an edge $e_1'$ in $\overline{\iota}(\overline{\bold{P}_1}\times \overline{p}_2)$ and a deck transformation $\sigma$ such that $\sigma$ sends $e_1$ to $e'_1$. 
Let $e_2\subset\overline{\iota}(\overline{p}_1\times \bold{P}_2)$ be an edge one of whose endpoints contains $\overline{y}$.
The existence of $\sigma$ implies that $(\overline{\iota}(\overline{\bold{P}_1}\times\overline{p}_2)\cup e_1)\times e_2$ is a product subcomplex.
Inductively, we can deduce that $\overbar{P}_1\times\overbar{P}_2$ is a product subcomplex whose image under $p_Y$ is $K$.
Lastly, there is no larger product subcomplex in $p_Y^{-1}(K)$ which properly contains $\overbar{P}_1\times \overbar{P}_2$; otherwise, the image under $p_Y$ is larger than $K$.
\end{proof}

For a subcomplex $K\subset Y$, a product subcomplex $\overbar{K}$ of $\overline{Y}$ is said to be a \textit{p-lift} of $K$ if (1) $p_Y(\overbar{K})=K$ and (2) there is no other product subcomplex which properly contains $\overbar{K}$. Lemma \ref{Converse} implies that if $K$ is the image of the product of two graphs under a local isometry, then it has a p-lift.
If $\overbar{K}_1$ and $\overbar{K}_2$ are two p-lifts of $K$, then there must be a deck transformation sending $\overbar{K}_1$ to $\overbar{K}_2$.

\begin{definition}[Standard Product Subcomplex]\label{SPS}
A \textit{standard product subcomplex} $K$ of $Y$ is defined as the image of a local isometry from the product of two standard graphs into $Y$, and a \textit{standard product subcomplex} $\overbar{K}$ of $\overline{Y}$ is defined as a p-lift of $K$.
Among standard product subcomplexes of either $Y$ or $\overline{Y}$, a maximal one under the set inclusion is said to be a \textit{maximal product subcomplex}.
\end{definition}

\begin{remark}
Since the image of a standard product subcomplex of $\overline{Y}$ under $p_Y$ is a standard product subcomplex of $Y$, if the intersection of standard product subcomplexes $\overbar{K}_i\subset\overline{Y}$ contains a standard product subcomplex $\overbar{K}$, then the intersection of $p_Y(\overbar{K}_i)$'s contains the standard product subcomplex $p_Y(\overbar{K})$.
Conversely, if the intersection of standard product subcomplexes $K_i\subset Y$ contains a standard product subcomplex $K$, then there are p-lifts $\overbar{K}_i$ of $K_i$, which are standard product subcomplexes of $\overline{Y}$, such that the intersection of $\overbar{K}_i$'s contains a p-lift of $K$.  
\end{remark}

Let $K$ be a standard product subcomplex of $Y$. For a p-lift $\overbar{K}$ of $K$, the product $\bold{P}_1\times\bold{P}_2$ of two standard graphs $\bold{P}_1$ and $\bold{P}_2$ obtained as in Lemma \ref{ProjofPS} has the following two properties: 
\begin{itemize}
\item
Consider the collection $\{G_i\}$ of subgroups of  the stabilzer of $\overbar{K}$ under the canonical action $\pi_1(Y)\curvearrowright \overline{Y}$ where the quotient of $\overbar{K}$ by $G_i$ is the product of two graphs. Then the image of the induced homomorphism $\iota_*:\pi_1(\bold{P}_1\times\bold{P}_2)\rightarrow \pi_1(Y)$ is the largest one.
\item
Since p-lifts of $K$ is unique up to deck transformation, the product of two standard graphs obtained from another p-lift of $K$ is isometric to $\bold{P}_1\times\bold{P}_2$. 
\end{itemize}
We say that such $\bold{P}_1\times\bold{P}_2$ is the \textit{base} of $K$ (or $\overbar{K}$).

Let $K\subset Y$ be a standard product subcomplex with base $\bold{P}_1\times\bold{P}_2$ and local isometry $\bold{P}_1\times\bold{P}_2\rightarrow Y$. 
Let $P_1$ and $P_2$ be the images of $\bold{P}_1\times \bold{p}_{2}$ and $\bold{p}_1\times\bold{P}_2$ under the local isometry for some (indeed any) vertices $\bold{p}_1\in\bold{P}_1$ and $\bold{p}_2\in\bold{P}_2$, respectively.
For convenience, $K$ is then denoted by $P_1\times P_2$.
As we did for the inclusion relation between product subcomplexes of $\overline{Y}$, the inclusion relation between standard product subcomplexes of $Y$ can be interpreted coordinate-wisely.

\begin{lemma}\label{InclusionBetweenSPSes}
Suppose that $K$ and $K'$ are standard product subcomplexes of $Y$ such that $K\subset K'$. Then $K$ and $K'$ can be denoted by $P_1\times P_2$ and $P'_1\times P'_2$, respectively, such that $P_i\subset P'_i$ for $i=1,2$.
\end{lemma}
\begin{proof}
Let $\bold{P}'_1\times\bold{P}'_2$ be the base of $K'$ with the local isometry $\iota':\bold{P}'_1\times\bold{P}'_2\rightarrow Y$ and let $\overline{\iota'}:\overline{\bold{P}'_1}\times\overline{\bold{P}'_2}\hookrightarrow \overline{Y}$ be an elevation whose image is a standard product subcomplex $\overbar{K}'\subset\overline{Y}$.
Since $K\subset K'$, there exists a p-lift $\overbar{K}$ of $K$ contained in $\overbar{K}'$. It means that $\overbar{K}$ and $\overbar{K}'$ can be denoted by $\overbar{P}_1\times\overbar{P}_2$ and $\overbar{P}'_1\times\overbar{P}'_2$, respectively, where $\overbar{P}_i\subset\overbar{P}'_i$ for $i=1,2$. 
By Lemma \ref{ProjofPS} and the notation of standard product subcomplex of $Y$, $K$ and $K'$ are therefore denoted by $p_Y(\overbar{P}_1)\times p_Y(\overbar{P}_2)$ and $p_Y(\overbar{P}'_1)\times p_Y(\overbar{P}'_2)$, respectively, where $p_Y(\overbar{P}_i)\subset p_Y(\overbar{P}'_i)$. 
\end{proof}

\begin{convention} 
The following letters are usually used to denote objects related to product subcomplexes for a compact weakly special square complex $Y$ and the universal cover $\overline{Y}$ of $Y$:
\begin{itemize}
\item $P$ denotes a 1-dimensional subcomplex of $Y$ and $\overline{P}$ denotes a 1-dimensional subcomplex of $\overline{Y}$ which is isometric to the universal cover of $P$. 
\item $\overbar{K}$ and $\overbar{M}$ ($K$ and $M$, resp.) may denote standard and maximal product subcomplexes of $\overline{Y}$ ($Y$, resp.). A flat in $\overline{Y}$ may be denoted by $F$.
\item The base of a standard product subcomplex (or a subcomplex of the base) of $Y$ is denoted by a bold letter like $\bold{P}, \bold{K}, \mathrm{or}\ \bold{M}$.
\end{itemize}
In particular, the line over a letter denotes either the universal cover of the space which the letter refers or a subcomplex of the universal cover of something.
\end{convention}

For a standard product subcomplex $\overbar{K}\subset\overline{Y}$, the base of $p_Y(\overbar{K})\subset Y$ has the fundamental group quasi-isometric to exactly one of $\mathbb{Z}\times\mathbb{Z}$, $\mathbb{Z}\times\mathbb{F}$ or $\mathbb{F}\times\mathbb{F}$.
We say that $p_Y(\overbar{K})$ has one of the following \textit{quasi-isometric\ types}: $(\mathbb{Z}\times\mathbb{Z})$, $(\mathbb{Z}\times\mathbb{F})$ or $(\mathbb{F}\times\mathbb{F})$. 
The quasi-isometric type of $\overbar{K}$ is defined as the quasi-isometric type of $p_Y(\overbar{K})$. Note that if standard product subcomplexes of $\overline{Y}$ have different quasi-isometric types, then they are not quasi-isometric to each other. 

\vspace{0.5mm}

There are a series of facts about standard product subcomplexes of $\overline{Y}$ which we will use frequently in this paper.

\begin{lemma}\label{SubSPC}
Let $\overbar{K}\subset\overline{Y}$ be a product subcomplex $\overbar{P}_1\times\overbar{P}_2$ where $\overbar{P}_i$ is a subcomplex of $\overline{Y}$ which is isometric to an infinite tree for $i=1,2$.
Then there is a standard product subcomplex $\overbar{K}'\subset\overline{Y}$ such that $\overbar{K}'$ is contained in any standard product subcomplex whose neighborhood contains $\overbar{K}$.
\end{lemma}
\begin{proof}
By Lemma \ref{ProjofPS}, $p_Y(\overbar{K})\subset Y$ is the image of a local isometry $\iota:\bold{P}_1\times \bold{P}_2\rightarrow Y$ where $\bold{P}_1$ and $\bold{P}_2$ are graphs isometric to $p_Y(\overbar{P}_1)$ and $p_Y(\overbar{P}_2)$, respectively. 
Since $p_Y$ is a locally isometry and $\overbar{P}_i$ is infinite, $\bold{P}_i$ is non-simply connected for $i=1,2$.
Let $\bold{P}_3$ and $\bold{P}_4$ be the subgraphs of $\bold{P}_1$ and $\bold{P}_2$ obtained by removing leaves, respectively, i.e. $\bold{P}_3$ and $\bold{P}_4$ are the maximal standard subgraphs of $\bold{P}_1$ and $\bold{P}_2$, respectively.
Consider the preimage $\overbar{P}_{i+2}$ of $\iota(\bold{P}_{i+2})$ under the restriction of $p_Y$ to $\overbar{P}_i$ for $i=1,2$. Then $\overbar{P}_{i+2}$ is obviously connected and there exists a constant $C>0$ which depends on $Y$ such that $\overbar{P}_i$ lies in the $C$-neighborhood of $\overbar{P}_{i+2}$.
Then $\overbar{P}_3\times\overbar{P}_4\subset\overline{Y}$ is a product subcomplex which is contained in $\overbar{K}$ and whose neighborhood contains $\overbar{K}$.

The restriction $\iota|_{\bold{P}_3\times\bold{P}_4}$ of $\iota$ to $\bold{P}_3\times\bold{P}_4$ is a local isometry so that the image $K'$ of $\iota|_{\bold{P}_3\times\bold{P}_4}$ is a standard product subcomplex. Then there is a p-lift $\overbar{K}'$ of $K'$ in $\overline{Y}$ such that $\overbar{K}'$ contains $\overbar{P}_3\times\overbar{P}_4$.
Therefore, $\overbar{K}'$ is a standard product subcomplex of $\overline{Y}$ whose neighborhood contains $\overbar{K}$.
Suppose that $\overbar{K}''$ is a standard product subcomplex whose neighborhood contains $\overbar{K}$.
By Lemma \ref{NS}, $\overbar{K}''$ contains $\overbar{K}$ and thus $p_Y(\overbar{K}'')$ contains $K'$.
Since $\overbar{K}''$ is a p-lift of $p_Y(\overbar{K}'')$, it must contain $\overbar{K}'$.
\end{proof}

\begin{lemma}\label{IntofStd}
Let $\overbar{K}_1$ and $\overbar{K}_2$ be standard product subcomplexes of $\overline{Y}$. 
If $\overbar{K}_1\cap \overbar{K}_2$ contains a flat, then there exists a constant $C_1>0$ depending on $Y$ and a standard product subcomplex $\overbar{K}_3\subset \overbar{K}_1\cap \overbar{K}_2$ such that $\overbar{K}_1\cap \overbar{K}_2$ is contained in the $C_1$-neighborhood of $\overbar{K}_3$. In particular, $d_H(\overbar{K}_3,\overbar{K}_1\cap\overbar{K}_2)$ is finite.
This statement holds for finitely many standard product subcomplexes containing a flat in common. 
\end{lemma}
\begin{proof}
Let $\mathcal{H}$ be the collection of hyperplanes in $\overline{Y}$ intersecting $\overbar{K}_1\cap \overbar{K}_2$.
Since $\overbar{K}_1$ is a product subcomplex and $\overbar{K}_1\cap \overbar{K}_2\subset \overbar{K}_1$, $\mathcal{H}$ is the disjoint union of $\mathcal{H}_1$ and $\mathcal{H}_2$ such that every hyperplane $h_1\in\mathcal{H}_1$ intersects every hyperplane $h_2\in\mathcal{H}_2$. Since $\overbar{K}_1\cap \overbar{K}_2$ is convex, by Proposition 2.6 in \cite{CS11}, $\overbar{K}_1\cap \overbar{K}_2$ is a product subcomplex.
Let $\overbar{K}_3\subset\overline{Y}$ be the standard product subcomplex obtained from $\overbar{K}_1\cap \overbar{K}_2$ as in Lemma \ref{SubSPC}.
By construction, $\overbar{K}_3$ is then contained in both $\overbar{K}_1$ and $\overbar{K}_2$.
Therefore, the lemma holds.
\end{proof}

\begin{lemma}\label{AtMostOneMax}
Let $\gamma$ be a singular geodesic in $\overline{Y}$.
Then there is at most one maximal product subcomplex of $\overline{Y}$ whose neighborhood contains the parallel set $\mathbb{P}(\gamma)$.
\end{lemma}
\begin{proof}
Suppose that there is a maximal product subcomplex of $\overline{Y}$ which contains $\gamma$. Then the parallel set $\mathbb{P}(\gamma)$ is the product $\gamma\times\overbar{T}$ of $\gamma$ and an infinite tree $\overbar{T}$.
By Lemma \ref{SubSPC}, then, there is a standard product subcomplex $\overbar{K}\subset\overline{Y}$ such that $\overbar{K}=\overbar{P}_1\times\overbar{P}_2$ is contained in any standard product subcomplex whose neighborhood contains $\mathbb{P}(\gamma)$.

Suppose that there are two maximal product subcomplexes $\overbar{M}'$ and $\overbar{M}''$ whose neighborhoods contain $\mathbb{P}(\gamma)$. 
By Lemma \ref{NS} and the construction of $\overbar{K}$, then, $\overbar{M}'$ and $\overbar{M}''$ can be denoted by $\overbar{P}'_1\times\overbar{P}'_2$ and $\overbar{P}''_1\times\overbar{P}''_2$, respectively, such that $\overbar{P}_1$ is contained in $\overbar{P}'_1\cup \overbar{P}''_1$ and $\overbar{P}_2$ is equal to $\overbar{P}'_2$ and $\overbar{P}''_2$. 
However, it is a contradiction since $\overbar{P}'_1$ and $\overbar{P}''_1$ are not identical so that $(\overbar{P}'_1\cup\overbar{P}''_1)\times\overbar{P}_2$ is a standard product subcomplex which properly contains $\overbar{M}'$ and $\overbar{M}''$.
Therefore, the lemma holds.
\end{proof}

Before closing this subsection, we introduce a rigidity result about flats in $\overline{Y}$, the result in \cite{Hua(b)}, which is a key tool in this paper (especially in Section \ref{3.1}). 

\begin{theorem}[Theorem 1.3 in \cite{Hua(b)}]\label{FlattoFlat}
Let $Y_1$ and $Y_2$ be two compact weakly special square complexes with their universal covers $\overline{Y_1}$ and $\overline{Y_2}$, respectively.
If there is a $(\lambda,\varepsilon)$-quasi-isometry $\phi:\overline{Y_1}\rightarrow \overline{Y_2}$, then there exists a constant $C=C(\lambda,\varepsilon)>0$ such that for any flat $F\subset \overline{Y_1}$, there exists a unique flat $\phi_H(F)\subset \overline{Y_2}$ such that $d_H(\phi(F),\phi_H(F))<C$ where $d_H$ denotes the Hausdorff distance.
\end{theorem}

%
% 2.3 Right-angled Artin groups
%
\subsection{Right-angled Artin Groups}\label{2.3}
Let $\Lambda$ be a (possibly disconnected) simplicial graph. 
The Salvetti complex $S(\Lambda)$ associated to $\Lambda$ is a cube complex with one vertex such that each $n$-cube is an $n$-dimensional torus and corresponds to a complete subgraph $K_n$ of $\Lambda$. From the construction, it can be easily shown that $S(\Lambda)$ is a (weakly) special cube complex. Note that (1) if $\Lambda$ is the disjoint union of $\Lambda_1,\cdots,\Lambda_n$, then $S(\Lambda)$ is the wedge sum of $S(\Lambda_i)$'s, and (2) $S(\Lambda)$ is 2-dimensional if and only if $\Lambda$ is triangle-free and contains at least one edge.
The right-angled Artin group (RAAG) $A(\Lambda)$ with defining graph $\Lambda$ is the fundamental group of $S(\Lambda)$; the dimension of $A(\Lambda)$ is defined as the dimension of $S(\Lambda)$.
The universal cover of $S(\Lambda)$ and the universal covering map are denoted by $X(\Lambda)$ and $p_{S(\Lambda)}:X(\Lambda) \rightarrow S(\Lambda)$, respectively. Then vertices in $X(\Lambda)$ correspond to elements of $A(\Lambda)$ (after choosing a vertex in $X(\Lambda)$ corresponding to the identity of $A(\Lambda)$) due to the fact that $S(\Lambda)$ has a unique vertex. 
We refer to \cite{CH}, \cite{CHD}, \cite{HW08} for properties of Salvetti complexes and RAAGs. 

\begin{remark}[\cite{HW08}]\label{SC}
An alternative definition of a compact special cube complex is the following: a compact NPC cube complex $Y$ is special if there is a local isometry from $Y$ into a Salvetti complex.
\end{remark}

Each (possibly disconnected) full subgraph $\Lambda_1\leq \Lambda$ gives rise to a subgroup $A(\Lambda_1)$ of $A(\Lambda)$ and a locally isometric embedding $S(\Lambda_1)\hookrightarrow S(\Lambda)$; $S(\Lambda_1)$ is considered as a locally convex subcomplex of $S(\Lambda)$.
This kind of subcomplex $S(\Lambda_1)\subset S(\Lambda)$ is called a \textit{standard subcomplex of $S(\Lambda)$ with defining graph $\Lambda_1$}.
In particular, each edge of $S(\Lambda)$ is a standard subcomplex whose defining graph is a vertex; the collection of vertices in $\Lambda$ can be considered as the standard generating set of $A(\Lambda)$.
Each p-lift of $S(\Lambda_1)$ in $X(\Lambda)$ is a convex subcomplex (isometric to $X(\Lambda_1)$) which is called a \textit{standard subcomplex of $X(\Lambda)$ with defining graph $\Lambda_1$}. 
Especially, a standard subcomplex of $X(\Lambda)$ whose defining graph is a vertex $v\in \Lambda$ is called a \textit{standard geodesic labelled by} $v$ (which is isometric to $\mathbb{R}$). We remark that each edge of $S(\Lambda)$ or $X(\Lambda)$ is labelled by a vertex in $\Lambda$.

Suppose that $\Lambda$ is connected and triangle-free. Then a standard product subcomplex of $S(\Lambda)$ or $X(\Lambda)$ is a standard subcomplex whose defining graph $\Lambda_1\leq\Lambda$ admits a join decomposition $\Lambda_2 \circ \Lambda_3$ where $\Lambda_2$ and $\Lambda_3$ are non-empty. 
Since $\Lambda_2$ and $\Lambda_3$ have no edges, $A(\Lambda_1)$ is isomorphic to $\mathbb{F}_n\times\mathbb{F}_m$ where $n$ and $m$ are the number of vertices of $\Lambda_2$ and $\Lambda_3$, respectively.  
Since the groups isomorphic to $\mathbb{F}_n \times \mathbb{F}_m$ ($n,m\geq 1$) are frequently used in this paper, for convenience, they will be called \textit{join groups}.
The defining graph of a maximal product subcomplex of $S(\Lambda)$ or $X(\Lambda)$ is called a \textit{maximal join subgraph} of $\Lambda$. 

\vspace{0.5mm}

\textbf{Whenever we talk about RAAGs or Salvetti complexes, their defining graphs are always assumed to be simplicial even if we do not say.}

%
% 2.4 GBG
%

\subsection{Graph Braid Groups}\label{2.4}
We start from a (possibly non-simplicial) connected graph $\Gamma$. The ordered and unordered configuration spaces of $n$ points on $\Gamma$ are respectively
$$C_n(\Gamma) = \{(x_1 ,x_2, \cdots, x_n)\in \Gamma^n\ |\ x_i \neq x_j\ \mathrm{if}\ i \neq j \}\ \ \mathrm{and}$$  
$$UC_n(\Gamma) = \{\{x_1 ,x_2, \cdots, x_n\}\subset \Gamma\ |\ x_i \neq x_j\ \mathrm{if}\ i \neq j \}.$$
Under the action of the symmetric group $S_n$ on $\Gamma^n$ by permuting $n$ coordinates, $UC_n(\Gamma)$ is the quotient space of $C_n(\Gamma)$ by $S_n$.
The graph $n$-braid group $B_n(\Gamma)$ and the pure graph $n$-braid group $PB_n(\Gamma)$ are the fundamental groups $\pi_1(UC_n(\Gamma))$ and $\pi_1(C_n(\Gamma))$ of $UC_n(\Gamma)$ and $C_n(\Gamma)$, respectively. We note that $B_n(\Gamma)$ (and thus $PB_n(\Gamma)$) is always non-trivial except when $\Gamma$ is homeomorphic to a line segment.

Unlike configuration spaces on higher dimensional spaces, there is a discrete version of configuration space on $\Gamma$.
The ordered and unordered discrete configuration spaces of $n$ points on $\Gamma$ are 
$$D_n(\Gamma) = \{ (\sigma_1, \cdots, \sigma_n)\subset \Gamma^n\ |\ {\sigma_i} \cap {\sigma_j} = \emptyset\ \mathrm{if}\ i \neq j\}\ \ \mathrm{and}$$ 
$$UD_n(\Gamma)=\{ \{\sigma_1, \cdots, \sigma_n\} \subset \Gamma\ |\ \ \sigma_i \cap \sigma_j = \emptyset\ \mathrm{if}\ i \neq j \},$$ respectively, where $\sigma_i$ is either a 0-cell or a 1-cell. 
Abrams \cite{Abrams00} showed that $D_n(\Gamma)$ and $UD_n(\Gamma)$ are NPC cube complexes, and Crisp and Wiest \cite{CW} showed that there are local isometries from $D_n(\Gamma)$ and $UD_n(\Gamma)$ into Salvetti complexes; by Remark \ref{SC}, $D_n(\Gamma)$ and $UD_n(\Gamma)$ are special cube complexes.

\begin{lemma}\label{EmbeddingInGBG}
Let $\Gamma$ be a graph and $\Gamma'$ a subgraph of $\Gamma$. Then the embedding $\Gamma'\hookrightarrow \Gamma$ induces a locally isometric embedding $D_n(\Gamma')\hookrightarrow D_n(\Gamma)$.
\end{lemma}
\begin{proof}
Obviously, the embedding $\Gamma'\hookrightarrow \Gamma$ induces the embedding $D_n(\Gamma')\hookrightarrow D_n(\Gamma)$. It can be easily seen that the induced link map is injective and its image is a full subcomplex so that and this embedding is a local isometry.
\end{proof}

However, $D_n(\Gamma)$ and $UD_n(\Gamma)$ are not topological invaraints of $\Gamma$. More precisely, if vertices are added to edges of $\Gamma$ (the number of vertices of valency 2 is increased), then the homotopy type of the (un)ordered discrete configuration space may be changed.
According to \cite{Abrams00}, \cite{KKP12} and \cite{PS12}, this kind of problem can be settled if there exists the assumption that $\Gamma$ is suitably subdivided.

\begin{theorem} [\cite{Abrams00}, \cite{KKP12}, \cite{PS12}]\label{Abrams}
For any $n>1$ and any graph $\Gamma$ with at least $n$ vertices, $D_n(\Gamma)$ is a deformation retract of $C_n(\Gamma)$ if and only if $\Gamma$ satisfies the following two conditions.
\begin{enumerate}
\item Each path between two vertices of valency not equal to 2 passes through at least $n-1$ edges.
\item Each homotopically non-trivial loop passes through at least $n+1$ edges.
\end{enumerate}
In particular, $C_2(\Gamma)$ and $D_2(\Gamma)$ are homotopy equivalent if $\Gamma$ is simplicial.
\end{theorem}

By the above theorem, we can deduce that $PB_n(\Gamma)$ and $B_n(\Gamma)$ are cocompactly cubulated groups.

\begin{Ex}\label{Tripod}
Let $T_3$ be a star tree with central vertex $v$ and three leaves $a,b,c$ given on the left of Figure \ref{tripodFig}. Since $D_2(T_3)$ is homeomorphic to a circle and $T_3$ is simplicial, $PB_2(T_3)$ is isomorphic to $\mathbb{Z}$.
The movement of two points on $T_3$ corresponding to a generator of $PB_2(T_3)$ is called an \textit{Y-exchange}, denoted by $(a,b,c;v)$.
\end{Ex}

\begin{figure}[h]
\centering
\begin{tikzpicture}
\tikzstyle{every node}=[draw,circle,fill=black,minimum size=3pt,inner sep=0pt]
  \node[anchor=center] (n1) at (0.2,1.7) [label={[label distance=0cm]above:$a$}]{};
  \node[anchor=center] (n2) at (-1,-1) [label={left:$b$}]{};
  \node[anchor=center] (n3) at (1.4,-1)  [label={right:$c$}]{};
  \node[anchor=center] (n4) at (0.2,0) [label={[label distance=0.05cm]left:$v$}]{};
  \foreach \from/\to in {n1/n4,n4/n2,n4/n3}
  \draw (\from) -- (\to);
  
  \draw (4.6,2) node (n5) [label={[label distance=-0.3cm]4:$(b,c)$}]{};
  \draw (5.8,2) node (n6) [label={[label distance=-0.3cm]4:$(v,c)$}]{};
  \draw (7,2) node (n7) [label={[label distance=-0.3cm]4:$(a,c)$}]{};
  \node[anchor=center] (n8) at (8,1.3){};
  \node[anchor=center] (n9) at (8,0.4){};
  \node[anchor=center] (n10) at (8,-0.5){};
  \draw (7,-1.2) node (n11) [label={[label distance=-0.2cm]below:$(c,b)$}]{};
  \draw (5.8,-1.2) node (n12) [label={[label distance=-0.2cm]below:$(c,v)$}]{};
  \draw (4.6,-1.2) node (n13) [label={[label distance=-0.2cm]below:$(c,a)$}]{};
  \node[anchor=center] (n14) at (3.6,-0.5){};
  \node[anchor=center] (n15) at (3.6,0.4){};
  \node[anchor=center] (n16) at (3.6,1.3){};
  \foreach \from/\to in {n5/n6,n6/n7,n7/n8,n8/n9,n9/n10,n10/n11,n11/n12,n12/n13,n13/n14,n14/n15,n15/n16,n16/n5}
  \draw (\from) -- (\to);
\end{tikzpicture}
\caption{$T_3$ and $D_2(T_3)$}
	\label{tripodFig}
\end{figure}
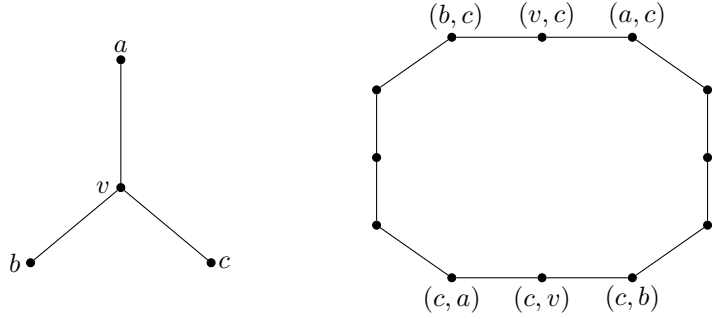

Let $\Gamma$ be a simplicial graph so that $C_2(\Gamma)$ is homotopy equivalent to $D_2(\Gamma)$.
Let $K\subset D_2(\Gamma)$ be a standard product subcomplex with base $\mathbf{K}=\mathbf{P}_1\times\mathbf{P}_2$ and local isometry $\iota:\mathbf{K}\rightarrow K\subset D_2(\Gamma)$.
Since $D_2(\Gamma)$ is a subcomplex of $\Gamma\times\Gamma$ and the image of each coordinate of $\mathbf{K}$ is injective, $\iota$ must be an embedding.
From the fact that there is a square $\sigma_1\times\sigma_2$ in $D_2(\Gamma)$ if and only if $\sigma_1$ and $\sigma_2$ are disjoint edges of $\Gamma$, we can easily derive the following fact.

\begin{lemma}\label{DisjointSubgraphs}
Let $\mathrm{Pr}_i :\Gamma\times\Gamma\rightarrow \Gamma$ be the $i$-th projection map of $\Gamma\times\Gamma$ for $i=1,2$.
Let $K\subset{D_2(\Gamma)}$ be a standard product subcomplex. Then the images of $K$ under $\mathrm{Pr}_1$ and $\mathrm{Pr}_2$ are disjoint standard subgraphs of $\Gamma$.
\end{lemma}

By the same idea, a product subcomplex of $D_2(\Gamma)$ can be defined canonically. 

\begin{definition}[Product subcomplex of $D_2(\Gamma)$]\label{PSofD2}
A subcomplex $P\subset D_2(\Gamma)$ is a \textit{product subcomplex} if $P$ is equal to the intersection of $\mathrm{Pr}^{-1}_1\circ\mathrm{Pr}_1(P)$ and $\mathrm{Pr}^{-1}_2\circ\mathrm{Pr}_2(P)$. Equivalently, a product subcomplex $P$ is the image of a locally isometric embedding $A\times B\rightarrow D_2(\Gamma)$ for two graphs $A$ and $B$ (each of which are possibly a single vertex).
\end{definition}

A (standard) product subcomplex of $D_2(\Gamma)$ is denoted by $\Gamma_i \times \Gamma'_i$ where $\Gamma_i$ and $\Gamma'_i$ are disjoint (standard) subgraphs of $\Gamma$.  
The intersection of two product subcomplexes $\Gamma_1 \times \Gamma'_1$ and $\Gamma_2 \times \Gamma'_2$ can be denoted by $(\Gamma_1\cap \Gamma_2)\times (\Gamma'_1\cap \Gamma'_2)$. Note that $(\Gamma_1\cap \Gamma_2)$ and $(\Gamma'_1\cap \Gamma'_2)$ are possibly disconnected but they must be disjoint in $\Gamma$.
Each p-lift of a standard product subcomplex $\Gamma_1 \times \Gamma_2\subset{D_2(\Gamma)}$ in $\overline{D_2(\Gamma)}$ is isometric to $\overline{\Gamma_1}\times\overline{\Gamma_2}$ where $\overline{\Gamma_i}$ is the universal cover of $\Gamma_i$ for $i=1,2$.

\vspace{1mm}

Let $\Gamma$ be a cactus without leaves. Suppose that there is a separaing vertex $v$ in $\Gamma$ and there is a maximal product subcomplex $\Gamma_i\times\Gamma'_i$ of $D_2(\Gamma)$. Then the fact that $\Gamma$ is a cactus induces that exactly one of $\Gamma_i$ or $\Gamma'_i$, say $\Gamma_i$, contains all the closures of components of the complement of $v$ in $\Gamma$ except one, which properly contains $\Gamma'_i$ (note that $\Gamma'_i$ must not contain $v$).
If $a$ is a boundary cycle contained in one of $\Gamma_i$ or $\Gamma'_i$, then the previous statement still holds when we replace $v$ by $a$. If $a$ is a boundary cycle which is not contained in both $\Gamma_i$ and $\Gamma'_i$, then each of $a\cap\Gamma_i$ and $a\cap\Gamma'_i$ is a path (possibly, a vertex), and the union of $\Gamma_i$, $\Gamma'_i$ and $a$ is the whole graph $\Gamma$. 
From this fact, we deduce the following fact.
%Thus, there is a subgraph $T_i\subset\Gamma$, which is either a (boundary) cycle or an edge, such that (1) each of $T_i\cap\Gamma_i$ and $T_i\cap\Gamma'_i$ is simply connected, and (2) the union of $\Gamma_i$, $\Gamma'_i$ and $T_i$ is the whole graph $\Gamma$. Moreover, each of $T_i\cap\Gamma_i$ and $T_i\cap\Gamma'_i$ separates $\Gamma$ into at least two components.

\begin{lemma}\label{IntofMax}
For a cactus $\Gamma$ without leaves, suppose that there are two maximal product subcomplexes $M_1=\Gamma_1\times\Gamma'_1$ and $M_2=\Gamma_2\times\Gamma'_2$ of $D_2(\Gamma)$. If $M_1 \cap M_2$ is non-empty, then it is connected and contains a standard product subcomplex which contains any standard product subcomplex contained in $M_1\cap M_2$.
\end{lemma}
\begin{proof}
Since $\Gamma_1\cap\Gamma_2$ is non-empty by the assumption, it suffices to show that $\Gamma_1\cap\Gamma_2$ is a connected subgraph of $\Gamma$ containing a standard subgraph (by the same way, we can show that $\Gamma'_1\cap\Gamma'_2$ also contains a boundary cycle and thus the lemma holds). 

There are two cases depending on the existence of a boundary cycle not contained in both $\Gamma_1$ and $\Gamma'_1$:

$\mathbf{Case\ 1.}$
Suppose that there is no boundary cycle which is not contained in both $\Gamma_1$ and $\Gamma'_1$. Then there exist a unique edge $e$ such that $e$ is contained in neither $\Gamma_1$ nor $\Gamma'_1$, and $e\cap\Gamma_1$ is a separating vertex $v$ in $\Gamma$. Let $A_0,\cdots,A_n$ be the closures of components of the complement of $v$, where $A_0$ is the one containing $e$.
The paragraph above this lemma implies that $\Gamma_1$ is equal to $A_1\cup\cdots\cup A_n$.
If $\Gamma_2$ contains $v$, then $\Gamma_2$ either contains $A_1,\cdots,A_n$ so that $\Gamma_1\cap\Gamma_2=\Gamma_1$ or is contained in $A_i$ for $i>0$ so that $\Gamma_1\cap\Gamma_2=\Gamma_2$; since $\Gamma'_1$ is contained in $A_0$, if $\Gamma_2$ contains $A_0$, then $\Gamma'_1\cap\Gamma'_2$ is empty.
Otherwise, either $\Gamma_2\subset\Gamma_1$ and in particular, $\Gamma_1\cap\Gamma_2=\Gamma_2$, or $\Gamma_2\subseteq\Gamma'_1$ but the latter cannot happen since $\Gamma_1\cap\Gamma_2$ is not empty.

$\mathbf{Case\ 2.}$
Suppose that there exists a boundary cycle $a$ which is not contained in both $\Gamma_1$ and $\Gamma'_1$.
 
Suppose that $\Gamma_2$ contains $a$. If $\Gamma_1$ is not contained in $\Gamma_2$, by the paragraph above this lemma, $\Gamma'_2$ must be contained in $\Gamma_1$ but it is a contradiction since $\Gamma'_1\cap\Gamma'_2$ becomes empty. Thus, $\Gamma_1\subset\Gamma_2$.

Suppose that $\Gamma_2$ does not contain $a$ but $\Gamma'_2$ does. The paragraph above this lemma implies that $\Gamma_2$ is contained in the closure of one component of the complement of $a$. Then $\Gamma_2$ must be contained in $\Gamma_1$ since $\Gamma_1\cap\Gamma_2$ is non-empty.

Finally, suppose that both $\Gamma_2$ and $\Gamma'_2$ do not contain $a$. 
Let $v_i$ be the vertices in $a$ each of which is the intersection of $a$ and the closure of a component of the complement of $a$.
Then $\Gamma_1\cap\Gamma_2$ is non-empty if and only if $(\Gamma_1\cap a)\cap(\Gamma_2\cap a)$ is non-empty. In particular, $\Gamma_1\cap\Gamma_2$ is the union of the path $\Gamma_1\cap\Gamma_2\cap a$ and all the closures $A_i$ of components of the complement of $a$ satisfying that each $A_i$ contains a vertex in $\Gamma_1\cap\Gamma_2\cap a$. Since each $A_i$ contains a standard subgraph, $\Gamma_1\cap\Gamma_2$ is connected and contains a standard subgraph.
\end{proof}

Let $\mathcal{C}$ be the collection of all boundary cycles in $\Gamma$.
If $\Gamma$ is a simplest cactus, then every standard product subcomplex of $D_2(\Gamma)$ can be represented by the product $\mathcal{C}_1\times\mathcal{C}_2$ of two disjoint subcollections $\mathcal{C}_1$, $\mathcal{C}_2$ of $\mathcal{C}$ where $\mathcal{C}_i$ consists of boundary cycles in $\Gamma_i$ for $i=1,2$; if the standard product subcomplex is maximal, then $\mathcal{C}_1\cup\mathcal{C}_2=\mathcal{C}$, i.e. any boundary cycle in $\Gamma$ is contained in one of the coordinates of any maximal product subcomplex.

Using the discrete Morse theory, most of the results about graph braid groups are obtained and the following fact shows an algebraic similarity between RAAGs and graph 2-braid groups.

\begin{theorem}[\cite{KP12}]\label{scrg}
Let $\Gamma$ be a planar graph. Then $B_2(\Gamma)$ ($PB_2(\Gamma)$, resp.) have group presentations whose relators are commutators corresponding to (ordered, resp.) pairs of disjoint boundary cycles.
In particular, if $\Gamma$ has no pair of disjoint boundary cycles, then $PB_2(\Gamma)$ and $B_2(\Gamma)$ are isomorphic to $\mathbb{F}_n$ ($n\geq 1$).
\end{theorem}
\begin{proof}[Sketch of Proof]
In \cite{KP12}, a cube complex $M_2(\Gamma)$ ($UM_2(\Gamma)$, resp.) is obtained from $D_2(\Gamma)$ ($UD_2(\Gamma)$, resp.) such that (1) $M_2(\Gamma)$ ($UM_2(\Gamma)$, resp.) has a unique vertex, (2) there is a homotopy equivalence $f_D:D_2(\Gamma)\rightarrow M_2(\Gamma)$ ($f_U:UD_2(\Gamma)\rightarrow UM_2(\Gamma)$, resp.), and (3) the union of 2-cubes in $M_2(\Gamma)$ ($UM_2(\Gamma)$, resp.) is the image of the union of all standard tori in $D_2(\Gamma)$ ($UD_2(\Gamma)$, resp.) under $f_D$ ($f_U$, resp.).
The existences of $UM_2(\Gamma)$ and $M_2(\Gamma)$ shows that the theorem holds.
\end{proof}

\textbf{From now on, the defining graphs of graph 2-braid groups and discrete configulation spaces of 2 points (and their universal covers) are assumed to be connected and simplicial even if we do not say.}

%%%%%%%%%%%%%%%%%%%%%%%%%%
%
%
%              CHAPTER 3. Intersection complex
%
%
%%%%%%%%%%%%%%%%%%%%%%%%%%

\section{(Reduced) Intersection Complexes}\label{3}
Throughout this section, let $Y$ and $Y'$ be compact weakly special square complexes with their universal covers $\overline{Y}$ and $\overline{Y'}$, and the universal covering maps $p_Y:\overline{Y}\rightarrow Y$ and $p_{Y'}:\overline{Y'}\rightarrow Y'$, respectively. Suppose that there is a $(\lambda, \varepsilon)$-quasi-isometry  $\phi:\overline{Y}\rightarrow \overline{Y'}$.
We first show that maximal product subcomplexes of $\overline{Y}$ and their intersections are preserved (up to finite Hausdorff distance) by $\phi$. From this fact, we define the \textit{intersection complex} $I(\overline{Y})$ whose vertex set is the collection of maximal product subcomplexes of $\overline{Y}$ and the \textit{semi-isomorphism} $\Phi:I(\overline{Y})\rightarrow I(\overline{Y'})$, an isometry between intersection complexes which is induced by $\phi$. 

%
% 3.1
%

\subsection{Definitions of (Reduced) Intersection Complexes}\label{3.1}
We first deduce that maximal product subcomplexes of $\overline{Y}$ are preserved by a quasi-isometry up to finite Hausdorff distance by using Theorem \ref{FlattoFlat}.

\begin{lemma}[\cite{BKS(a)}]\label{SG}
There is a constant $D_1=D_1(\lambda, \varepsilon)>0$ with the following property:
Suppose that $\overbar{K}=\overbar{P}_1\times\overbar{P}_2\subset Y$ is a product subcomplex where $\overbar{P}_1$ is isometric to an infinite tree which is the union of three singular rays intersecting at one vertex $v$ and $\overbar{P}_2$ is isometric to $\mathbb{R}$.
Then the singular geodesic $\gamma=v\times\overbar{P}_2\subset \overline{Y}$ is mapped by $\phi$ to within Hausdorff distance at most $D_1$ of a singular geodesic $\gamma'\subset \overline{Y'}$ and $\phi(\overbar{K})$ lies in the $D_1$-neighborhood of the parallel set $\mathbb{P}(\gamma')$.
\end{lemma}
\begin{proof}
Let $F_1$, $F_2$ and $F_3$ be three flats contained in $\overbar{K}$ such that $\cup_i F_i=\overbar{K}$ and $\cap_iF_i=\gamma$.
Then the intersection of $C$-neighborhoods of $\phi_H(F_i)$'s contains the quasi-geodesic $\phi(\gamma)$. By Lemma \ref{parallel}, there exist a singular geodesic $\gamma'\subset\overline{Y'}$ and a constant $C_1=C_1(\lambda,\varepsilon,C)>0$ such that the parallel set $\mathbb{P}(\gamma')$ contains $\phi_H(F_i)$'s and $d_H(\gamma',\phi(\gamma))<C_1$.
Let $D_1=\max\{C_1,C\}$ which depends on $\lambda$, $\varepsilon$. Then $\phi(\overbar{K})$ lies in the $D_1$-neighborhood of $\mathbb{P}(\gamma')$.
\end{proof}

\begin{theorem}\label{SPStoSPS}
There exists a constant $D_2=D_2(\lambda, \varepsilon)>0$ such that for any standard product subcomplex $\overbar{K}\subset \overline{Y}$, there exists a standard product subcomplex $\overbar{K}'\subset\overline{Y'}$ such that $\phi(\overbar{K})$ lies in the $D_2$-neighborhood of $\overbar{K}'$.
\end{theorem}
\begin{proof}
Let $\overbar{K}=\overbar{P}_1\times\overbar{P}_2$. Since $\overbar{P}_i$ is isometric to either $\mathbb{R}$ or a standard infinite tree, we divide the statement into three cases. (Recall that $\mathbb{R}$ is not a standard infinite tree; see the first paragraph of Section \ref{2}.)
\vspace{1mm}
 
$\mathbf{Case\ 1.}$ Suppose that both $\overbar{P}_1$ and $\overbar{P}_2$ are isometric to $\mathbb{R}$, i.e. $\overbar{K}$ is a flat in $\overline{Y}$. By Lemma \ref{SubSPC}, the flat $\phi_H(\overbar{K})\subset\overline{Y'}$ is contained in a standard product subcomplex $\overbar{K}'\subset\overline{Y'}$ so that the theorem holds.
\vspace{1mm}

$\mathbf{Case\ 2.}$ Suppose that $\overbar{P}_1$ is isometric to a standard infinite tree and $\overbar{P}_2$ is isometric to $\mathbb{R}$. Let $v\in \overbar{P}_1$ be a vertex whose valency is $\geq 3$.
By Lemma \ref{SG}, $\gamma=v\times\overbar{P}_2$ is mapped by $\phi$ to within Hausdorff distance at most $D_1$ of a singular geodesic $\gamma'$ in $\overline{Y'}$ and $\phi(\overbar{K})\subset\overline{Y'}$ is contained in the $D_1$-neighborhood of $\mathbb{P}(\gamma')$ which is a product subcomplex.
By Lemma \ref{SubSPC}, there is a standard product subcomplex $\overbar{K}'\subset\overline{Y}'$ such that $\mathbb{P}(\gamma')$ is contained in the neighborhood of $\overbar{K}'$. Thus, $\overbar{K}'$ is the desired standard product subcomplex whose neighborhood contains $\phi(\overbar{K})$.
\vspace{1mm}

$\mathbf{Case\ 3.}$ Suppose that both $\overbar{P}_1$ and $\overbar{P}_2$ are isometric to standard infinite trees. 
Let $\mathbb{H}$ ($\mathbb{V}$, resp.) be the collection of singular geodesics of form $\alpha\times w$ ($v\times\beta$, resp.) in $\overbar{K}$ where $w\in \overbar{P}_2$ ($v\in \overbar{P}_1$, resp.) is a vertex. 
For any singular geodesic $\gamma$ in $\mathbb{H}$ or $\mathbb{V}$, by Lemma \ref{SG}, there is a singular geodesic in $\overline{Y'}$ which is within Hausdorff distance at most $D_1$ of $\phi(\gamma)$. Let $\mathbb{H}'$ ($\mathbb{V}'$, resp.) be the collection of such singular geodesics in $\overline{Y'}$ obtained from $\phi$ and the singular geodesics in $\mathbb{H}$ ($\mathbb{V}$, resp.). 
Since any $\alpha\in\mathbb{H}$ and $\beta\in\mathbb{V}$ induce a (unique) flat (i.e. there are singular geodesics $\alpha_1$ and $\beta_1$ parallel to $\alpha$ and $\beta$, respectively, such that $\alpha_1\times\beta_1$ is a flat), by Theorem \ref{FlattoFlat}, any $\alpha'\in\mathbb{H}'$ and $\beta'\in\mathbb{V}'$ also induce a flat (which is contained in $\mathbb{P}(\alpha')\cap\mathbb{P}(\beta')$).

Let $\overbar{U}'$ be the union of all the parallel sets of singular geodesics in $\mathbb{H}'\cup\mathbb{V}'$, 
$$\overbar{U}'=\bigcup_{\alpha'\in\mathbb{H}'}\mathbb{P}(\alpha') \cup\bigcup_{\beta'\in\mathbb{V}'}\mathbb{P}(\beta').$$
By the previous paragraph, $\overbar{U}'$ is connected and by $\mathbf{Case\ 2}$, $\phi(\overbar{K})$ is contained in the neighborhood of $\overbar{U}'$. However, $\overbar{U}'$ may not be a product subcomplex. So, we will find a product subcomplex $\overbar{U}''\subset\overline{Y'}$ such that $d_H(\overbar{U}',\overbar{U}'')$ is finite.

Let $\mathcal{H}'_1$ and $\mathcal{H}'_2$ be the collections of hyperplanes in $\overline{Y'}$ dual to edges of $\alpha'\in\mathbb{H}'$ and $\beta'\in\mathbb{V}'$, respectively.
Then every hyperplane in $\mathcal{H}'_1$ intersects every hyperplane in $\mathcal{H}'_2$ and vice-versa.
Let $\overbar{U}''\subset \overbar{U}'$ be the union of all the flats induced by $\alpha'\in\mathbb{H}'$, $\beta'\in\mathbb{V}'$. For any $h'_1\in\mathcal{H}'_1$, $h'_2\in\mathcal{H}'_2$, there is a unique square in $\overbar{U}''\subset\overline{Y'}$ which both $h_1$ and $h_2$ intersect. Since $\mathcal{H}'_1\cup\mathcal{H}'_2$ is the collection of all the hyperplanes intersecting $\overbar{U}''$, we need to show that $\overbar{U}''$ is connected in order to conclude that $\overbar{U}''$ is a product subcomplex by Proposition 2.6 in \cite{CS11}.

Let $F_1\subset \overbar{K}$ be the flat induced by $\alpha_1\in\mathbb{H}$ and $\beta_1\in\mathbb{V}$, and $F_2\subset \overbar{K}$ the flat induced by $\alpha_2\in\mathbb{H}$ and $\beta_2\in\mathbb{V}$.
If $\alpha_1$ and $\alpha_2$ are not parallel, then there is a singular geodesic $\alpha_3\in\mathbb{H}$ which contains two singular rays $\alpha^+_3,\alpha^-_3$ such that 
\begin{enumerate}
\item
$\alpha^+_3$ is parallel to a singular ray contained in $\alpha_1$ and $\alpha^-_3$ is parallel to a singular ray contained in $\alpha_2$, and 
\item 
$\alpha_i$ and $\alpha_j$ are not parallel for $i,j\in\{1,2,3\}$, $i\neq j$.
\end{enumerate}
Otherwise, let $\alpha_3=\alpha_1$. Similarly, we choose $\beta_3\in\mathbb{V}$ from $\beta_1$, $\beta_2$. 
Let $G_1\subset \overbar{K}$ be the flat induced by $\alpha_3$ and $\beta_1$, $G_2\subset \overbar{K}$ the flat induced by $\alpha_3$ and $\beta_3$, and $G_3\subset \overbar{K}$ the flat induced by $\alpha_3$ and $\beta_2$.
Then $F_1\cap G_1$, $G_1\cap G_2$, $G_2\cap G_3$ and $G_3\cap F_2$ contain half-flats in $\overline{Y}$ so that $\phi_H(F_1)\cap \phi_H(G_1)$, $\phi_H(G_1)\cap \phi_H(G_2)$, $\phi_H(G_2)\cap \phi_H(G_3)$ and $\phi_H(G_3)\cap \phi_H(F_2)$ contain half-flats in $\overline{Y'}$.
Thus, $\overbar{U}''$ is connected.

By Lemma \ref{SubSPC}, there is a standard product subcomplex $\overbar{K}'\subset\overline{Y'}$ such that $\overbar{U}''$ is contained in the neighborhood of $\overbar{K}'$.
Since $\overbar{K}$ is the union of flats and any flats in $\overbar{K}$ are mapped by $\phi$ to within controlled Hausdorff distance of flats in $\overbar{U}''$, $\phi(\overbar{K})$ lies in the neighborhood of $\overbar{U}''$. 
Therefore, $\phi(\overbar{K})$ lies in the neighborhood of the standard product subcomplex $\overbar{K}'$.
\end{proof}

\begin{corollary}\label{MaxtoMax}
There exists a constant $D=D(\lambda, \varepsilon)>0$ such that for any maximal product subcomplex $\overbar{M}\subset\overline{Y}$, there exists a unique maximal product subcomplex $\phi_H(\overbar{M})\subset\overline{Y'}$ such that $d_H(\phi(\overbar{M}),\phi_H(\overbar{M}))<D$.
In particular, the quasi-isometric types of $\overbar{M}$ and $\phi_H(\overbar{M})$ are the same.
\end{corollary}	
\begin{proof}
By Theorem \ref{SPStoSPS}, $\overbar{M}$ is mapped by $\phi$ to within finite Hausdorff distance of a standard product subcomplex $\overbar{M}'\subset\overline{Y'}$.
Let us assume that $\overbar{M}'$ is a maximal product subcomplex. 
By the inverse $\phi^{-1}$, then, $\overbar{M}'$ is mapped into the neighborhood of a standard product subcomplex $\overbar{K}\subset\overline{Y}$. Then $\overbar{M}$ is contained in the neighborhood of $\overbar{K}$.
By Lemma \ref{NS} and the maximality, $\overbar{M}=\overbar{K}$. Therefore, $\overbar{M}'$ is the desired maximal product subcomplex.  
\end{proof}

Now let us see how the quasi-isometry $\phi:\overline{Y}\rightarrow\overline{Y'}$ preserves the intersection of maximal product subcomplexes. Suppose that $\overbar{M}_i\subset\overline{Y}$ is a maximal product subcomplex for $i=1,\cdots,n$ such that the intersection $\overbar{W}=\cap_i\overbar{M}_i$ contains a flat $F$. By Lemma \ref{IntofStd}, there is a standard product subcomplex $\overbar{K}\subset\overbar{W}$ such that $d_H(\overbar{K},\overbar{W})$ is finite, and by Lemma \ref{NS}, any flats in $\overbar{W}$ are contained in $\overbar{K}$.
Since $\phi_H(F)$ is contained in the neighborhood of $\phi_H(\overbar{M}_i)$, by Lemma \ref{NS}, $\phi_H(F)$ is contained in $\phi_H(\overbar{M}_i)$ and thus contained in $\overbar{W}'=\cap_i\phi_H(\overbar{M}_i)$. 
As $\overbar{K}$ is obtained from $\overbar{W}$, a standard product subcomplex $\overbar{K}'\subset\overbar{W}'$ can be obtained such that $d_H(\overbar{K}',\overbar{W}')$ is finite and $\overbar{K}'$ contains any flats in $\overbar{W}'$.
Since $\overbar{W}'$ does not depend on the choice of flat $F\subset\overbar{W}$, we deduce that any flats in $\overbar{W}$ are mapped by $\phi$ into the neighborhood of $\overbar{K}'$. 
From the fact that $\overbar{K}$ is the union of flats contained in $\overbar{W}$, we know that the neighborhood of $\overbar{K}'$ contains $\phi(\overbar{K})$ and from the fact that $\phi^{-1}$ is also a quasi-isometry, $\overbar{K}'$ is quasi-isometric to $\overbar{K}$. 
Therefore, we prove the following theorem which is a more explicit form of Theorem \ref{MaxcontainingFlat}.

\begin{theorem}\label{MaxcontainingFlat2}
For a finite collection of maximal product subcomplexes $\overbar{M}_i\subset\overline{Y}$, if the intersection $\overbar{W}=\cap_i\overbar{M}_i$ contains a flat, then the intersection $\overbar{W}'=\cap_i\phi_H(\overbar{M}_i)$ also contains a flat.
For the standard product subcomplex $\overbar{K}\subset\overline{Y}$ obtained from $\overbar{W}$ as in Lemma \ref{IntofStd}, there is a unique standard product subcomplex $\phi_H(\overbar{K})\subset\overline{Y'}$ such that $\phi_H(\overbar{K})$ is the standard product subcomplex obtained from $\overbar{W}'$ as in Lemma \ref{IntofStd} and $d_H(\phi(\overbar{K}),\phi_H(\overbar{K}))<D$ where $D$ is given in Corollary \ref{MaxtoMax}. 
\end{theorem}

From Corollary \ref{MaxtoMax} and Theorem \ref{MaxcontainingFlat2}, the following complex is obtained from $\overline{Y}$.

\begin{definition}[Intersection Complex, $\mathit{cf}$.\cite{Fer12}]\label{IC}
Let $Y$ be a compact weakly special square complex and $\overline{Y}$ the universal cover of $Y$ with the universal covering map $p_Y:\overline{Y}\rightarrow Y$.
The \textit{intersection complex} $I(\overline{Y})$ is defined as follows:
\begin{enumerate}
\item 
Each vertex in $I(\overline{Y})$ corresponds to a maximal product subcomplex of $\overline{Y}$.
\item 
If the intersection of $(k+1)$ maximal product subcomplexes contains a flat, then the corresponding $(k+1)$ vertices in $I(\overline{Y})$ span a unique $k$-simplex. 
\item 
Each simplex $\triangle\subset I(\overline{Y})$ corresponds to the standard product subcomplex $\overbar{K}_\triangle\subset\overline{Y}$ obtained as in the paragraph above Theorem \ref{MaxcontainingFlat2} and is labelled by the base of $\overbar{K}_\triangle$.
\end{enumerate}
The underlying complex is denoted by $|I(\overline{Y})|$. 
\end{definition}

\begin{convention}
The standard product subcomplex $\overbar{K}_{\mathbf{v}}\subset\overline{Y}$ corresponding to a vertex $\mathbf{v}\in I(\overline{Y})$ is sometimes denoted by $\overbar{M}_{\mathbf{v}}$ when we emphasize that it is a maximal product subcomplex.
\end{convention}

Note that Lemma \ref{IntofStd} guarantees that $|I(\overline{Y})|$ is a simplicial complex.

\begin{proposition}\label{Diam2}
Let $\gamma\subset\overline{Y}$ be a singular geodesic contained in a maximal product subcomplex. Consider maximal product subcomplexes of $\overline{Y}$ containing singular geodesics parallel to $\gamma$ and let $\overbar{V}_\gamma$ be the set of vertices in $I(\overline{Y})$ corresponding to such maximal product subcomplexes.
Then the full subcomplex of $I(\overline{Y})$ spanned by $\overbar{V}_\gamma$ is connected and contained in the star of a vertex. 
\end{proposition}
\begin{proof}
By Lemma \ref{AtMostOneMax} and the assumption, there is a unique maximal product subcomplex $\overbar{M}_\gamma\subset\overline{Y}$ whose neighborhood contains the parallel set $\mathbb{P}(\gamma)$.  
If a maximal product subcomplex $\overbar{M}$ contains a singular geodesic parallel to $\gamma$, then there is a flat $F$ in $\overbar{M}\cap \mathbb{P}(\gamma)$. 
It means that $\overbar{M}$ and $\overbar{M}_\gamma$ have $F$ in common since $F$ is contained in $\overbar{M}_\gamma$. Therefore, the vertex in $I(\overline{Y})$ corresponding to $\overbar{M}_\gamma$ and the vertex in $I(\overline{Y})$ corresponding to $\overbar{M}$ are either identical or joined by an edge. 
\end{proof}

In \cite{Fer12}, the intersection complex of $\overline{D_2(\Gamma)}$ was defined as the underlying complex $|I(\overline{D_2(\Gamma)})|$ of $I(\overline{D_2(\Gamma)})$ in our definition and it was shown that a quasi-isometry $\phi:\overline{D_2(\Gamma)}\rightarrow \overline{D_2(\Gamma')}$ induces an isometry $|\Phi|:|I(\overline{D_2(\Gamma)})|\rightarrow|I(\overline{D_2(\Gamma')})|$. 
The same story holds for a quasi-isometry between the universal covers of compact weakly special square complexes.

\begin{theorem}[$\mathit{cf}$.\cite{Fer12}]\label{IsobetInt}
Let $\phi:\overline{Y}\rightarrow \overline{Y'}$ be a $(\lambda, \varepsilon)$-quasi-isometry. Then $\phi$ induces an isometry $|\Phi|:|I(\overline{Y})| \rightarrow |I(\overline{Y'})|$.
Moreover, $|\Phi|$ preserves the quasi-isometric type of the label of each simplex, i.e. the quasi-isometric type of the label of $\triangle\subset I(\overline{Y})$ is the same as the quasi-isometric type of the label of $|\Phi|(\triangle)$.
\end{theorem}
\begin{proof}
By Corollary \ref{MaxtoMax}, there exists a one-to-one correspondence between the vertex set of $I(\overline{Y})$ and the vertex set of $I(\overline{Y'})$. By Theorem \ref{MaxcontainingFlat2}, there exists a one-to-one correspondence between the set of $k$-simplices of $I(\overline{Y})$ and the set of $k$-simplices of $I(\overline{Y'})$ for any integer $k\geq 0$.
So we can construct a combinatorial isometry $|\Phi|:|I(\overline{Y})|\rightarrow|I(\overline{Y'})|$ from $\phi:\overline{Y}\rightarrow \overline{Y'}$. Since $|\Phi|$ is bijective, $|\Phi|$ is an isometry.
Moreover, by Theorem \ref{MaxcontainingFlat2}, the labels of $\triangle\subset I(\overline{Y_1})$ and $|\Phi|(\triangle)\subset I(\overline{Y_2})$ have the same quasi-isometric type. 
\end{proof}

Let $\triangle\subset I(\overline{Y})$ be a simplex and $\triangle_1$ a codimension-1 face of $\triangle$.
By Lemma \ref{InclusionBetweenSPSes}, $\overbar{K}_{\triangle}$ and $\overbar{K}_{\triangle_1}$ are denoted by $\overbar{P}_1\times \overbar{P}_2$ and $\overbar{P}_3\times \overbar{P}_4$, respectively, such that $\overbar{P}_1\subseteq \overbar{P}_3$ and $\overbar{P}_2\subseteq \overbar{P}_4$. 
Similarly, $\phi_H(\overbar{K}_{\triangle})$ and $\phi_H(\overbar{K}_{\triangle_1})$ are denoted by $\overbar{P}'_1\times \overbar{P}'_2$ and $\overbar{P}'_3\times \overbar{P}'_4$, respectively, such that $\overbar{P}'_1\subseteq \overbar{P}'_3$, $\overbar{P}'_2\subseteq \overbar{P}'_4$.
Note that $\overbar{P}_i$ and $\overbar{P}'_i$ are isometric to either $\mathbb{R}$ or standard infinite trees.
Let us see how $\phi$ preserves the inclusion relation between $\overbar{K}_{\triangle}$ and $\overbar{K}_{\triangle_1}$. First, the numbers of coordinates of $\overbar{K}_{\triangle}$ and $\phi_H(\overbar{K}_{\triangle})$ which are isometric to $\mathbb{R}$, respectively, are the same (for instance, if $\overbar{P}_1$ is isometric to $\mathbb{R}$ but $\overbar{P}_2$ is not, then only one of $\overbar{P}'_1$ or $\overbar{P}'_2$ is isometric to $\mathbb{R}$).
And the numbers of equalities are also the same. More precisely, by choosing suitable flats in $\overline{Y}$ and $\overline{Y'}$, and by using Lemma \ref{SG} and Theorem \ref{FlattoFlat}, we can show the following:
If $\overbar{P}_1=\overbar{P}_3$ and $\overbar{P}_2\subsetneq\overbar{P}_4$, then either $\overbar{P}'_1=\overbar{P}'_3$ and $\overbar{P}'_2\subsetneq\overbar{P}'_4$, or $\overbar{P}'_1\subsetneq\overbar{P}'_3$ and $\overbar{P}'_2=\overbar{P}'_4$.
If $\overbar{P}_1=\overbar{P}_3$ and $\overbar{P}_2=\overbar{P}_4$, then $\overbar{P}'_1=\overbar{P}'_3$ and $\overbar{P}'_2=\overbar{P}'_4$.
If $\overbar{P}_1\subsetneq\overbar{P}_3$ and $\overbar{P}_2\subsetneq\overbar{P}_4$, then $\overbar{P}'_1\subsetneq\overbar{P}'_3$ and $\overbar{P}'_2\subsetneq\overbar{P}'_4$.
In summary, after interchanging the coordinates of $\overbar{K}$ as necessary, 
\begin{enumerate}
\item
$\overbar{P}_i$ is a standard infinite tree if and only if $\overbar{P}'_i$ is a standard infinite tree for $i=1,2,3,4$, and 
\item
$\overbar{P}_{j+1}$ is strictly contained in $\overbar{P}_{j+3}$ if and only if $\overbar{P}'_{j+1}$ is strictly contained in $\overbar{P}'_{j+3}$ for $j=0,1$.
\end{enumerate}
In Section \ref{3.2}, we talk about the way to encode this kind of information in the combinatorial map $|\Phi|:|I(\overline{Y})| \rightarrow |I(\overline{Y'})|$.

\vspace{1mm}

In general, an intersection complex is locally infinite and has infinite diameter. 
To see the local and global structure of an intersection complex, we define a reduced intersection complex, a complex of simplices (possibly not a simplicial complex) which is derived from $Y$ as the simplicial complex $I(\overline{Y})$ is derived from $\overline{Y}$.

\begin{definition}[Reduced Intersection Complex]\label{RI}
Let $Y$ be a compact weakly special square complex. The \textit{reduced intersection complex} $RI(Y)$ is defined as follows:
\begin{enumerate}
\item Each vertex in $RI(Y)$ corresponds to a maximal product subcomplex of $Y$.
\item Suppose that the intersection $W$ of $(k+1)$ maximal product subcomplexes of $Y$ contains $m$ standard product subcomplexes $K_1,\cdots,K_m$ such that each $K_i$ is maximal under the set inclusion among standard product subcomplexes in $W$.
Then, the corresponding $(k+1)$ vertices in $RI(Y)$ span $m$ $k$-simplices.
\item By (2), each simplex $\triangle\subset RI(Y)$ corresponds to a standard product subcomplex $K_\triangle\subset Y$ and is labelled by the base $\bold{K}_\triangle$ of $K_\triangle$.
\end{enumerate}
The underlying complex is denoted by $|RI(Y)|$. 
\end{definition}

\begin{convention}
The standard product subcomplex $K_{\mathbf{u}}\subset{Y}$ corresponding to a vertex $\mathbf{u}\in RI(Y)$ is sometimes denoted by $M_{\mathbf{u}}$ with the base $\bold{M}_{\mathbf{u}}$ when we emphasize that it is a maximal product subcomplex.
\end{convention}

As $Y$ is obtained from $\overline{Y}$ via the action $\pi_1(Y)\curvearrowright\overline{Y}$, $|RI(Y)|$ can also be obtained from $|I(\overline{Y})|$ via the action of $\pi_1(Y)$ on $|I(\overline{Y})|$.

\begin{theorem}\label{TPBCM}
The action of $\pi_1(Y)$ on $\overline{Y}$ by deck transformations induces the action of $\pi_1(Y)$ on $|I(\overline{Y})|$ by isometries and $|RI(Y)|$ is the quotient $|I(\overline{Y})|/\pi_1(Y)$; the combinatorial map $|\rho_Y|:|I(\overline{Y})|\rightarrow |RI(Y)|$ obtained from the action is said to be the $\mathrm{canonical\ quotient\ map}$. 
Moreover, the action of $\pi_1(Y)$ on $|I(\overline{Y})|$ and thus $|\rho_Y|$ preserve the label of each simplex. 
\end{theorem}
\begin{proof}
Each deck transformation sends a standard product subcomplex $\overbar{K}_1\subset\overline{Y}$ to a standard product subcomplex $\overbar{K}_2\subset\overline{Y}$ such that $p_{Y}(\overbar{K}_1)=p_{Y}(\overbar{K}_2)$. Thus, the action of $\pi_1(Y)$ on $\overline{Y}$ induces the action of $\pi_1(Y)$ on $|I(\overline{Y})|$ by isometries such that for any element $g\in\pi_1(Y)$ and any simplex $\triangle\subset I(\overline{Y})$, $g.\triangle$ and $\triangle$ have the same label.
Conversely, for any two standard product subcomplexes $\overbar{K}_1$, $\overbar{K}_2\subset\overline{Y}$ whose images under $p_Y$ are the same, there exists an element $g\in\pi_1(Y)$ such that $g.\overbar{K}_1=\overbar{K}_2$.
This implies that the simplices of $|I(\overline{Y})|$ which have the same label are in the same orbit under the induced action $\pi_1(Y)\curvearrowright |I(\overline{Y})|$. Moreover, for any simplex $\triangle\subset|I(\overline{Y})|$, there is a unique simplex of $|RI(Y)|$ whose label is the same with the label of $\triangle$.
Therefore, the theorem holds.
\end{proof}

Though the reasons are slightly different, both the isometry $|\Phi|$ in Theorem \ref{IsobetInt} and the canonical quotient map $|\rho_Y|$ in Theorem \ref{TPBCM} preserve the inclusion relation between labels. To emphasize this fact, these maps will be denoted by $\Phi:I(\overline{Y})\rightarrow I(\overline{Y'})$ and $\rho_{Y}:I(\overline{Y})\rightarrow RI(Y)$, respectively.
\vspace{1mm}

Let us define the set maps $f_{Y}:RI(Y)\rightarrow 2^Y$ and $\tilde{f}_{Y}:I(\overline{Y})\rightarrow 2^{\overline{Y}}$ as follows:
\begin{enumerate}
\item These maps send each (open) simplex to the corresponding standard product subcomplex.
\item For a (closed) subcomplex $A$ of $RI(Y)$ or $I(\overline{Y})$, the image of $A$ is the union of the images of (open) simplices in $A$ under $f_Y$ or $\tilde{f}_Y$, respectively.
\end{enumerate}  
Then, $f_{Y}\circ \rho_{Y} = p_{Y} \circ \tilde{f}_{Y}$ as in Figure \ref{CommDiagram}. For a component $RI_0$ of $RI(Y)$, a component $I_0$ of $\rho_{Y}^{-1}(RI_0)$ is said to be a $\it{lift}$ of $RI_0$. Then the underlying complexes of lifts of $RI_0$ are all isometric.

\begin{figure}[h]
\begin{tikzpicture}
  \matrix (m) [matrix of math nodes,row sep=4em,column sep=6em,minimum width=4em]
  { I(\overline{Y}) & 2^{\overline{Y}} \\
    RI(Y) & 2^Y \\};
  \path[-stealth]
    (m-1-1) edge node [left] {$\rho_{Y}$} (m-2-1)
            edge node [above] {$\tilde{f}_{Y}$} (m-1-2)
    (m-2-1.east|-m-2-2) edge node [above] {$f_{Y}$} (m-2-2)
 %           node [above] {$\exists$} (m-2-2)
    (m-1-2) edge node [right] {$p_{Y}$} (m-2-2);
\end{tikzpicture}
	\caption{The commutative diagram of intersection complexes}
    	\label{CommDiagram}
\end{figure}

\begin{proposition}\label{Components}
Let $I_1$ and $I_2$ be two lifts of a component $RI_0$ of $RI(Y)$ in $I(\overline{Y})$. Then there is an isometry $|I_1|\rightarrow |I_2|$ which preserves the label of each simplex.
\end{proposition}
\begin{proof}
If $|I_1| \cap |I_2|$ is not empty, then $I_1=I_2$ from the definition of an intersection complex so that $|I_1|\cap |I_2|=\emptyset$. Let $\mathbf{v}_1\in I_1$ and $\mathbf{v}_2\in I_2$ be vertices whose images under the canonical quotient map $\rho_{Y}$ are the same. 
Then $\tilde{f}_{Y}(\mathbf{v}_1)$ and $\tilde{f}_{Y}(\mathbf{v}_2)$ are maximal product subcomplexes of $\overline{Y}$ whose images under $p_{Y}$ are the same and there exists a non-trivial element $g\in \pi_1(Y)$ such that $g.\tilde{f}_{Y}(\mathbf{v}_1)=\tilde{f}_{Y}(\mathbf{v}_2)$. 
It means that $g.|I_1|\cap |I_2|$ is non-empty so that $g.|I_1|=|I_2|$. Since the action of $\pi_1(Y)$ on $|I(\overline{Y})|$ preserves the label of each simplex, the proposition holds. 
\end{proof}

If $RI(Y)$ is connected, then a lift of $RI(Y)$ in $I(\overline{Y})$ is denoted by $I_0(\overline{Y})$ and the restriction of $\rho_Y$ to $I_0(\overline{Y})$ is also called the canonical quotient map.

\begin{convention}
In this paper, the term `lift' is used to denote a component of the preimage of something.
\end{convention}

The following fact about $RI(Y)$ is analogous to Proposition \ref{Diam2} about $I(\overline{Y})$.

\begin{proposition}\label{Diam2'}
Let $\gamma$ be a locally geodesic loop which is in the 1-skeleton $Y^{(1)}$ of $Y$. 
Suppose that there is at least one maximal product subcomplex of $Y$ containing a loop which is freely homotopic to $\gamma$. Let $V_\gamma$ be the set of vertices in $RI(Y)$ corresponding to such maximal product subcomplexes. Then the full subcomplex of $RI(Y)$ spanned by $V_\gamma$ is connected and contained in the star of a vertex. 
\end{proposition}
\begin{proof}
Suppose that $M\subset Y$ is a maximal product subcomplex which contains a loop $\alpha$ freely homotopic to $\gamma$.
Then any p-lift $\overbar{M}$ of $M$ in $\overline{Y}$ contains a lift $\tilde{\alpha}$ of $\alpha$ which is at finite Hausdorff from a lift $\tilde{\gamma}$ of $\gamma$ and by the assumption, any lift of $\gamma$ is a singular geodesic.
Since $\overbar{M}$ is convex, it contains a singular geodesic parallel to $\tilde{\gamma}$.
By Proposition \ref{Diam2}, $\overbar{V}_{\tilde{\gamma}}$ (defined as in Proposition \ref{Diam2}) is then connected and contained in the star of a vertex $\mathbf{v}\in I(\overline{Y})$.
Since $\rho_Y$ is the quotient map obtained from the action of $\pi_1(Y)$ on $|I(\overline{Y})|$ by isometries, for any other lift $\tilde{\gamma}_1$ of $\gamma$, there is an isometry sending $\overbar{V}_{\tilde{\gamma}}$ to $\overbar{V}_{\tilde{\gamma}_1}$.
It means that $\rho_Y(\overbar{V}_{\tilde{\gamma}})=V_\gamma$, and therefore the proposition holds.  
\end{proof}

Contrast to similarities between $p_Y$ and $\rho_Y$, there are also differences between them.
First, the canonical quotient map $\rho_{Y}:I(\overline{Y})\rightarrow RI(Y)$ may not be a local isometry since $|I(\overline{Y})|$ is locally infinite in general. Second, $|I(\overline{Y})|$ may not be contractible though $\overline{Y}$ and even $\tilde{f}_{Y}(I(\overline{Y}))$ are contractible.

%
%  3.2 Realization
%

\subsection{Geometric Realization of $RI(Y)$}\label{3.2}
For a polyhedral complex $\mathcal{X}$, if a group $G$ acts on $\mathcal{X}$ by isometries, then the action $G\curvearrowright \mathcal{X}$ gives rise to a complex of groups $G(\mathcal{Y})$ over the quotient $\mathcal{Y}=\mathcal{X}/G$. 
If $\mathcal{X}$ is simply connected, then the complex of groups $G(\mathcal{Y})$ is said to be $\it{developable}$ and $\mathcal{X}$ is said to be the $\it{development}$ of $G(\mathcal{Y})$. 
The development of a developable complex of groups is analogous to the universal covering tree of a graph of groups in Bass-Serre theory. We refer to Chapter III.$\mathcal{C}$ in \cite{BH}, or \cite{LM} for more about the theory of complexes of groups.

By Theorem \ref{TPBCM}, there is an action $G=\pi_1(Y)$ on $I(\overline{Y})$ such that this action gives rise to the complex of groups structures on $RI(Y)$ and $I_0(\overline{Y})$.
Before we see the details, in order to simplify and avoid some technical difficulties, we assume that (1) both $|RI(Y)|$ and $|I(\overline{Y})|$ are connected, and (2) every standard product subcomplex $K$ of $Y$ is the image of its base $\bold{K}$ under a locally isometric embedding; a compact weakly special square complex $Y$ satisfying item (2) is said to be \textit{simple}.
In particular, the stabilizer of any p-lift of $K$ is isomorphic to $\pi_1(\bold{K})$ which is a join group. Recall that a \textit{join group} is a group isomorphic to $\mathbb{F}_n\times \mathbb{F}_m$ ($n,m\geq 1$).
Under the above assumptions, the complex of group structures of $I(\overline{Y})$ and $RI(Y)$ are described as follows: 
\begin{itemize}
\item
To each simplex $\triangle\subset I(\overline{Y})$, if we assign the stabilizer $G_\triangle$ of the standard product subcomplex $\overbar{K}_\triangle\subset\overline{Y}$, then $I(\overline{Y})$ becomes a complex of groups with naturally induced monomorphisms.
\item 
To each simplex $\triangle'\subset RI(Y)$, the fundamental group of the base of the standard product subcomplex $K_{\triangle'}\subset Y$, denoted by $[G_{\triangle'}]$, is assigned; monomorphisms between assigned groups are induced from inclusions between standard product subcomplexes of $Y$.
\end{itemize}
Note that $[G_{\triangle'}]$ is the conjugacy class of $G_\triangle$ where $\triangle'\subset RI(Y)$ is the image of $\triangle\subset I(\overline{Y})$ under the canonical quotient map $\rho_Y:I(\overline{Y})\rightarrow RI(Y)$.
Both $G_\triangle$ and $[G_{\triangle'}]$ are isomorphic to join groups so that each of them has one of three quasi-isometric types: $\mathbb{Z}\times\mathbb{Z}$, $\mathbb{Z}\times\mathbb{F}$ and $\mathbb{F}\times\mathbb{F}$.

Now we see how the observation below Theorem \ref{IsobetInt} gives a special relation between $G_\triangle$ and $G_{\triangle_1}$ when $\triangle_1$ is a face of $\triangle$.
For a finitely generated non-trivial free group $A_1$, if $A_2$ is a non-trivial free factor of $A_1$, then we say that $A_2$ is contained in $A_1$ under the \textit{free factor inclusion}, denoted by $A_2\leq_f A_1$  ($A_2$ is allowed to be $A_1$).
For a join group $A_1\times B_1$, if $A_2\leq_{f} A_1$ and $B_2\leq_{f} B_1$, then we say that $A_2\times B_2$ is contained in $A_1\times B_1$ under the \textit{pairwise free factor inclusion}, denoted by $A_2\times B_2 \leq_{p} A_1\times B_1$.
And we say that $A_2\times B_2$ is a \textit{free factor subgroup} of $A_1\times B_1$ if $A_2$ is a free factor of exactly one of $A_1$ or $B_1$, and $B_2$ is a free factor of the other one, i.e. $A_2\times B_2$ is contained in exactly one of $A_1\times B_1$ and $B_1\times A_1$ under the pairwise free factor inclusion.

\begin{remark}\label{freefactor}
The notion of `pairwise free factor inclusion' distinguishes $A\times B$ from $B\times A$ but the notion of `free factor subgroup' does not.
So, we say that $A_1\times B_1$ is $\it{pairwise\ isomorphic}$ to $A_2\times B_2$ if and only if $A_1$ is isomorphic to $A_2$ and $B_1$ is isomorphic to $B_2$.
\end{remark}

Using the terms introduced in the paragraph above Remark \ref{freefactor}, we can say that the groups assigned to simplices of $I(\overline{Y})$ satisfy the following properties: 
\begin{enumerate}[label=(\roman*)]
\item  
For any face $\triangle_2$ of $\triangle_1$, $G_{\triangle_1}$ is a free factor subgroup of $G_{\triangle_2}$.

\item
For two simplices $\triangle_1$, $\triangle_2$, suppose that $\triangle_1\not\subset\triangle_2$ and $\triangle_2\not\subset\triangle_1$.
If $G_{\triangle_1}\cap G_{\triangle_2}$ is a join group, then $\triangle_1$ and $\triangle_2$ span a simplex $\triangle$ such that $G_\triangle$ is $G_{\triangle_1}\cap G_{\triangle_2}$.

\item\label{CJCond3} 
Suppose that $\triangle_1$, $\triangle_2$ are two distinct maximal simplices such that $\triangle_1\cap\triangle_2$ contains a simplex $\triangle_3$. 
Then, $G_{\triangle_3}$ properly contains $G_{\triangle_1}$ and $G_{\triangle_2}$ as free factor subgroups.

\item\label{CJCond4}
Let $g_1$ and $g_2$ be two distinct non-trivial elements in $G_{\triangle_1}$ for some simplex $\triangle_1$. Then, neither $(g_1,g_2)$ nor $(g_2,g_1)$ is in $G_{\triangle_2}$ for any simplex $\triangle_2$.

\item\label{CJCond5}
Suppose that $\cap_{i} G_{\triangle_i}$ is not trivial for finitely many simplices $\triangle_i$. Then, the union of $\triangle_i$'s is connected and contained in the star of a vertex (Proposition \ref{Diam2}).
\end{enumerate}
The groups assigned to simplices of $RI(Y)$ also have these properties.
Moreover, for each simplex $\triangle \subset|I(\overline{Y})|$, the inclusion map $\triangle\rightarrow|I(\overline{Y})|$ is injective and this statement also holds for $|RI(Y)|$.
Therefore, we define a class of complexes of groups which contains $I(\overline{Y})$ and $RI(Y)$ as follows.

\begin{definition}\label{CJ}
A \textit{complex of join groups} $X$ is a complex of groups such that
\begin{enumerate}
\item 
for the underlying complex $|X|$ of $X$ and any simplex $\triangle\subset |X|$, the inclusion map $\triangle\rightarrow |X|$ is injective, and
\item
a join group $G_\triangle\cong A_{\triangle}\times B_{\triangle}$ is assigned to each (open) simplex $\triangle\subset|X|$ such that these \textit{assigned groups} satisfy the properties (i)-(v).
\end{enumerate}
\end{definition}

Let $X$ be a complex of join groups. For an $n$-simplex $\triangle^n\subset X$, a \textit{chain} $\mathbf{ch}(\triangle^n)$ of $\triangle^n$ is the sequence $(\triangle^n,\cdots,\triangle^0)$ where $\triangle^{i}$ is a codimension-1 face of $\triangle^{i+1}$ for $i=0,\cdots, n-1$ and a \textit{maximal chain} is a chain of a maximal simplex. 
For a chain $\mathbf{ch}(\triangle^n)$, the sequence $(G_{\triangle^{n}},\cdots,G_{\triangle^{0}})$ of the corresponding assigned groups is called the \textit{assigned group sequence} associated to $\mathbf{ch}(\triangle^n)$. 

Consider the collection $\mathcal{P}_{\mathbf{ch}(\triangle^n)}$ of the assigned groups of simplices in a chain $\mathbf{ch}(\triangle^n)$. From the property \ref{CJCond4} for Definition \ref{CJ}, the coordinates of each assigned group in $\mathcal{P}_{\mathbf{ch}(\triangle^n)}$ can be suitably chosen such that the assigned group sequence associated to $\mathbf{ch}(\triangle^n)$ is denoted by $(A_{\triangle^{n}}\times B_{\triangle^{n}},\cdots,A_{\triangle^{0}}\times B_{\triangle^{0}})$ where 
$$A_{\triangle^{i+1}}\times B_{\triangle^{i+1}}\leq_p A_{\triangle^{i}}\times B_{\triangle^i}\ \mathrm{for}\ i=0,\cdots n-1,$$ 
i.e. if we fix the coordinates of one assigned group in $\mathcal{P}_{\mathbf{ch}(\triangle^n)}$, then the coordinates of all the other assigned groups in $\mathcal{P}_{\mathbf{ch}(\triangle^n)}$ are suitably chosen so that the pairwise free factor inclusion is well-defined in $\mathcal{P}_{\mathbf{ch}(\triangle^n)}$. Thus, we can consider $\mathcal{P}_{\mathbf{ch}(\triangle^n)}$ as a poset under the pairwise free factor inclusion (in this poset, even when $A_{\triangle^{i}}\times B_{\triangle^{i}}=A_{\triangle^{i+1}}\times B_{\triangle^{i+1}}$, we distinguish these join groups in $\mathcal{P}_{\mathbf{ch}(\triangle^n)}$).
\vspace{1mm}

Given two complexes of join groups $X_1$ and $X_2$, suppose that there is a combinatorial map $|\Phi|:|X_1|\rightarrow |X_2|$. 
For a chain $\mathbf{ch}(\triangle^n)=(\triangle^n,\cdots,\triangle^0)$ of an $n$-simplex $\triangle^n\subset X_1$, let $\Phi(\mathbf{ch}(\triangle^n)):=(\Phi(\triangle^n),\cdots,\Phi(\triangle^0))$ be the chain of $\Phi(\triangle^n)$.
For every simplex $\triangle^n\subset X_1$ and every chain $\mathbf{ch}(\triangle^n)$, suppose that the assigned group sequences associated to $\mathbf{ch}(\triangle^n)$ and $\Phi(\mathbf{ch}(\triangle^n))$ are
$$(A_n\times B_n,\cdots, A_0\times B_0)\ \mathrm{and}\ (A'_n\times B'_n,\cdots, A'_0\times B'_0),\ \mathrm{respectively},$$ such that (interchanging the coordinates of assigned groups of all simplices in $\mathbf{ch}(\triangle^n)$ as necessary) $A_{i+1}=A_{i}$ if $A'_{i+1}=A'_i$, and $B_{i+1}=B_{i}$ if $B'_{i+1}=B'_i$. 
Then $\Phi:X_1\rightarrow X_2$ is said to be a \textit{semi-morphism}.
A semi-morphism $\Phi:X_1\rightarrow X_2$ is said to be injective if $|\Phi|:|X_1|\rightarrow |X_2|$ is injective. If the inverse map $\Phi^{-1}$ is also a semi-morphism, then $\Phi$ is said to be a \textit{semi-isomorphism}; $X_1$ and $X_2$ are said to be \textit{semi-isomorphic}. Note that a semi-isomorphism preserves the quasi-isometric type of assigned groups.

By the observation below Theorem \ref{IsobetInt}, if $Y$ and $Y'$ are simple, then the isometry $|\Phi|:|I(\overline{Y})|\rightarrow |I(\overline{Y'})|$ induced by a quasi-isometry $\phi:\overline{Y}\rightarrow\overline{Y'}$ in Theorem \ref{IsobetInt} becomes a semi-isomorphism which proves Theorem \ref{QIimpliesIso1}. Moreover, the canonical quotient map $\rho_Y:I(\overline{Y})\rightarrow RI(Y)$ can be considered as a semi-morphism. 

\begin{theorem}[$\mathit{cf}$. Theorem \ref{IsobetInt}]\label{QIimpliesIso2}
Let $\phi:\overline{Y}\rightarrow \overline{Y'}$ be a $(\lambda, \varepsilon)$-quasi-isometry. 
If both $Y$ and $Y'$ are simple, then $\phi$ induces a semi-isomorphism $\Phi:I(\overline{Y})\rightarrow I(\overline{Y'})$.
\end{theorem}

\begin{theorem}[$\mathit{cf}$. Theorem \ref{TPBCM}]\label{TPBCM2}
If $Y$ is simple, then the induced action of $\pi_1(Y)$ on $I(\overline{Y})$ in Theorem \ref{TPBCM} is by semi-isomorphisms and the canonical quotient map $\rho_Y:I(\overline{Y})\rightarrow RI(Y)$ is a semi-morphism.
\end{theorem}
\begin{proof}
The induced action of $\pi_1(Y)$ on $I(\overline{Y})$ preserves the labels of simplices so that the action is by semi-isomorphisms on $I(\overline{Y})$. Moreover, $\rho_Y$ is also a combinatorial map which preserves the labels of simplices so that $\rho_Y$ is a semi-morphism. 
\end{proof}

Recall that for each simplex $\triangle'\subset RI(Y)$, the standard product subcomplex of $Y$ corresponding to $\triangle'$ is denoted by $K_{\triangle'}$ and the base of $K_{\triangle'}$ is denoted by $\bold{K}_{\triangle'}$.
As $K_{\triangle'}$ is the quotient of a p-lift of $K_{\triangle'}$ by $\pi_1(Y)$, not the stabilizer of the p-lift, $f_Y(RI(Y))$ is a subcomplex of $Y$ which is the quotient of $\tilde{f}_Y(I_0(\overline{Y}))$ by $\pi_1(Y)$, not the stabilizer of $\tilde{f}_Y(I_0(\overline{Y}))$. 
So let us see the quotient of $\tilde{f}_Y(I_0(\overline{Y}))$ by its stabilizer.

For each vertex $\mathbf{u}\in RI(Y)$, fix a local isometry $\iota_{\mathbf{u}}:\bold{K}_{\mathbf{u}}\rightarrow Y$ whose image is $K_{\mathbf{u}}$.
If a simplex $\triangle'\subset RI(Y)$ contains $\mathbf{u}$, then the preimage $\bold{K}_{\mathbf{u},\triangle'}$ of $K_{\triangle'}$ under $\iota_{\mathbf{u}}$ is a standard product subcomplex of $\bold{K}_{\mathbf{u}}$ which can be considered as the base of $K_{\triangle'}$.
For two vertices $\mathbf{u}_1,\mathbf{u}_2\in\triangle'$, $\bold{K}_{\mathbf{u}_1,\triangle'}$ and $\bold{K}_{\mathbf{u}_2,\triangle'}$ are the product of two graphs and there is a canonical coordinate-wise isometry between $\bold{K}_{\mathbf{u}_1,\triangle'}$ and $\bold{K}_{\mathbf{u}_2,\triangle'}$ which comes from the property \ref{CJCond4} for Definition \ref{CJ}.
Thus, a compact special square complex $\mathcal{R}(Y)$ is obtained from the disjoint union of $\bold{K}_{\mathbf{u}}$'s for all the vertices $\mathbf{u}\in RI(Y)$ by coordinate-wisely identifying $\bold{K}_{\mathbf{u}_i,\triangle'}$'s for each simplex $\triangle'$ spanned by $\mathbf{u}_i$'s.
By construction, there exists an immersion $\mathcal{R}(Y)\rightarrow Y$ such that the image is $f_{Y}(RI(Y))$ and the restriction to each $\bold{K}_{\triangle'}$ is equal to the local isometry $\bold{K}_{\triangle'}\rightarrow Y$ whose image is $K_{\triangle'}$.

Until now, we have three groups: the fundamental group $\pi_1(RI(Y))$ of $RI(Y)$ (considered as a complex of groups), $\pi_1(\mathcal{R}(Y))$ and the stabilizer $G$ of $I_0(\overline{Y})$.
In the remaining part of this subsection, we will see a sufficient condition that $I_0(\overline{Y})$ is the development of $RI(Y)$ using the relation between $\pi_1(\mathcal{R}(Y))$ and $G$. In order to do this, we mention a series of results in Chapter III.$\mathcal{C}$ of \cite{BH}.  

\begin{theorem}[Corollary 2.15 in \cite{BH}]\label{developable}
A complex of groups $G(\mathcal{Y})$ is developable if and only if there exists a morphism $\phi$ from $G(\mathcal{Y})$ to some group $G$ which is injective on the local group.
\end{theorem}

\begin{theorem}[Corollary 2.17 in \cite{BH}]\label{GOG}
Every complex of groups $G(\mathcal{Y})$ of dimension one is developable, i.e. a graph of groups is always developable.
\end{theorem}

\begin{theorem}[Theorem 3.13 and Corollary 3.15 in \cite{BH}]\label{development}
The development of a developable complex of groups is simply connected and unique up to isomorphisms.
\end{theorem}

Suppose that every $\bold{K}_{\triangle'}$ is the product $\bold{P}_i\times \bold{P}'_i$ of planar graphs $\bold{P}_i$ and $\bold{P}'_i$. 
Let $\bold{D}_i$ ($\bold{D}'_i$, resp.) be the topological object obtained from $\bold{P}_i$ ($\bold{P}_i$, resp.) by attaching a disk $D^2$ to each boundary cycle and let $\bold{K}_{\triangle'}^d$ be the product $\bold{D}_i\times \bold{D}'_i$ (which is simply connected). 
While constructing $\mathcal{R}(Y)$, replace the bases $\bold{K}_{\triangle'}$ by $\bold{K}^d_{\triangle'}$ and let $\mathcal{R}^d(Y)$ be the resulting topological object.
For each vertex $\mathbf{u}\in RI(Y)$, choose one point $v_{\mathbf{u}}\in \bold{K}^d_{\mathbf{u}}\subset\mathcal{R}^d(Y)$, and for the vertex set $V(RI(Y))$ of $RI(Y)$, let $\psi^{(0)}:V(RI(Y))\rightarrow \mathcal{R}^d(Y)$ be the map sending $\mathbf{u}$ to $v_{\mathbf{u}}$.
For two vertices $\mathbf{u}_1,\mathbf{u}_2\in RI(Y)$ joined by an edge $\bold{E}$, consider a path in $\bold{K}^d_{\mathbf{u}_1}\cup \bold{K}^d_{\mathbf{u}_2}$ joining $v_{\mathbf{u}_1}$ and $v_{\mathbf{u}_2}$ such that if this path meets $\bold{K}^d_{\triangle'}$, then the simplex $\triangle'$ must contain $\bold{E}$. 
Using these paths, the map $\psi^{(0)}$ can be extended to a map $\psi^{(1)}:|RI(Y)|^{(1)}\rightarrow\mathcal{R}^d(Y)$ where $|RI(Y)|^{(1)}$ is the 1-skeleton of $|RI(Y)|$.
From the map $\psi^{(1)}$, a continuous map $\psi:|RI(Y)|\rightarrow\mathcal{R}^d(Y)$ can be obtained such that for a simplex $\triangle'\subset RI(Y)$, if $\psi(\triangle')$ meets $\bold{K}^d_{\triangle''}$, then the simplex $\triangle''$ must contain $\triangle$. 

If the map $\psi$ is a homotopy equivalence, then $\pi_1(RI(Y))$ is isomorphic to $\pi_1(\mathcal{R}(Y))$, and by Theorem \ref{developable}, $RI(Y)$ is developable (in this case, $\mathcal{R}(Y)$ can be considered as a \textit{geometric realization} of $RI(Y)$).
Additionally, if the immersion $\mathcal{R}(Y)\rightarrow Y$ induces an isomorphism $\pi_1(\mathcal{R}(Y))\rightarrow G$, by Theorem \ref{development}, $I_0(\overline{Y})$ is the development of $RI(Y)$ since $RI(Y)$ is the quotient of $I_0(\overline{Y})$ by $G$. In summary,

\begin{proposition}\label{CpxofGps}
Suppose that $Y$ is a compact weakly special square complex such that every standard product subcomplex of $Y$ is the image of the product of two planar standard graphs under a locally isometric embedding.
If the map $\psi:|RI(Y)|\rightarrow\mathcal{R}^d(Y)$ defined as above is a homotopy equivalence, then $RI(Y)$ is a developable complex of groups and $\pi_1(RI(Y))$ is isomorphic to $\pi_1(\mathcal{R}(Y))$.
Additionally, if the immersion $\mathcal{R}(Y)\rightarrow Y$ induces an injective homomorphism $\pi_1(\mathcal{R}(Y))\rightarrow \pi_1(Y)$, then $I_0(\overline{Y})$ is the development of $RI(Y)$ and in particular, $|I_0(\overline{Y})|$ is simply connected.
\end{proposition}

\begin{remark}
It is not true that $\mathcal{R}^d(Y)$ is always homotopy equivalent to $|RI(Y)|$. For instance, if $Y$ is $RI(S(\Lambda))$, then $\mathcal{R}(S(\Lambda))=S(\Lambda)$ and in particular, $\mathcal{R}^d(S(\Lambda))$ is always contractible. 
On the other hand, if $\Lambda$ is an $n$-cycle for $n\geq 5$, then $|RI(S(\Lambda))|$ is also an $n$-cycle (Corollary \ref{Ncycle2}) and in particular, $|RI(S(\Lambda))|$ is not simply connected.
\end{remark}

\begin{remark}\label{RmkofS}
In \cite{BKS(a)} and \cite{HK}, from a simplicial graph $\Lambda$, the \textit{exploded Salvetti complex} $\tilde{S}(\Lambda)$ which is homotopy equivalent to $S(\Lambda)$ is defined. 
Suppose that $\Lambda$ has no induced $n$-cycles for $n\leq 4$.
Then $RI(\tilde{S}(\Lambda))$ is semi-isomorphic to $RI(S(\Lambda))$ and $|RI(S(\Lambda))|$ is homotopy equivalent to $\mathcal{R}^d(\tilde{S}(\Lambda))$ so that $\pi_1(RI(S(\Lambda)))$ is isomorphic to $\pi_1(\mathcal{R}(\tilde{S}(\Lambda)))$. 
But $A(\Lambda)$ is the quotient of $\pi_1(\mathcal{R}(\tilde{S}(\Lambda)))$ by the normal closure of $\pi_1(|RI(\tilde{S}(\Lambda))|)$. In particular, if $\pi_1(|RI(S(\Lambda))|)$ is not trivial, then $\pi_1(RI(S(\Lambda)))$ is not isomorphic to $A(\Lambda)$.
\end{remark}

%%%%%%%%%%%%%%%%%%%%%%%%%%
%
%
%              CHAPTER 4. Intersection complex
%
%
%%%%%%%%%%%%%%%%%%%%%%%%%%

\section{The cases of RAAGs and graph 2-braid groups}\label{4}
Let $Y$ be a compact weakly special square complex which is simple. Suppose that $RI(Y)$ is connected and has a separating vertex $\mathbf{u}$, i.e. $RI(Y)$ is the union of two subcomplexes $RI_1$, $RI_2$ such that $\mathbf{u}=RI_1\cap RI_2$. Let $M\subset Y$ be the maximal product subcomplex corresponding to $\mathbf{u}$.
Then $f_Y(RI(Y))$ is the union of two subcomplexes $Y_1=f_Y(RI_1)$, $Y_2=f_Y(RI_2)$ of $Y$ such that $M$ is the only one maximal product subcomplex contained in $Y_1\cap Y_2$ (note that $Y_1\cap Y_2$ may not be equal to $M$). In particular, $RI_i$ is the reduced intersection complex of $Y_i$ for $i=1,2$.
Then $I(\overline{Y})$ has separating vertices which are induced from $\mathbf{u}$.

\begin{lemma}\label{SeparatingVertex}
Suppose that $Y$, $Y_i$, $M$ and $\mathbf{u}$ are given as above. Then, any vertex in $I(\overline{Y})$ whose image under $\rho_{Y}:I(\overline{Y})\rightarrow RI(Y)$ is $\mathbf{u}$ is a separating vertex.
\end{lemma}
\begin{proof}
For convenience, let us assume that $I(\overline{Y})$ is connected.
Then $\tilde{f}_{Y}(I(\overline{Y}))$, the universal cover of $\mathcal{R}(Y)$, is constructed from the universal covers $\overline{\mathcal{R}(Y_1)}$ and $\overline{\mathcal{R}(Y_2)}$ of $\mathcal{R}(Y_1)$ and $\mathcal{R}(Y_2)$ as follows:
\begin{enumerate}
\item 
Start from $I_{(0)}=\overline{\mathcal{R}(Y_1)}$. To each p-lift $\overbar{M}_1$ of $M$ in $I_{(0)}$, attach a copy of $\overline{\mathcal{R}(Y_2)}$ by identifying a p-lift of $M$ in $\overline{\mathcal{R}(Y_2)}$ with $\overbar{M}_1$. Let $I_{(1)}$ be the resulting complex.
\item
To each p-lift $\overbar{M}_2$ of $M$ in $I_{(1)}$ but not in $I_{(0)}$, attach a copy of $\overline{\mathcal{R}(Y_1)}$ by identifying a p-lift of $M$ in $\overline{\mathcal{R}(Y_1)}$ with $\overbar{M}_2$. Let $I_{(2)}$ be the resulting complex.
\item 
Inductively, construct $I_{(n)}$ by attaching a copy of $\overline{\mathcal{R}(Y_i)}$ to $I_{(n-1)}$ along each p-lift of $M$ in $I_{(n-1)}$ but not in $I_{(n-2)}$ where $i$ is 1 if $n$ is even, and 2 if $n$ is odd. 
Then $\tilde{f}_{Y}(I(\overline{Y}))$ is isometric to the direct limit of $I_{(n)}$.
\end{enumerate}
If a copy of $\overline{\mathcal{R}(Y_1)}$ and a copy of $\overline{\mathcal{R}(Y_2)}$ in $\tilde{f}_{Y}(I(\overline{Y}))$ share a standard product subcomplex, then their intersection is a p-lift of $M$.
Thus, $I(\overline{Y})$ is constructed from copies of $I(\overline{\mathcal{R}(Y_1)})$ and $I(\overline{\mathcal{R}(Y_2)})$.
In particular, the vertex in $I(\overline{Y})$ corresponding to a p-lift of $M$ is a separating vertex.
\end{proof}

Let $Y$ be either a 2-dimensional Salvetti complex or the discrete configuration space of 2 points on a (simplicial) graph. Then $Y$ is a compact weakly special square complexes which are simple so that for every standard product subcomplex $K\subset Y$ with base $\bold{K}$, the local isometry $\bold{K}\rightarrow Y$ is injective. Thus, we will not distinguish between a standard product subcomplex of $Y$ and its base. Moreover, $\mathcal{R}(Y)\rightarrow Y$ is an embedding so that $\mathcal{R}(Y)$ will be considered as a subcomplex of $Y$.
If $|RI(Y)|$ is homotopy equivalent to a point, then it is easily deduced that $RI(Y)$ is a developable complex of groups. 

\begin{lemma}\label{Decomposition}
Let $Y$ be given as above. If $|RI(Y)|$ is contractible, then $RI(Y)$ is a developable complex of groups and a component of $I(\overline{Y})$ is the development of $RI(Y)$. 
\end{lemma}
\begin{proof}
Since $|RI(Y)|$ is contractible, $\mathcal{R}^d(Y)$ is obviously contractible. In particular, the map $\psi:|RI(Y)|\rightarrow\mathcal{R}^d(Y)$ constructed in the paragraph below Theorem \ref{development} is a homotopy equivalence. Moreover, it can be seen that the embedding $\mathcal{R}(Y)\hookrightarrow Y$ induces an injective homomorphism $\pi_1(\mathcal{R}(Y))\rightarrow \pi_1(Y)$. By Proposition \ref{CpxofGps}, therefore, the lemma holds. 
\end{proof}

In the upcoming two subsections, we will see properties of $RI(Y)$ (and $I(\overline{Y})$) when $Y$ is either a 2-dimensional Salvetti complex or $D_2(\Gamma)$ for a planar graph $\Gamma$ with the following questions in mind:
\begin{enumerate}
\item Is $RI(Y)$ always a developable complex of groups? Or when will $RI(Y)$ satisfy the assumptions in Proposition \ref{CpxofGps}?
\item How can $I(\overline{Y})$ be explicitly constructed from $RI(Y)$?
\end{enumerate} 

%
%  4.1 RAAG
%

\subsection{(Reduced) Intersection Complexes for RAAGs}\label{4.1}
In this subsection, we will see the structures of the reduced intersection complex and the intersection complex for a 2-dimensional RAAG.
Recall that for a triangle-free simplicial graph $\Lambda$ which contains at least one edge, there is the associated 2-dimensional Salvetti complex $S(\Lambda)$ with its universal cover $X(\Lambda)$ such that $\pi_1(S(\Lambda))$ is isomorphic to the RAAG $A(\Lambda)$. 
Due to the following facts, we will only consider the case that a triangle-free simplicial $\Lambda$ is connected and has diameter $\geq 3$.
\begin{enumerate}
\item
If $\Lambda$ consists of isolated vertices $v_1,\cdots,v_n$ and components $\Lambda_1,\cdots,\Lambda_m$ each of which contains at least two vertices, then $RI(S(\Lambda))$ is the disjoint union of $RI(S(\Lambda_1)),\cdots,RI(S(\Lambda_m))$. 
\item
When $\Lambda$ is connected, $\Lambda$ is a complete bipartite graph if and only if $S(\Lambda)$ has only one maximal product subcomplex if and only if $|RI(S(\Lambda))|$ (and $|I(X(\Lambda))|$) consists of one vertex.
\end{enumerate}

Each simplex $\triangle'$ of the reduced intersection complex $RI(S(\Lambda))$ of $S(\Lambda)$ corresponds to the standard product subcomplex $K_{\triangle'}\subset S(\Lambda)$ whose defining is a join subgraph $\Lambda_{\triangle'}\leq \Lambda$.
From this fact, we will use the join subgraph $\Lambda_{\triangle'}$ as the label of the simplex $\triangle'$. Similarly, the label of a simplex $\triangle$ of the intersection complex $I(X(\Lambda))$ of $X(\Lambda)$ is the defining graph of the standard product subcomplex $\overbar{K}_{\triangle}\subset X(\Lambda)$ corresponding to $\triangle$.
When $RI(S(\Lambda))$ is considered as a complex of join groups, since $S(\Lambda)$ has a unique vertex, the assigned groups are directly related to the labels of simplices; the assigned group $G_{\triangle'}$ of $\triangle'$ is the canonical subgroup of $A(\Lambda)$ generated by $\Lambda_{\triangle'}$.   

\vspace{1mm}

First, let us see properties of $RI(S(\Lambda))$ which are inherited from properties of $\Lambda$.
\begin{proposition}\label{RIConnected}
$|RI(S(\Lambda))|$ is a connected simplicial complex.
\end{proposition}
\begin{proof}
First, we will show that $|RI(S(\Lambda))|$ is a simplicial complex.
Suppose that there are two edges $E_1,E_2\subset RI(S(\Lambda))$ joining two vertices $\mathbf{u}_1,\mathbf{u}_2\in RI(S(\Lambda))$. Suppose that $\mathbf{u}_i$ and $E_i$ are labelled by $\Lambda_{\mathbf{u}_i}$ and $\Lambda_{E_i}$, respectively, for $i=1,2$.
Since vertices in $\Lambda_{E_1}\cup\Lambda_{E_2}$ are contained in $\Lambda_{\mathbf{u}_1}$ which is a join subgraph, vertices in $\Lambda_{E_1}\cup\Lambda_{E_2}$ span a join subgraph $\Lambda'$ which is contained in $\Lambda_{\mathbf{u}_1}$. Similarly, $\Lambda'$ is also contained in $\Lambda_{\mathbf{u}_2}$.
It means that $\mathbf{u}_1$ and $\mathbf{u}_2$ are actually joined by an edge whose label contains $\Lambda'$. Thus, there is no multi-edge in $|RI(S(\Lambda))|$ and, for the same reason, there is no multi-simplex.

Now let us show that $|RI(S(\Lambda))|$ is connected.
Let $\mathbf{u}$ and $\mathbf{u}'$ be two distinct vertices in $RI(S(\Lambda))$ labelled by join subgraphs $\Lambda_{\mathbf{u}}$ and $\Lambda_{\mathbf{u}'}$, respectively. If $\Lambda_{\mathbf{u}}$ and $\Lambda_{\mathbf{u}'}$ have an edge of $\Lambda$ in common, then $\mathbf{u}$ and $\mathbf{u}'$ are connected by an edge. 
Otherwise, since $\Lambda$ is connected, there exists a geodesic path $(v_0,v_1,\cdots,v_n)$ in $\Lambda^{(1)}$ such that $v_i$'s are vertices in $\Lambda$ and $v_0\in\Lambda_{\mathbf{u}}$, $v_n\in\Lambda_{\mathbf{u}'}$. (It is possible that $n=0$ if $\Lambda_{\mathbf{u}} \cap \Lambda_{\mathbf{u}'}$ consists of vertices.)
For $i=0,\cdots,n$, there exists a maximal join subgraph $\Lambda_i$ containing the star $st(v_i)$ of $v_i$. Then the vertices in $RI(\Lambda)$ labelled by $\Lambda_i$ and $\Lambda_{i+1}$ are connected by an edge of $RI(S(\Lambda))$ for $i=0,\cdots,n-1$ since the intersection of $\Lambda_i$ and $\Lambda_{i+1}$ contains the edge $(v_i,v_{i+1})$ of $\Lambda$. 
Moreover, $\mathbf{u}$ and the vertex in $RI(\Lambda)$ labelled by $\Lambda_0$ are also connected by an edge of $RI(S(\Lambda))$ since $\Lambda_{\mathbf{u}}$ and $\Lambda_0$ have an edge in common. 
Similarly, $\mathbf{u}'$ and the vertex in $RI(\Lambda)$ labelled by $\Lambda_n$ are also connected by an edge. This implies that $\mathbf{u}$ and $\mathbf{u}'$ are in the same component of $RI(S(\Lambda))$. Thus, $|RI(S(\Lambda))|$ is connected.
\end{proof}

\begin{lemma}\label{Ncycle}
If $\Lambda$ contains an induced $n$-cycle for $n\geq 5$, then $|RI(S(\Lambda))|$ contains either an induced $n$-cycle or a $k$-simplex for $k\geq 2$. 
\end{lemma}
\begin{proof}
For an induced $n$-cycle $l=(v_0,\cdots,v_{n-1},v_0)$ in $\Lambda$, let $\Lambda_i\leq\Lambda$ be a maximal join subgraph containing $st(v_i)$ and $\mathbf{u}_i\in RI(S(\Lambda))$ the vertex whose label is $\Lambda_i$. 
In particular, $\Lambda_i=lk(v_i)\circ lk^{\perp}(v_i)$ where $lk(v_i)$ is the link of $v_i$ in $\Lambda$ which contains at least two vertices and $lk^{\perp}(v_i)$ is a discrete subgraph in $\Lambda$. Then $\mathbf{u}_i$ and $\mathbf{u}_{i+1}$ are distinct since $\Lambda_i$ contains $v_{i-1}$ but $\Lambda_{i+1}$ does not (indices mod $n$). 
Moreover, if $v_i$ and $v_j$ are not adjacent, then $\mathbf{u}_i$ and $\mathbf{u}_j$ are distinct since $\Lambda_i$ does not contain $v_j$ but $\Lambda_j$ does. 
Thus, $\mathbf{u}_i$'s are all distinct in $RI(S(\Lambda))$.

Choose two distinct vertices $v_i$ and $v_j$ in $l$.
If $v_i$ and $v_j$ are adjacent in $\Lambda$, then $\Lambda_i\cap\Lambda_j$ contains the edge $(v_i,v_j)$ so that $\mathbf{u}_i$ and $\mathbf{u}_j$ are adjacent in $RI(S(\Lambda))$.
Suppose that $v_i$ and $v_j$ are not adjacent but $\Lambda_i\cap\Lambda_j$ contains an edge $e$. Since it is impossible that $e\subset st(v_i)\cap st(v_j)$), there are two remaining possibilities:\\
(1) Suppose that $e\subset st(v_i)$ but not in $st(v_j)$. Then one of the endpoints of $e$ is $v_i$ and the other is in $lk(v_j)$; if the other is in $lk^{\perp}(v_j)$, then $v_i$ and $v_j$ are adjacent, a contradiction.
Since $e$ is contained in $\Lambda_j$, $v_i$ is in $lk^{\perp}(v_j)$ and in particular, $lk(v_i)$ contains $lk(v_j)$. 
But it is impossible since two vertices in $l$ adjacent to $v_j$ are in $lk(v_j)$ so that they are adjacent to $v_i$. \\
(2) Suppose that $e$ is contained in neither $st(v_i)$ nor $st(v_j)$. In particular, $lk^{\perp}(v_i)$ and $lk^{\perp}(v_j)$ also have at least two vertices.
Let $v$ and $w$ be the endpoints of $e$. 
By re-indexing, let $\Lambda_i=\Lambda'_i\circ\Lambda''_i$ and $\Lambda_j=\Lambda'_j\circ\Lambda''_j$ so that $v\in\Lambda'_i\cap\Lambda'_j$ and $w\in\Lambda''_i\cap\Lambda''_j$.
Then $lk(v)$ contains $\Lambda''_i\cup\Lambda''_j$ and $lk(w)$ contains $\Lambda'_i\cup\Lambda'_j$.
This implies that at least one of $st(v)$ and $st(w)$ is not contained in $\Lambda_i\cap\Lambda_j$; for if both $st(v)$ and $st(w)$ are contained in $\Lambda_i\cap\Lambda_j$, then $lk(v)=\Lambda''_i=\Lambda''_j$ and $lk(w)=\Lambda'_i=\Lambda'_j$, a contradiction.
So, say $st(v)$ is not contained in $\Lambda_i\cap\Lambda_j$ and let $\mathbf{u}\in RI(S(\Lambda))$ be the vertex whose label contains $st(v)$. Then $\mathbf{u}$ is distinct from $\mathbf{u}_i$ and $\mathbf{u}_j$, and three vertices $\mathbf{u}$, $\mathbf{u}_i$, $\mathbf{u}_j$ span a $2$-simplex since their labels have the edge $e$ in common.

In summary, for two non-adjacent vertices $v_i$ and $v_j$ in $l$, either $\mathbf{u}_i$ and $\mathbf{u}_j$ are adjacent in $RI(S(\Lambda))$ or $\mathbf{u}_i$ and $\mathbf{u}_j$ are contained in a $k$-simplex for $k\geq 2$. Therefore, if the loop $(\mathbf{u}_0,\cdots,\mathbf{u}_{n-1},\mathbf{u}_0)$ in $RI(S(\Lambda))$ is not an induced $n$-cycle, then $RI(S(\Lambda))$ contains a $k$-simplex for $k\geq 2$.
\end{proof}

\begin{corollary}\label{Ncycle2}
Suppose that $\Lambda$ has no induced $n$-cycles for $n\leq 4$. If $\Lambda$ has an induced $m$-cycle for $m\geq 5$, then $|RI(S(\Lambda))|$ has an induced $m$-cycle.
\end{corollary}
\begin{proof}
By Lemma \ref{Ncycle}, an induced $m$-cycle $(v_0,\cdots,v_{m-1},v_0)$ in $\Lambda$ induces an $m$-cycle $\mathbf{l}=(\mathbf{u}_0,\cdots,\mathbf{u}_{m-1},\mathbf{u}_0)$ in $RI(S(\Lambda))$ such that the label of $\mathbf{u}_i$ contains $st(v_i)$. 
By the assumption that $\Lambda$ has no squares, $RI(S(\Lambda))$ has no vertices whose assigned groups are quasi-isometric to $\mathbb{F}\times\mathbb{F}$ and in particular, $st(v_i)$ is a maximal join subgraph. It means that the second possibility in the proof of Lemma \ref{Ncycle} cannot happen and therefore, $\mathbf{l}$ is an $m$-induced cycle in $RI(S(\Lambda))$.
\end{proof}

Let us see the structure of $I(X(\Lambda))$ inherited from $RI(S(\Lambda))$.
For each vertex $x\in X(\Lambda)$, consider all the maximal product subcomplexes of $X(\Lambda)$ containing $x$ and let $R_x\subset I(X(\Lambda))$ be the full subcomplex spanned by the vertices corresponding to these maximal product subcomplexes. Since $S(\Lambda)$ has one vertex, the restriction $\rho_{S(\Lambda)}|_{R_x}$ of the canonical quotient map $\rho_{S(\Lambda)}:I(X(\Lambda))\rightarrow RI(S(\Lambda))$ to $R_x$ is a semi-isomorphism and $I(X(\Lambda))$ is covered by $\{R_x\ |\ x\in X(\Lambda)\}$.
Thus, connectedness of $RI(S(\Lambda))$ induces connectedness of $I(X(\Lambda))$.

\begin{proposition}
$|I(X(\Lambda))|$ is connected.
\end{proposition}
\begin{proof}
Let $\mathbf{v}$ and $\mathbf{v}'$ be two distinct vertices in $I(X(\Lambda))$, and let $\overbar{M}_\mathbf{v}$ and $\overbar{M}_{\mathbf{v}'}$ be the maximal product subcomplexes of $X(\Lambda)$ corresponding to $\mathbf{v}$ and $\mathbf{v}'$, respectively.
If $\overbar{M}_\mathbf{v}$ and $\overbar{M}_{\mathbf{v}'}$ have a standard product subcomplex in common, then $\mathbf{v}$ and $\mathbf{v}'$ are connected by an edge of $I(X(\Lambda))$. 
Otherwise, consider a geodesic path $(x_0,x_1,\cdots,x_m)$ in $X(\Lambda)$ where $x_0\in \overbar{M}_{\mathbf{v}}$, $x_m\in \overbar{M}_{\mathbf{v}'}$ (it is possible that $m=0$ if $\overbar{M}_\mathbf{v}\cap \overbar{M}_{\mathbf{v}'}$ is not empty).
Let $\Lambda_i\leq\Lambda$ be a maximal join subgraph containing the star of the vertex $v_i\in\Lambda$ which is the label of the edge $(x_i,x_{i+1})\subset X(\Lambda)$ for $i=0,\cdots,m-1$.
Let $\overbar{M}_i\subset X(\Lambda)$ be the maximal product subcomplex with defining graph $\Lambda_i$ which contains $x_i$, and let $\mathbf{v}_i\in I(X(\Lambda))$ be the vertex corresponding to $\overbar{M}_i$. 
Since $\overbar{M}_\mathbf{v}$ and $\overbar{M}_{\mathbf{v}_0}$ have $x_0$ in common, $\mathbf{v}$ and $\mathbf{v}_0$ are contained in $R_{x_0}$. Since $R_{x_0}$ is semi-isomorphic to $RI(S(\Lambda))$ and $RI(S(\Lambda))$ is connected, by Proposition \ref{RIConnected}, there exists a path $c_1$ in $I(X(\Lambda))$ from $\mathbf{v}$ to $\mathbf{v}_0$. 
Similarly, there exists a path $c_2$ in $I(X(\Lambda))$ from $\mathbf{v}_{m-1}$ to $\mathbf{v}'$ since $\overbar{M}_{\mathbf{v}_{m-1}}$ and $\overbar{M}_{\mathbf{v}'}$ have $x_m$ in common.
For each $i\in\{0,\cdots,m-1\}$, there exists a path $t_i$ in $I(X(\Lambda))$ from $\mathbf{v}_i$ to $\mathbf{v}_{i+1}$ since $\overbar{M}_{\mathbf{v}_i}$ and $\overbar{M}_{\mathbf{v}_{i+1}}$ have $x_{i+1}$ in common. 
Therefore, $\mathbf{v}$ and $\mathbf{v}'$ are joined by the concatenation of $c_1$, $t_i$'s and $c_2$.
\end{proof}

By Lemma \ref{Decomposition}, if $|RI(S(\Lambda))|$ is contractible, then $|I(X(\Lambda))|$ is simply connected. The following is a partial converse with its immediate consequence.

\begin{proposition}\label{ContractibleRI}
If $|I(X(\Lambda))|$ is simply connected, then $|RI(S(\Lambda))|$ is also simply connected. In particular, if $|RI(S(\Lambda))|$ is not simply connected, then $I(X(\Lambda))$ is not the development of $RI(S(\Lambda))$.
\end{proposition}
\begin{proof}
Assume that $|RI(S(\Lambda))|$ contains a non-trivial loop $l$. For a vertex $x\in X(\Lambda)$, let $\overline{l}$ be the loop in $R_x\subset I(X(\Lambda))$ which is the preimage of $l$ under $\rho_{S(\Lambda)}|_{R_x}$.
By the assumption, there exists a disk $D$ in $|I(X(\Lambda))|$ such that $D$ is contained in the union of simplices of $|I(X(\Lambda))|$ and the boundary of $D$ is $\overline{l}$. This implies that $\rho_{S(\Lambda)}(D)$ is an immersed disk in $|RI(S(\Lambda))|$ with the boundary $l$ which means that $l$ is a trivial loop, a contradiction. 
Therefore, $|RI(S(\Lambda))|$ is simply connected. Theorem \ref{development} immediately deduces the latter statement.
\end{proof}

The covering $\{R_x\}_{x\in X(\Lambda)}$ of $I(X(\Lambda))$ is related to the join length metric on $A(\Lambda)$.
The \textit{join length} $||g||_J$ of $g\in A(\Lambda)$ is the minimum $l$ such that $g$ can be written as the product of $l$ elements in $\coprod_{\Lambda'\in\mathcal{J}(\Lambda)} A(\Lambda')$ where $\mathcal{J}(\Lambda)$ is the collection of all maximal join subgraphs of $\Lambda$ (the identity $id$ of $A(\Lambda)$ has join length 0). The join length induces a metric $d_J$ on $A(\Lambda)$.
Since $|I(X(\Lambda))|$ is a finite dimensional simplicial complex, $|I(X(\Lambda))|$ can be considered as a metric space endowed with the path metric by considering that each edge has lengh 1 and each simplex is a Euclidean regular simplex. Then $|I(X(\Lambda))|$ is a kind of geometric realization of $A(\Lambda)$ endowed with the metric $d_J$.

\begin{theorem}\label{Loop}
Consider $|I(X(\Lambda))|$ as a metric space endowed with the usual path metric as above. Then $|I(X(\Lambda))|$ is quasi-isometric to the metric space $(A(\Lambda),d_J)$ and in particular, $|I(X(\Lambda))|$ has infinite diameter.
\end{theorem}
\begin{proof}
For two distinct vertices $x,y\in X(\Lambda)$, $y$ is in $\tilde{f}_{S(\Lambda)}(R_x)$ if and only if there is a standard product subcomplex containing both $x$ and $y$. It means that $R_x\cap R_y$ is the subcomplex of $I(X(\Lambda))$ spanned by the vertices corresponding to the maximal product subcomplexes containing both $x$ and $y$. 
From the fact that vertices in $X(\Lambda)$ correspond to elements of $A(\Lambda)$, we deduce that $R_x\cap R_y\neq\emptyset$ if and only if $d_J(x,y)=1$. Therefore, $|I(X(\Lambda))|$ is quasi-isometric to $(A(\Lambda),d_J)$.
Since we exclude the case that $\Lambda$ is a complete bipartite graph, the diameter of $|I(X(\Lambda))|$ is infinite.
\end{proof}

The definition of the join length is similar to the definition of the star length. The \textit{star length} $||g||_*$ of $g\in A(\Lambda)$ is the word length with respect to the generating set $\coprod_{v\in \Lambda} \langle st(v) \rangle$. 
In \cite{KK14}, Kim and Koberda showed that $A(\Lambda)$ with the metric $d_*$ induced from the star length is quasi-isometric to the extension graph of $A(\Lambda)$ which is a quasi-tree. Moreover, they showed that $A(\Lambda)$ is weakly hyperbolic relative to $\{\langle st(v)\rangle\ |\ v\in \Lambda\}$.

The star metric and the join metric on $A(\Lambda)$ are actually Lipschitz equivalent due to the following two observations: First, all the vertices in a maximal join subgraph are contained in the union of vertices in two stars, i.e. if $\Lambda'=\Lambda_1\circ\Lambda_2$ is a maximal join subgraph of $\Lambda$, then $\Lambda'$ is contained in the subgraph spanned by $st(v)\cup st(w)$ for some $v\in\Lambda_1$ and $w\in\Lambda_2$.
Second, every star is contained in some maximal join subgraph. 
These two facts imply that $||g||_J\leq ||g||_*\leq 2||g||_J$ so that $(A(\Lambda),d_J)$ is quasi-isometric to $(A(\Lambda),d_*)$. Therefore, we immediately obtain the following result.

\begin{corollary}\label{Quasitree}
For a triangle-free graph $\Lambda$, $|I(X(\Lambda))|$ is a quasi-tree. In particular, $A(\Lambda)$ is weakly hyperbolic relative to $\{A(\Lambda')\ |\ \Lambda'\in \mathcal{J}(\Lambda)\}$.
\end{corollary}

The fact that $|I(X(\Lambda))|$ is a quasi-tree implies that there are infinitely many vertices whose $\delta$-neighborhoods separate $|I(X(\Lambda))|$ for some $\delta\geq 0$. Actually, we can show that the 2-neighborhoods of vertices in $I(X(\Lambda))$ whose labels contain the closed stars of vertices in $\Lambda$ separate $|I(X(\Lambda))|$. Before we see this, let us see the intersection of $R_x$ and $R_y$ for two vertices $x,y\in X(\Lambda)$.

\begin{lemma}\label{IntersectingR}
Let $R_x$ and $R_y$ be two distinct elements in the covering $\{R_x\}_{x\in X(\Lambda)}$ of $I(X(\Lambda))$. If $R_x\cap R_y$ is non-empty, then one of the following holds:
\begin{enumerate}
\item $R_x\cap R_y$ is a simplex.
\item There is a vertex $\mathbf{v}\in R_x\cap R_y$ such that $R_x\cap R_y=st_{R_x}(\mathbf{v})=st_{R_y}(\mathbf{v})$.
\end{enumerate}
\end{lemma}
\begin{proof}
Note that a simplex $\triangle$ is contained in $R_x\cap R_y$ if and only if the standard product subcomplex $\overbar{K}_{\triangle}\subset X(\Lambda)$ contains both $x$ and $y$. 
If there is one maximal product subcomplex of $X(\Lambda)$, then $R_x\cap R_y$ is a vertex corresponding to this maximal product subcomplex.

Suppose that there are at least two maximal product subcomplexes of $X(\Lambda)$ containing both $x$ and $y$. Then there are two cases:\\
$\mathbf{Case\ 1.}$ Suppose that there is a standard geodesic $\gamma$ containing both $x$ and $y$. By Lemma \ref{AtMostOneMax}, there is a unique maximal product subcomplex $\overbar{M}_{\gamma}$ whose defining graph contains $st(v)$. Following the proof of Proposition \ref{Diam2}, any maximal product subcomplex of $X(\Lambda)$ containing a singular geodesic parallel to $\gamma$ shares a flat with $\overbar{M}_{\gamma}$.
Thus, $R_x\cap R_y=st_{R_x}(\mathbf{v})=st_{R_y}(\mathbf{v})$.\\
$\mathbf{Case\ 2.}$ Suppose that there is no standard geodesic containing both $x$ and $y$.
Let $\overbar{K}\subset X(\Lambda)$ be the minimal standard product subcomplex containing both $x$ and $y$; the defining graph of $\overbar{K}$ is the join subgraph $\Lambda_{\overbar{K}}$ spanned by vertices which are the labels of edges in a geodesic segment between $x$ and $y$. 
For the collection $\{\Lambda_i\leq\Lambda\}_{i\in I}$ of labels of simplices of $RI(S(\Lambda))$, let $\Lambda_{\{x,y\}}$ be the intersection of $\Lambda_i$'s which contain $\Lambda_{\overbar{K}}$. Then $R_x\cap R_y$ is the simplex whose label is $\Lambda_{\{x,y\}}$.
\end{proof}

\begin{lemma}\label{Separating}
Let $v\in\Lambda$ be a vertex of valency $\geq 2$ and let $\mathbf{v}\in I(X(\Lambda))$ be a vertex whose label contains $st_{\Lambda}(v)$. Then the star $st_{I(X(\Lambda))}(\mathbf{v})$ separates $I(X(\Lambda))$.
\end{lemma}
\begin{proof}
Let $\overbar{M}_\mathbf{v}$ be the maximal product subcomplex of $X(\Lambda)$ corresponding to $\mathbf{v}$. For a standard geodesic $l$ in $\overbar{M}_\mathbf{v}$ labelled by $v$, $\overbar{M}_\mathbf{v}$ contains the parallel set $\mathbb{P}(l)$. 
Since $\Lambda$ has diameter $\geq 3$, there exists a vertex $v_0\in\Lambda$ which is not contained in the defining graph of $\overbar{M}_\mathbf{v}$ (containing $st(v)$). Let $e\subset X(\Lambda)$ be an edge in $l$ and let $e_1$ and $e_2$ be the edges of $X(\Lambda)$ labelled by $v_0$ such that $e_1\cap e$ and $e_2\cap e$ are two distinct endpoints of $e$; in particular, there is a singular geodesic in $X(\Lambda)$ which contains three consecutive edges $e_1$, $e$, $e_2$.  
For two components of $X(\Lambda)- h_e$, let $X_1$ be the component containing $e_1$ and $X_2$ the component containing $e_2$.
Let $y_i\in X(\Lambda)$ be the endpoint of $e_i$ not in $e$ for $i=1,2$. 
Then every combinatorial path in $X(\Lambda)^{(1)}$ joining $y_1$ to $y_2$ passes through $h_e$.
Let $\mathbf{v}_1\in R_{y_1}$ and $\mathbf{v}_2\in R_{y_2}$ be the vertices whose images under $\rho_{S(\Lambda)}$ are $\rho_{S(\Lambda)}(\mathbf{v})$. Then $\mathbf{v}$ is adjacent to neither $\mathbf{v}_1$ nor $\mathbf{v}_2$ in $I(\overline{Y})$; the defining graphs of $\overbar{M}_{\mathbf{v}}$ and $\overbar{M}_{\mathbf{v}_i}$ are the same but $\overbar{M}_{\mathbf{v}}$ does not contain $y_i$. 
In order to prove the lemma, we will show the following claim.

\begin{claim}
Every path in $I(X(\Lambda))$ joining $\mathbf{v}_1$ to $\mathbf{v}_2$ passes through $st(\mathbf{v})$.
\end{claim}

For any combinatorial path $\overline{c}$ in $I(X(\Lambda))$ joining $\mathbf{v}_1$ to $\mathbf{v}_2$, there is a sequence $(y_1=x_1,\cdots,x_n=y_2)$ of vertices in $X(\Lambda)$ such that $\cup R_{x_i}$ contains $\overline{c}$ and $R_{x_i}\cap R_{x_{i+1}}$ is non-empty, i.e. there is a standard product subcomplex of $X(\Lambda)$ containing $x_i$ and $x_{i+1}$.
Since $x_1$ is in $X_1$ but $x_n$ is in $X_2$, there exists $1\leq i\leq n-1$ such that $x_{i}$ is in $X_1$ but $x_{i+1}$ is in $X_2$.
Let $\overbar{K}\subset X(\Lambda)$ be a standard product subcomplex containing $x_i$ and $x_{i+1}$. Then $\overbar{K}$ must contain a standard geodesic parallel to $l$.
From Lemma \ref{Separating} and its proof, we deduce that $R_{x_i}\cap R_{x_{i+1}}$ is either a simplex containing $\mathbf{v}$ or the star of $\mathbf{v}$ in $R_{x_i}$. In particular, $R_{x_i}\cap R_{x_{i+1}}$ is contained in $st(\mathbf{v})$. It means that $\overline{c}$ passes through $st(\mathbf{v})$ and therefore, the claim holds.
\end{proof}

\begin{remark}\label{ConstructionOfX}
Let $A(\Lambda)^{(i)}$ be the subset of $A(\Lambda)$ whose elements have join length $i$ for integers $i\geq 0$. In particular, $A(\Lambda)$ (as a set) is the disjoint union of $A(\Lambda)^{(i)}$.
From $RI(S(\Lambda))$, $I(X(\Lambda))$ is constructed by the following three steps:

(1) Let $R$ be the copy of $RI(S(\Lambda))$ correspondig to the identity of $A(\Lambda)$ and let $gR$ be the copy of $RI(S(\Lambda))$ corresponding to $g\in A(\Lambda)^{(1)}$. Let $\triangle'_g$ be the simplex of $gR$ corresponding to the simplex $\triangle'\subset RI(S(\Lambda))$ such that the assigned group is the conjugation $(G_{\triangle'})^g$ of $G_{\triangle'}$.
If $\triangle'_g$ is assumed to be labelled by the left coset $gG_{\triangle'}$ of $G_{\triangle'}$, with the left action of $A(\Lambda)$ on $A(\Lambda)$, $(G_{\triangle'})^g$ is the stabilizer of $gG_{\triangle'}$ in $A(\Lambda)$. 
From $R$ and $gR$'s for $g\in A(\Lambda)^{(1)}$, a complex of join groups is obtained by identifying simplices which have the same assigned groups, i.e. for $g,h\in \{id\}\cup A(\Lambda)^{(1)}$, if $g^{-1}h\in G_{\triangle'}$, then, the simplex $\triangle'_g$ of $gR$ and the simplex $\triangle'_h$ of $hR$ are identified. Let $R^{(1)}$ be the resulting complex of join groups.
 
(2) As above, consider the copies of $RI(S(\Lambda))$ such that these copies correspond to elements of $A(\Lambda)^{(2)}$. 
For $g\in A(\Lambda)^{(2)}$, we define $gR$, $\triangle'_g$ as above such that the assigned group of $\triangle'_g$ is $(G_{\triangle'})^g$.
Then a new complex of join groups $R^{(2)}$ is obtained from $R^{(1)}$ and $gR$'s for $g\in A(\Lambda)^{(2)}$ by identifying simplices which have the same assigned groups.

(3) Inductively, for $g\in A(\Lambda)^{(i)}$, we define $gR$, $\triangle'_g$ and the assigned group $(G_{\triangle'})^g$. 
Then $R^{(i)}$ is constructed from $R^{(i-1)}$ and $gR$'s for $g\in A(\Lambda)^{(i)}$, and there is a canonical semi-morphism $\rho_{S(\Lambda)}^{(i)}:R^{(i)}\rightarrow RI(S(\Lambda))$. The direct limit of $R^{(i)}$ is $I(X(\Lambda))$ and the induced semi-morphism $\rho_{S(\Lambda)}:I(X(\Lambda))\rightarrow RI(S(\Lambda))$ is the canonical quotient map obtained in Theorem \ref{TPBCM2}. 
\end{remark}

A vertex $\mathbf{u}\in RI(S(\Lambda))$ is said to be a \textit{type-1 vertex} if it is not a separating vertex and the assigned group $G_\mathbf{u}$ of $\mathbf{u}$ is contained in the subgroup of $A(\Lambda)$ generated by $G_{\mathbf{u}_i}$'s for all the vertices $\mathbf{u}_i$ adjacent to $\mathbf{u}$. 
Otherwise, $\mathbf{u}$ is said to be a \textit{type-2 vertex}.
If a non-separating vertex $\mathbf{u}$ is type-2, then there is a vertex $a\in \Lambda$ such that $G_{\mathbf{u}}$ is the only assigned group which contains $a$. 
A vertex $\mathbf{v}\in I(X(\Lambda))$ is said to be a \textit{type-i} vertex if $\rho_{S(\Lambda)}(\mathbf{v})$ is type-$i$ for $i=1,2$.
Note that if $\mathbf{v}$ is type-1, then any $g\in G_\mathbf{v}$ is decomposed as $g=g_1\cdots g_n$ where $g_i$ is contained in $G_\mathbf{v}\cap G_{\mathbf{v}_i}$ for a vertex $\mathbf{v}_i$ adjacent to $\mathbf{v}$.

\begin{lemma}\label{Type2Vertex}
Let $\mathbf{v}$ be a vertex in $I(X(\Lambda))$. If $\mathbf{v}$ is type-1, then $\mathbf{v}$ is not a separating vertex. Otherwise, $\mathbf{v}$ is a separating vertex. 
\end{lemma}
\begin{proof}
If $\rho_{S(\Lambda)}(\mathbf{v})$ is a separating vertex in $RI(S(\Lambda))$, then $RI(S(\Lambda))$ consists of two subcomplexes $RI_1$ and $RI_2$ whose intersection is $\rho_{S(\Lambda)}(\mathbf{v})$.
Let $S_1$ and $S_2$ be $f_{S(\Lambda)}(RI_1)$ and $f_{S(\Lambda)}(RI_2)$, respectively; $S(\Lambda)$ is the union of $S_1$ and $S_2$ and $S_1\cap S_2$ is the maximal product subcomplex corresponding to $\rho_{S(\Lambda)}(\mathbf{v})$.
By Lemma \ref{SeparatingVertex}, then, $\mathbf{v}$ is a separating vertex in $I(X(\Lambda))$.

Suppose that $\rho_{S(\Lambda)}(\mathbf{v})$ is not a separating vertex. Let $R$ be the copy of $RI(S(\Lambda))$ in $I(\overline{Y})$ corresponding to the identity element in $A(\Lambda)$.
Without loss of generality, assume that $\mathbf{v}$ is in $R$. Then there are two cases depending on the type of $\mathbf{v}$.

$\mathbf{Case\ 1.}$ Suppose that there is an element $a\in \mathcal{S}$ which makes $\mathbf{v}$ type-2. Then $G_\mathbf{v}$ is the only assigned group of a simplex of $I(X(\Lambda))$ containing $a$ and the intersection of $R$ and $aR$ is $\mathbf{v}$. 
If $\mathbf{v}$ is not a separating vertex in $I(X(\Lambda))$, then there exists a sequence $R=g_0 R, g_1 R,\cdots, g_n R=aR$ such that the intersection of $g_i R$ and $g_{i+1} R$ is non-empty but not equal to $\{\mathbf{v}\}$. It means that $\overline{g_i}g_{i+1}$ is contained in the assigned group of a vertex in $RI(S(\Lambda))$ which is not $\rho_{S(\Lambda)}(\mathbf{v})$.
However, it is a contradiction by the definition of type-2 vertices since $a$ is the concatenation of such $\overline{g_i}g_{i+1}$'s. Hence, $\mathbf{v}$ is a seperating vertex of $I(X(\Lambda))$. 

$\mathbf{Case\ 2.}$ Suppose that $\mathbf{v}$ is type-1. Then, $g\in G_{\mathbf{v}}$ is decomposed as $g=g_1\cdots g_n$ where $g_i$ is contained in $G_{\mathbf{v}}\cap G_{\mathbf{v}_i}$ for a vertex $\mathbf{v}_i$ adjacent to $\mathbf{v}$; $R\cap g_i R$ contains at least one edge which contains $\mathbf{v}$. 
Then $g_1\cdots g_{k+1}R\cap g_1\cdots g_{k}R$ contains an edge for $k=1,\cdots,n-1$. Therefore, $\mathbf{v}$ is not a seperating vertex of $I(X(\Lambda))$.
\end{proof}

Even though $\mathcal{R}(S(\Lambda))=f_{S(\Lambda)}(RI(S(\Lambda)))=S(\Lambda)$, $I(X(\Lambda))$ is not the development of $RI(S(\Lambda))$ in general (for instance, if $\pi_1(|RI(S(\Lambda))|)$ is non-trivial, by Proposition \ref{ContractibleRI}, $I(X(\Lambda))$ is not the development of $RI(S(\Lambda))$). 
Based on Proposition \ref{CpxofGps}, however, there is a way to make a developable complex of groups related to arbitrary $RI(S(\Lambda))$ (to do this, the definition of complexes of join groups must be changed slightly, but we will not do this here).
Let $C(RI(S(\Lambda)))$ be a cone on $RI(S(\Lambda))$ such that the assigned groups of the added simplices is the trivial group. 
Obviously, $|C(RI(S(\Lambda)))|$ is contractible and thus homotopy equivalent to $\mathcal{R}^d(S(\Lambda))$. Since $\pi_1(\mathcal{R}(S(\Lambda)))$ is isomorphic to $A(\Lambda)$ and $\pi_1(C(RI(S(\Lambda))))$ is isomorphic to $A(\Lambda)$, $C(RI(S(\Lambda)))$ can be considered as a developable complex of groups and the development of $C(RI(S(\Lambda)))$ is obtained from $I(X(\Lambda))$ by adding a cone on each $R_x$.

%
%  4.2 Structure of $RI(D_2(\Gamma))$
%

\subsection{(Reduced) Intersection Complex for $PB_2(\Gamma)$}\label{4.2}
In this subsection, we will see the structures of the reduced intersection complex and the intersection complex for a graph 2-braid group. 
Recall that for a connected simplicial graph $\Gamma$ (as we assumed in the end of Section \ref{2.4}), $D_2(\Gamma)$ denotes the discrete configuration space of 2 points on $\Gamma$ and $\overline{D_2(\Gamma)}$ denotes the universal cover of $D_2(\Gamma)$.
We will exclude the case that $\Gamma$ has no pair of disjoint cycles due to the following three equivalent statements:
\begin{itemize}
\item
$\Gamma$ has no pair of disjoint cycles.
\item
$D_2(\Gamma)$ has no standard product subcomplexes, i.e. $RI(D_2(\Gamma))$ is empty.
\item 
$B_2(\Gamma)$ (and thus $PB_2(\Gamma)$) is a finitely generated free group (Theorem \ref{scrg}). 
\end{itemize}
Moreover, we exclude the case that $\Gamma$ has leaves or redundant vertices (vertices of valency 2 which is not contained in any boundary cycles).

The action of $S_2$ on $D_2(\Gamma)$ induces the action of $S_2$ on $RI(D_2(\Gamma))$ by semi-isomorphisms. 
For a vertex $\mathbf{u}\in RI(D_2(\Gamma))$ labelled by $\Gamma_1\times\Gamma_2$, the \textit{switched vertex} of $\mathbf{u}$ is the vertex in $RI(D_2(\Gamma))$ labelled by $\Gamma_2\times\Gamma_1$ and denoted by $\mathbf{u}^s$. 
Suppose that a component $C$ of $RI(D_2(\Gamma))$ contains both a vertex $\mathbf{u}$ and its switched vertex $\mathbf{u}^s$. 
Let $\mathbf{w}$ be another vertex of $C$ and $p$ a path in $C^{(1)}$ joining $\mathbf{u}$ to $\mathbf{w}$. Then the switched vertices of $p$ induce the path $p^s$ in $C^{(1)}$ from $\mathbf{u}^s$ to $\mathbf{w}^s$.
It means that $\mathbf{w}^s$ is contained in $C$ so that for any vertex in $C$, its switched vertex is also contained in $C$; this kind of component of $RI(D_2(\Gamma))$ is called an \textit{M-component}. 
If $C$ is not an M-component, then there exists a component $C^s$ whose vertex set is the collection of $\mathbf{u}^s$'s for all vertices $\mathbf{u}\in C$; in this case, $C$ and $C^s$ are called \textit{S-components} and $C^s$ is the \textit{switched component} of $C$. For instance, $RI(D_2(\mathcal{O}_{3}))$ in Figure \ref{Ex1} consists of one M-component and $RI(D_2(\mathcal{O}_V))$ in Figure \ref{ExV} consists of six S-components.

When $\Gamma$ is a simplest cactus, every standard product subcomplex of $D_2(\Gamma)$ (thus the label of a simplex of $RI(D_2(\Gamma))$ or $I(\overline{D_2(\Gamma)})$) can be represented by the product of two disjoint subcollections of $\mathcal{C}$ (see the paragraph below Lemma \ref{IntofMax}). 
\textbf{From now on, $\Gamma$ is assumed to be a simplest cactus in addition to the assumption in the first paragraph of this subsection. Specific examples of simplest cacti are denoted by $\mathcal{O}$ with indices.}

\vspace{1mm}

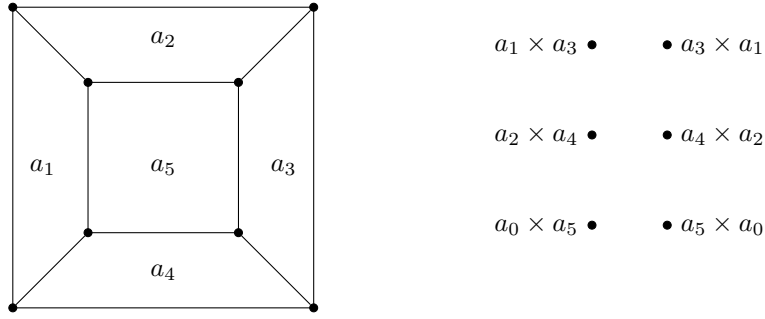
\begin{figure}[h]
\begin{tikzpicture}
\tikzstyle{every node}=[draw,circle,fill=black,minimum size=3pt,inner sep=0pt]
  \node[anchor=center] (a1) at (-0.5,4){};
  \node[anchor=center] (a2) at (3.5,4){};  
  \node[anchor=center] (a3) at (3.5,0)[label={[shift={(-0.4,1.6)}]:$a_3$}]{};  
  \node[anchor=center] (a4) at (-0.5,0)[label={[shift={(0.4,1.6)}]:$a_1$}]{};
  \node[anchor=center] (b1) at (0.5,3)[label={[shift={(1,0.3)}]:$a_2$}]{};  
  \node[anchor=center] (b2) at (2.5,3){};  
  \node[anchor=center] (b3) at (2.5,1)[label={[shift={(-1,0.6)}]:$a_5$}]{};
  \node[anchor=center] (b4) at (0.5,1)[label={[shift={(1,-0.8)}]:$a_4$}]{};  
  \draw (a1.center) -- (a2.center) -- (a3.center) -- (a4.center) -- cycle;  
  \draw (b1.center) -- (b2.center) -- (b3.center) -- (b4.center) -- cycle;  
  \foreach \from/\to in {a1/b1,a2/b2,a3/b3,a4/b4}
  \draw (\from) -- (\to);
  
  \draw (7.2,3.5) node (p1) [label={[label distance=0.1cm]left:$a_1\times a_3$}]{};
  \draw (8.2,3.5) node (p2) [label={[label distance=0.1cm]right:$a_3\times a_1$}]{};
  \draw (7.2,2.3) node (p3) [label={[label distance=0.1cm]left:$a_2\times a_4$}]{};
  \draw (8.2,2.3) node (p4) [label={[label distance=0.1cm]right:$a_4\times a_2$}]{};
  \draw (7.2,1.1) node (p5) [label={[label distance=0.1cm]left:$a_0\times a_5$}]{};
  \draw (8.2,1.1) node (p6) [label={[label distance=0.1cm]right:$a_5\times a_0$}]{};

\end{tikzpicture}
\caption{The graph $\mathcal{O}_V$ which is not a cactus, and $RI(D_2(\mathcal{O}_V))$. Each boundary cycle $a_i$ is considered as a subgraph.}
\label{ExV}
\end{figure}

The collection of labels of all simplices of $RI(D_2(\Gamma))$ is a poset under the following inclusion relation: $\mathcal{C}_1\times \mathcal{C}'_1\subset \mathcal{C}_2\times \mathcal{C}'_2$ if and only if $\mathcal{C}_1\subset \mathcal{C}_2$ and $\mathcal{C}'_1\subset \mathcal{C}'_2$ where $\mathcal{C}_i\times \mathcal{C}'_i$ is the label of a simplex of $RI(D_2(\Gamma))$ for $i=1,2$ (in particular, $\mathcal{C}_i$ and $\mathcal{C}'_i$ are disjoint subcollections of $\mathcal{C}$).
Using this inclusion relation, the existence of an M-component in $RI(D_2(\Gamma))$ can be detected from $\Gamma$.  

\begin{proposition}\label{Mcomponent}
If there exist three disjoint boundary cycles in $\Gamma$ such that any two of three are joined by a path in $\Gamma$ avoiding the other one, then $RI(D_2(\Gamma))$ is connected, i.e. $RI(D_2(\Gamma))$ consists of one M-component.
\end{proposition}
\begin{proof}
Let $a_1$, $a_2$, $a_3$ be three disjoint boundary cycles in $\Gamma$ satisfying the assumption. 
Let $\mathbf{u}$, $\mathbf{u}_1$ and $\mathbf{u}_2$ be vertices in $RI(D_2(\Gamma))$ whose labels contain $\{a_1,a_3\}\times\{a_2\}$, $\{a_3\}\times\{a_1,a_2\}$ and $\{a_2,a_3\}\times\{a_1\}$, respectively.
Then there is a loop $l=(\mathbf{u},\mathbf{u}_1,\mathbf{u}_2,\mathbf{u}^s,\mathbf{u}^s_1,\mathbf{u}^s_2,\mathbf{u})$ in $RI(D_2(\Gamma))$; for instance, $\mathbf{u}$ and $\mathbf{u}_1$ are joined by an edge whose label contains $\{a_3\}\times\{a_2\}$.
This implies that the component of $RI(D_2(\Gamma))$ containing $\mathbf{u}$ is an M-component.

Assume that there is a component $C$ of $RI(D_2(\Gamma))$ except the component containing $\mathbf{u}$ such that $\{b_j\}_{j\in J}\subset\mathcal{C}$ is the collection of boundary cycles appeared in the labels of vertices in $C$. Lemma \ref{DisjointSubgraphs} implies that each $b_j$ must intersect $a_1$, $a_2$ and $a_3$ which is a contradiction since $\Gamma$ is a simplest cactus. Therefore, $RI(D_2(\Gamma))$ consists of one M-component.
\end{proof}

Let $\Gamma$ be a simplest cactus. Under the assumption in the first paragraph of this subsection, $\Gamma$ is the union of a tree $\mathcal{K}_\Gamma$, which will be called the \textit{spine} of $\Gamma$, and 3-cycles attached to some vertices of $\mathcal{K}_\Gamma$ such that if $v\in\mathcal{K}_\Gamma$ has valency $\leq 2$ in $\mathcal{K}_\Gamma$, then there must exist a 3-cycle attached to $v$. 
There is a canonical quotient map $\Psi:\Gamma\rightarrow \mathcal{K}_\Gamma$ collapsing each 3-cycle to a vertex.
For a maximal product subcomplex $\Gamma_i\times\Gamma'_i\subset D_2(\Gamma)$, then $V_\Gamma$ is the disjoint union of $V_\Gamma\cap\Psi(\Gamma_i)$ and $V_\Gamma\cap\Psi(\Gamma'_i)$.

The existence of three boundary cycles $a_1$, $a_2$, $a_3$ in $\Gamma$ satisfying the assumption in Proposition \ref{Mcomponent} is equivalent to the existence of a star tree with central vertex of valency $n\geq 3$ in $\mathcal{K}_\Gamma$. A simplest cactus $\Gamma$ is said to be of $\it{type}$-$S$ if $\mathcal{K}_\Gamma$ is homeomorphic to a line segment, and of $\it{type}$-$M$ otherwise. 

\begin{proposition}\label{Scomponent}
For a simplest cactus $\Gamma$ of type-S, $RI(D_2(\Gamma))$ consists of two semi-isomorphic components whose underlying complexes are simplices. 
\end{proposition}
\begin{proof}
Since $\mathcal{K}_\Gamma$ is isometric to a line segment, $\Gamma$ is the graph which has $n$ boundary cycles $a_1,\cdots,a_n$ and $n-1$ edges such that $a_i$ and $a_{i+1}$ are joined by a unique edge for $i=1,\cdots,n-1$.
Then the vertices in $RI(D_2(\Gamma))$ are labelled by $\{a_1,\cdots,a_i\}\times\{a_{i+1},\cdots,a_n\}$ or $\{a_{i+1},\cdots,a_n\}\times\{a_1,\cdots,a_i\}$ for $i=1,\cdots,n-1$.
Therefore, $RI(D_2(\Gamma))$ consists of two $(n-1)$-simplices whose labels are $\{a_1\}\times\{a_n\}$ and $\{a_n\}\times\{a_1\}$, respectively.
\end{proof}

Combining Proposition \ref{Mcomponent} and Proposition \ref{Scomponent}, $RI(D_2(\Gamma))$ consists of one component if $\Gamma$ is of type-M, and two semi-isomorphic components if $\Gamma$ is of type-S. 
From now on, for convenience, whenever a simplest cactus $\Gamma$ is of type-M or type-S, let $RI(D_2(\Gamma))$ denote a component of $RI(D_2(\Gamma))$.
If $\Gamma$ is of type-S, Lemma \ref{IntofMax} implies that $\mathcal{R}(D_2(\Gamma))$ consists of two isometric components, and thus, let $\mathcal{R}(D_2(\Gamma))$ denote one of two components as in the previous sentence; equivalently, $\mathcal{R}(D_2(\Gamma))$ is isometric to the union of maximal product subcomplexes of $UD_2(\Gamma)$.

By the paragraphs above and below Lemma \ref{IntofMax}, there is a one-to-one correspondence between maximal product subcomplexes and oriented edges of $\mathcal{K}_\Gamma$. More precisely, a maximal product subcomplex $M_i=\Gamma_i\times\Gamma'_i$ corresponds to a separating open edge $e_i$ in $\mathcal{K}_\Gamma$ with the orientation of $e_i$ indicating the first coordinate of $M_i$. Then we can say that $\Gamma_i$ ($\Gamma_i$, resp.) is the union of $\Psi(\Gamma_i)$ ($\Psi(\Gamma'_i)$, resp.) and the cycles in $\Gamma$ intersecting $\Psi(\Gamma_i)$.
Recall that a product subcomplex of $D_2(\mathcal{K}_\Gamma)$ is defined as a product of two disjoint subgraphs of $\mathcal{K}_\Gamma$, possibly, vertices (see Definition \ref{PSofD2}). Using the fact that $D_2(\mathcal{K}_\Gamma)$ is the union of `maximal' product subcomplexes, product subcomplexes, which are maximal under the set inclusion, the homotopy type of $|RI(D_2(\Gamma))|$ can be deduced from $\mathcal{K}_\Gamma$. In order to do this, we first need the following strong version of Lemma \ref{IntofMax} for a simplest cactus $\Gamma$.

\begin{lemma}\label{IntisS}
Let $\Gamma$ be a simplest cactus. If there are finitely many maximal product subcomplexes $M_i=\Gamma_i\times\Gamma'_i$ whose intersection is non-empty, then the intersection is exactly a standard product subcomplex.
\end{lemma}
\begin{proof}
By the paragraph above this lemma, each $\Gamma_i\times\Gamma'_i$ corresponds to an edge $e_i$ with an orientation.
Then it can easily be seen that $M_i \cap M_j$ is non-empty if and only if the orientations of $e_i$ and $e_j$ are coherent, i.e. $\Gamma_i\subset \Gamma_j$ and $\Gamma'_j\subset\Gamma'_i$, or $\Gamma_j\subset \Gamma_i$ and $\Gamma'_i\subset\Gamma'_j$. In particular, $M_i\cap M_j$ is a standard product subcomplex which is either $\Gamma_i\times\Gamma'_j$ or $\Gamma_j\times\Gamma'_i$.
Using this fact, we can see that $\cap_i M_i$ is non-empty if and only if $e_i$'s are coherent if and only if $\cap_i M_i$ is a standard product subcomplex.
\end{proof}

\begin{proposition}\label{Cactus1}
Let $\Gamma$ be a simplest cactus and $\mathcal{K}_\Gamma$ the spine of $\Gamma$. Then $D_2(\mathcal{K}_\Gamma)$, $|RI(D_2(\Gamma))|$ and $\mathcal{R}^d(D_2(\Gamma))$ are homotopy equivalent.
In particular, $|RI(D_2(\Gamma))|$ is non-simply connected if and only if $\Gamma$ is of type-M.
\end{proposition}
\begin{proof}
By the paragraph above Lemma \ref{IntisS}, there is a ono-to-one correspondence between the collection of maximal product subcomplexes of $D_2(\Gamma)$ and the collection of `maximal' product subcomplexes of $D_2(\mathcal{K}_\Gamma)$.
Suppose that the intersection $K$ of maximal product subcomplexes $M_i$ of $D_2(\Gamma)$ is non-empty. By Lemma \ref{IntisS}, $K$ is exactly a standard product subcomplex and thus the intersection of the `maximal' product subcomplexes of $D_2(\mathcal{K}_\Gamma)$ corresponding to $M_i$'s is a product subcomplex.
Conversely, if the intersection $P$ of the `maximal' product subcomplexes $P_i$'s of $D_2(\mathcal{K}_\Gamma)$ is non-empty, by Lemma \ref{IntisS}, the intersection of the maximal product subcomplexes of $D_2(\Gamma)$ corresponding to $P_i$'s is a standard product subcomplex and thus $P$ is a product subcomplex.

Then a component of $D_2(\mathcal{K}_\Gamma)$ is homotopy equivalent to $|RI(D_2(\Gamma))|$; vertices in $|RI(D_2(\Gamma))|$ are blowed up to `maximal' product subcomplexes of $D_2(\mathcal{K}_\Gamma)$ and the intersection of a finite number of `maximal' product subcomplexes of $D_2(\mathcal{K}_\Gamma)$ is non-empty if and only if it is a product subcomplex if and only if the corresponding vertices in $|RI(D_2(\Gamma))|$ span a simplex. 
Therefore, $|RI(D_2(\Gamma))|$ is homotopy equivalent to $D_2(\mathcal{K}_\Gamma)$. Similarly, we can show that $|RI(D_2(\Gamma))|$ is homotopy equivalent to $\mathcal{R}^d(D_2(\Gamma))$ using Lemma \ref{IntofMax}.

Note that $D_2(\mathcal{K}_\Gamma)$ is connected and non-simply connected if and only if $\mathcal{K}_\Gamma$ is not homeomorphic to a line segment. And $D_2(\mathcal{K}_\Gamma)$ consists of two isometric components if and only if $\mathcal{K}_\Gamma$ is homeomorphic to a line segment. Therefore, the last statement in the proposition holds.
\end{proof}

In the previous subsection, $RI(S(\Lambda))$ is a complex of join groups whose assigned groups are considered as subgroups of $A(\Lambda)$ since $S(\Lambda)$ has a unique vertex. 
Contrast to $S(\Lambda)$, for any graph $\Gamma$, $D_2(\Gamma)$ has no canonical vertex so that the assigned groups may not be represented by subgroups of $PB_2(\Gamma)$. Instead, if $\Gamma$ is a simplest cactus, then we can know that $RI(D_2(\Gamma))$ always becomes a developable complex of groups and the fundamental group of $RI(D_2(\Gamma))$ (considered as a complex of groups) is isomorphic to $\pi_1(\mathcal{R}(D_2(\Gamma)))$ which is, for convenience, denoted by $SPB_2(\Gamma)$.

\begin{theorem}\label{pi_1-injective}
Let $\Gamma$ be a simplest cactus. Then $RI(D_2(\Gamma))$ can be considered as the developable complex of groups decomposition of $SPB_2(\Gamma)$. 
\end{theorem}
\begin{proof}
Since $\Gamma$ is a simplest cactus, every standard product subcomplex of $D_2(\Gamma)$ is the product of planar graphs.
By Proposition \ref{Cactus1}, moreover, $|RI(D_2(\Gamma))|$ is homotopy equivalent to $\mathcal{R}^d(D_2(\Gamma))$. 
By Proposition \ref{CpxofGps}, therefore, the theorem holds.
\end{proof}

\begin{corollary}\label{Development}
Let $\Gamma$ be a simplest cactus. Then a lift $I_0(\overline{D_2(\Gamma)})$ of $RI(D_2(\Gamma))$ in $I(\overline{D_2(\Gamma)})$ is the development of $RI(D_2(\Gamma))$. In particular, $|I_0((\overline{D_2(\Gamma)}))|$ is simply connected.
\end{corollary}
\begin{proof}
Theorem \ref{scrg} implies that $B_2(\Gamma)$ and $PB_2(\Gamma)$ admit the group presentations
$\langle \{g_i\}_{i\in I},\{h_j\}_{j\in J} \mid \{R_z\}_{z\in Z}\rangle$ and $\langle \{g'_{i'}\}_{i'\in I'},\{h'_{j'}\}_{j'\in J'} \mid \{R'_{z},R''_{z}\}_{z\in Z}\rangle$, respectively, where 
\begin{enumerate}
\item $R_z$ is the commutator corresponding to a pair of disjoint boundary cycles in $\Gamma$ and $R'_z$, $R''_z$ are the ordered versons of $R_z$, and
\item $\{g_i\}_{i\in I}$ is the set of generators of $B_2(\Gamma)$ used to present $R_z$'s and $\{g'_{i'}\}_{i'\in I'}$ is the set of generators of $PB_2(\Gamma)$ used to present $R'_z$'s and $R''_z$'s.
\end{enumerate}

Suppose that $\Gamma$ is a simplest cactus of type-M and let $f_D:D_2(\Gamma)\rightarrow M_2(\Gamma)$ be the homotopy equivalence in the sketch of proof of Theorem \ref{scrg}. Since $\mathcal{R}(D_2(\Gamma))$ is the union of all maximal product subcomplexes of $D_2(\Gamma)$, the image of $\mathcal{R}(D_2(\Gamma))$ under $f_D$ is the union of all 2-cubes in $M_2(\Gamma)$. It means that $\langle \{g'_{i'}\}_{i'\in I'}\mid \{R'_{z},R''_{z}\}_{z\in Z}\rangle$ is the quotient of $SPB_2(\Gamma)$.
But the embedding $\iota:\mathcal{R}(D_2(\Gamma))\hookrightarrow D_2(\Gamma)$ is a local isometry since if $(x,y)$ is a vertex in $\mathcal{R}(D_2(\Gamma))$, then the link of $(x,y)$ in $\mathcal{R}(D_2(\Gamma))$ is a full subcomplex of the link of $(x,y)$ in $D_2(\Gamma)$; it can easily be seen from the fact that $x\in\Gamma_i$ and $y\in\Gamma'_i$ for some maximal product subcomplex $\Gamma_i\times\Gamma'_i$.
Thus, $SPB_2(\Gamma)$ is indeed isomorphic to $\langle \{g'_{i'}\}_{i'\in I'}\mid \{R'_{z},R''_{z}\}_{z\in Z}\rangle$.
This means that the homomorphism $\iota_*:\pi_1(\mathcal{R}(D_2(\Gamma)))\rightarrow PB_2(\Gamma)$ induced from $\iota$ is injective, and thus, by the last part of Proposition \ref{CpxofGps}, the theorem holds.

Suppose that $\Gamma$ is a simplest cactus of type-S and let $f_U:UD_2(\Gamma)\rightarrow UM_2(\Gamma)$ be the homotopy equivalence in the sketch of proof of Theorem \ref{scrg}.
By the paragraph below Proposition \ref{Scomponent}, $\mathcal{R}(D_2(\Gamma))$ is equal to $\mathcal{R}(UD_2(\Gamma))$, and the image of $\mathcal{R}(UD_2(\Gamma))$ under $f_u$ is the union of all 2-cubes in $UM_2(\Gamma)$.
As in the previous paragraph, $\mathcal{R}(UD_2(\Gamma))$ is also locally convex in $UD_2(\Gamma)$, and thus, $SPB_2(\Gamma)$ is isomorphic to the subgroup $\langle \{g_{i}\}_{i\in I}\mid \{\mathcal{R}_z\}_{z\in Z}\rangle$ of $B_2(\Gamma)$. This means that the homomorphism $\pi_1(\mathcal{R}(UD_2(\Gamma)))\rightarrow B_2(\Gamma)$ induced from $\mathcal{R}(UD_2(\Gamma))\hookrightarrow UD_2(\Gamma)$ is injective, and thus, by the last part of Proposition \ref{CpxofGps}, the theorem holds (note that $I(\overline{UD_2(\Gamma)})=I(\overline{D_2(\Gamma)})$).
\end{proof}

\begin{figure}[h]
\begin{tikzpicture}
\tikzstyle{every node}=[draw,circle,fill=black,minimum size=3pt,inner sep=0pt]
  \node[anchor=center] (x) at (0,0){};
  \node[anchor=center] (a1) at (0,1){};
  \node[anchor=center] (a2) at (0.5,1.73){};  
  \node[anchor=center] (a3) at (-0.5,1.73){};  
  \node[anchor=center] (b1) at (1,0){};
  \node[anchor=center] (b2) at (1.73,0.5){};  
  \node[anchor=center] (b3) at (1.73,-0.5){};  
  \node[anchor=center] (c1) at (0,-1){};
  \node[anchor=center] (c2) at (0.5,-1.73){};  
  \node[anchor=center] (c3) at (-0.5,-1.73){};  
  \node[anchor=center] (d1) at (-1,0){};
  \node[anchor=center] (d2) at (-1.73,0.5){};  
  \node[anchor=center] (d3) at (-1.73,-0.5){};  
  
  \draw (a1.center) -- (a2.center) -- (a3.center) -- cycle;  
  \draw (b1.center) -- (b2.center) -- (b3.center) -- cycle;
  \draw (c1.center) -- (c2.center) -- (c3.center) -- cycle;
  \draw (d1.center) -- (d2.center) -- (d3.center) -- cycle;  
  \foreach \from/\to in {x/a1,x/b1,x/c1,x/d1}
  \draw (\from) -- (\to);
  
  \begin{scope}[xshift=5.5cm]
  \node[anchor=center] (x1) at (0,0){};
  \node[anchor=center] (y1) at (1,0){};
  \node[anchor=center] (a1) at (-0.5,0.73){};
  \node[anchor=center] (a2) at (-1.5,0.73){};  
  \node[anchor=center] (a3) at (-0.5,1.73){};  
  \node[anchor=center] (b1) at (-0.5,-0.73){};
  \node[anchor=center] (b2) at (-1.5,-0.73){};  
  \node[anchor=center] (b3) at (-0.5,-1.73){};  
  \node[anchor=center] (c1) at (1.5,0.73){};
  \node[anchor=center] (c2) at (2.5,0.73){};  
  \node[anchor=center] (c3) at (1.5,1.73){};  
  \node[anchor=center] (d1) at (1.5,-0.73){};
  \node[anchor=center] (d2) at (1.5,-1.73){};  
  \node[anchor=center] (d3) at (2.5,-0.73){};  
\end{scope}
  \draw (a1.center) -- (a2.center) -- (a3.center) -- cycle;  
  \draw (b1.center) -- (b2.center) -- (b3.center) -- cycle;
  \draw (c1.center) -- (c2.center) -- (c3.center) -- cycle;
  \draw (d1.center) -- (d2.center) -- (d3.center) -- cycle;  
  \foreach \from/\to in {x1/a1,x1/b1,y1/c1,y1/d1,x1/y1}
  \draw (\from) -- (\to);
  
\end{tikzpicture}
\caption{$\mathcal{O}_{4}$ and $\mathcal{O}'_{4}$.}
\label{Ex2}
\end{figure}
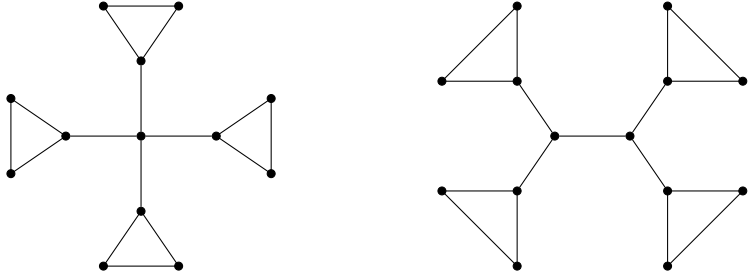

The relation between $I(\overline{D_2(\Gamma)})$ and $RI(D_2(\Gamma))$ is a little bit different from the relation between $I(X(\Lambda))$ and $RI(S(\Lambda))$.
For instance, when $\Gamma$ is a simplest cactus of type-M, $|I(\overline{D_2(\Gamma)})|$ can not be considered as the union of copies of $|RI(D_2(\Gamma))|$ since $|I(\overline{D_2(\Gamma)})|$ is simply connected but $|RI(D_2(\Gamma))|$ is not.

\begin{Ex}\label{T_4Ex}
Let $\mathcal{O}_{4}$ be the left graph in Figure \ref{Ex2}; a simplest cactus of type-M which is the union of four boundary cycles $a_1,\cdots,a_4$ and a star tree with central vertex of valency 4. The spine of $\mathcal{O}_4$ is a star tree with central vertex of valency 4.
The figure of $RI(D_2(\mathcal{O}_{4}))$ is given in Figure \ref{Ex2'}. In this case, $RI(D_2(\mathcal{O}_4))$ is the graph of groups decomposition of $SPB_2(\mathcal{O}_4)$.
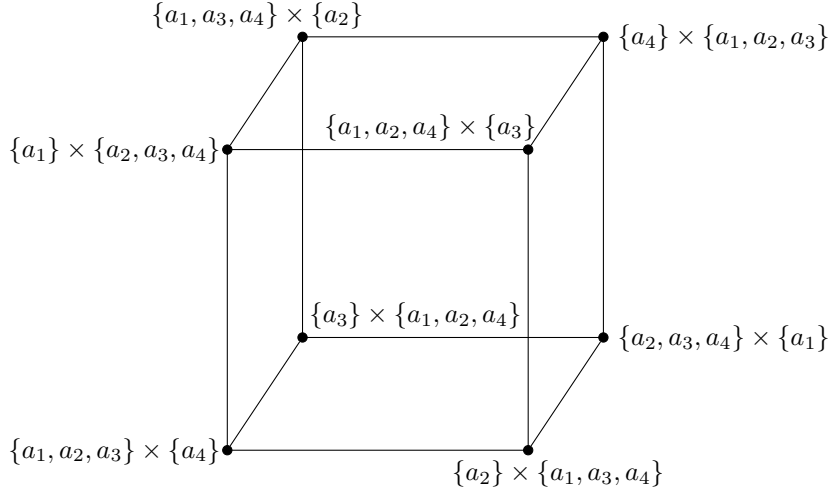
\begin{figure}[h]
\begin{tikzpicture}
\tikzstyle{every node}=[draw,circle,fill=black,minimum size=3.5pt,inner sep=0pt]
  
\draw (0,0) node (p1) [label={[shift={(-1.5,-1.5)}]:$\{a_1,a_2,a_3\}\times\{a_4\}$}]{};
\draw (4,0) node (p2) [label={[shift={(0.4,-1.8)}]:$\{a_2\}\times\{a_1,a_3,a_4\}$}]{};
\draw (4,4) node (p3) [label={[shift={(-1.3,-1.2)}]:$\{a_1,a_2,a_4\}\times\{a_3\}$}]{};
\draw (0,4) node (p4) [label={[shift={(-1.5,-1.5)}]:$\{a_1\}\times\{a_2,a_3,a_4\}$}]{};
\draw (1,1.5) node (p5) [label={[shift={(1.5,-1.2)}]:$\{a_3\}\times\{a_1,a_2,a_4\}$}]{};
\draw (5,1.5) node (p6) [label={[shift={(1.6,-1.5)}]:$\{a_2,a_3,a_4\}\times\{a_1\}$}]{};
\draw (5,5.5) node (p7) [label={[shift={(1.6,-1.5)}]:$\{a_4\}\times\{a_1,a_2,a_3\}$}]{};
\draw (1,5.5) node (p8) [label={[shift={(-0.6,-1.2)}]:$\{a_1,a_3,a_4\}\times\{a_2\}$}]{};

\draw (p1.center) -- (p2.center) -- (p3.center) -- (p4.center) -- cycle;  
\draw (p5.center) -- (p6.center) -- (p7.center) -- (p8.center) -- cycle;

  \foreach \from/\to in {p1/p5,p2/p6,p3/p7,p4/p8}
  \draw (\from) -- (\to);
\end{tikzpicture}
\caption{$RI(D_2(\mathcal{O}_{4}))$.}
\label{Ex2'}
\end{figure}
\end{Ex}

\begin{Ex}\label{T'_4Ex}
Let $\mathcal{O}'_{4}$ be the right graph of Figure \ref{Ex2}; a simplest cactus of type-M which has four boundary cycles $a_1,\cdots,a_4$ and six edges. The spine of $\mathcal{O}'_4$ is obtained from $\mathcal{O}'_{4}$ by removing vertices of valency 2 and (open) edges attached to these vertices. The right of Figure \ref{Ex3'} is $RI(D_2(\mathcal{O}'_4))$. 
In this case, $RI(D_2(\mathcal{O}'_4))$ is the developable complex of groups decomposition of $SPB_2(\mathcal{O}'_4)$.

\begin{figure}[h]
\begin{tikzpicture}
\tikzstyle{every node}=[draw,circle,fill=black,minimum size=3.5pt,inner sep=0pt]
\tikzstyle{point}=[circle,thick,draw=black,fill=black,inner sep=0pt,minimum width=4pt,minimum height=4pt]
	
\draw [gray, -triangle 60  ] (-2.5,1.4) -- (-0.1,0.1) node (x1) [minimum size=0pt]{};
\node (x)[point,label={[shift={(-3.2,0.2)}]:$\{a_1,a_2\}\times\{a_3,a_4\}$}] at (0,0) {};

\node (a)[point,label={[shift={(0.9,-1.2)}]:$\{a_1\}\times\{a_2,a_3,a_4\}$}] at (-0.7,1.5) {};
\node (b)[point,label={[shift={(-1.6,-1.5)}]:$\{a_1,a_2,a_4\}\times\{a_3\}$}] at (-3.1,0) {};
\node (c)[point,label={[shift={(1.55,-1.75)}]:$\{a_2\}\times\{a_1,a_3,a_4\}$}] at (0.7,-1.5) {};
\node (d)[point,label={[shift={(1.6,-1.5)}]:$\{a_1,a_2,a_3\}\times\{a_4\}$}] at (3.1,0) {};

\begin{scope}[yshift=-4cm]
    \node (y)[point] at (0,0) {};

    \node (e)[point] at (-0.7,1.5) {};
    \node (f)[point] at (-3.1,0) {};
    \node (g)[point] at (0.7,-1.5) {};
    \node (h)[point] at (3.1,0) {};
\end{scope}

    \draw[pattern=dots] (a.center) -- (x.center) -- (b.center) -- cycle;
    \draw[pattern=mydots] (b.center) -- (x.center) -- (c.center) -- cycle;
    \draw[pattern=mynewdots] (c.center) -- (x.center) -- (d.center) -- cycle;
    \draw[pattern=verynewdots] (d.center) -- (x.center) -- (a.center) -- cycle;

	\draw [red] plot [smooth] coordinates {(-0.7,1.5) (-1.1,-0.5) (-0.7,-2.5)};
	\draw [blue] plot [smooth] coordinates {(-3.1,0) (-3.8,-2) (-3.1,-4)};
	\draw [red] plot [smooth] coordinates {(0.7,-1.5) (1.1,-3.5) (0.7,-5.5)};
	\draw [blue] plot [smooth] coordinates {(3.1,0) (3.8,-2) (3.1,-4)};

    \draw[pattern=mynewdots] (e.center) -- (y.center) -- (f.center) -- cycle;
    \draw[pattern=verynewdots] (f.center) -- (y.center) -- (g.center) -- cycle;
    \draw[pattern=dots] (g.center) -- (y.center) -- (h.center) -- cycle;
    \draw[pattern=mydots] (h.center) -- (y.center) -- (e.center) -- cycle;
    
\end{tikzpicture}
\caption{$RI(D_2(\mathcal{O}'_{4}))$.}
\label{Ex3'}
\end{figure}
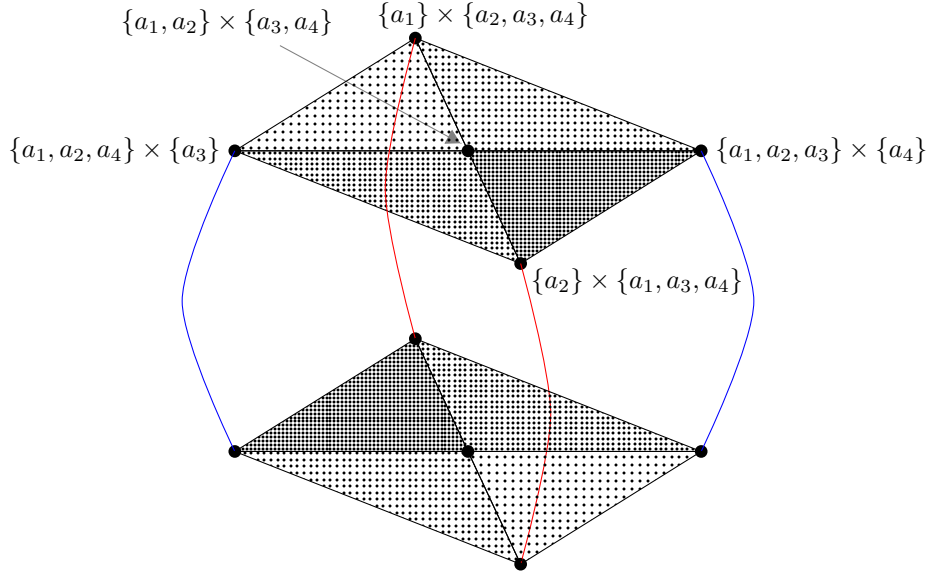
\end{Ex}

If $\Gamma$ is not a simplest cactus or even not a planar graph, then the situation becomes more complicated even when we determine whether $RI(D_2(\Gamma))$ contains an M-component. This is why, in this paper, we only deal with simplest cacti when we talk about graph 2-braid groups.

%%%%%%%%%%%%%%%%%%%%%%%%%%
%
%
%                CHAPTER 5.
%
%
%%%%%%%%%%%%%%%%%%%%%%%%%%

\section{Construction of quasi-isometries}\label{5}

Given a compact weakly special square complex $Y_i$ and its universal cover $\overline{Y_i}$ for $i=1,2$, each simplex of $I(\overline{Y_i})$ and $RI(Y_i)$ has two kinds of information: a label, the standard product subcomplex of $Y$ corresponding to a simplex (Definition \ref{IC} and \ref{RI}), and an assigned group (Definition \ref{CJ}).
When $Y_i$ is either a 2-dimensional Salvetti complex $S(\Lambda)$ or $D_2(\Gamma)$ for a simplest cactus $\Gamma$, these kinds of information give us an explicit way to describe the relation between $I(\overline{Y_i})$ and $RI(Y_i)$.   
Using this relation, in this section, we will see when a semi-isomorphism $I(\overline{Y_1})\rightarrow I(\overline{Y_2})$ induces a quasi-isometry $\overline{Y_1}\rightarrow\overline{Y_2}$.

%
% 5.1 Applications to 2-dimensional RAAGs
%

\subsection{Applications to 2-dimensional RAAGs}\label{5.1}
Two well-known facts about quasi-isometric rigidity of RAAGs are following: One is from the result in \cite{PW02},\cite{Hua(a)} that it suffices to consider RAAGs whose defining graphs are connected in order to classify all RAAGs up to quasi-isometry. More precisely, 

\begin{theorem}[\cite{PW02},\cite{Hua(a)}]
Suppose that $\Lambda$ and $\Lambda'$ are (possibly disconnected) simplicial graphs such that there is a $(\lambda,\varepsilon)$-quasi-isometry $\phi:X(\Lambda)\rightarrow X(\Lambda')$.
Then there exists $D=D(\lambda,\varepsilon)>0$ such that for any connected component $\Lambda_1$ of $\Lambda$ which is not a single vertex and a standard subcomplex $X_1$ of $X(\Lambda)$ whose defining graph is $\Lambda_1$, there is a unique connected component $\Lambda'_1$ of $\Lambda'$ and a standard subcomplex $X'_1$ of $X(\Lambda')$ whose defining graph is $\Lambda'_1$ such that $d_H(\phi(X_1),X'_1)<D$.
\end{theorem}

The other one is from Theorem \ref{FlattoFlat} (Corollary 1.4 in \cite{Hua(b)}) that if two RAAGs are quasi-isometric, then they have the same dimension. 
Hence, given a 2-dimensional RAAG $A(\Lambda)$, it suffices to only consider 2-dimensional RAAGs in order to find RAAGs quasi-isometric to $A(\Lambda)$.
\textbf{In this subsection, the defining graphs of RAAGs are thus assumed to be connected and triangle-free and contain at least one edge, as we assumed in Section \ref{4.1}.}

By Theorem \ref{QIimpliesIso2}, if $X(\Lambda)$ and $X(\Lambda')$ are quasi-isometric, then $I(X(\Lambda))$ and $I(X(\Lambda'))$ are semi-isomorphic. With some restrictions on defining graphs, a semi-isomorphism between $I(X(\Lambda))$ and $I(X(\Lambda'))$ induces a quasi-isometry between $X(\Lambda)$ and $X(\Lambda')$.
The first example is the class of RAAGs whose defining graphs are trees.
Let $P_4$ be a line segment with 4 vertices $a,b,c,d$ in this order. Then, 
$$A(P_4)=\langle a,b,c,d\ |\ [a,b]=[b,c]=[c,d]=1\rangle$$ 
and $RI(S(P_4))$ consists of two vertices and one edge such that the assigned groups of the vertices are $\langle b\rangle\times\langle a,c\rangle$ and $\langle b,d\rangle\times\langle c\rangle$, respectively, and the assigned group of the edge is $\langle b\rangle\times\langle c\rangle$. In particular, these vertices are type-2.
Since $|RI(S(P_4))|$ is contractible and 1-dimensional, by Lemma \ref{Decomposition} and Theorem \ref{GOG}, $I(X(P_4))$ is the Bass-Serre tree associated to $RI(S(P_4))$ and by Proposition \ref{ContractibleRI}, $|I(X(P_4))|$ is a locally infinite tree of infinite diameter.

\begin{lemma}\label{CoreofT}
Let $T$ be a tree of diameter $\geq 3$. Then $I(X(T))$ is semi-isomorphic to $I(X(P_4))$.
\end{lemma}
\begin{proof}
Since $T$ is a tree, any vertex in $RI(S(T))$ is labelled by $v\circ lk(v)$ for a vertex $v\in T$ whose valency is $\geq 2$; each vertex in $RI(S(T))$ corresponds to a vertex in $T$ whose valency is $\geq 2$. The vertices in $RI(S(T))$ labelled by $v\circ lk(v)$ and $w\circ lk(w)$ are joined by an edge if and only if $v$ and $w$ are adjacent in $T$. 
So, $|RI(S(T))|$ is isometric to the subgraph of $T$ obtained by removing leaves in $T$ and in particular, $|RI(S(T))|$ is contractible. 
Then $RI(S(T))$ is a complex of join groups such that:
\begin{enumerate}
\item $|RI(S(T))|$ is a finite tree and every vertex in $RI(S(T))$ is type-2.
\item 
The assigned group of each vertex is isomorphic to $\mathbb{Z}\times\mathbb{F}$ and the assigned group of each edge is isomorphic to $\mathbb{Z}\times\mathbb{Z}$.
\item
For two endpoints $\mathbf{u}_1$ and $\mathbf{u}_2$ of each edge $\bold{E}\subset RI(S(T))$, if $G_{\mathbf{u}_1}$ is pairwise isomorphic to $\mathbb{Z}\times\mathbb{F}_n$, then $G_{\mathbf{u}_2}$ is pairwise isomorphic to $\mathbb{F}_m\times\mathbb{Z}$ for $n,m>1$ and vice-versa. And $G_\bold{E}$ is contained in $G_{\mathbf{u}_i}$ under the pairwise free factor inclusion. 
\end{enumerate}

In the paragraph above this lemma, we showed that $|I(X(P_4))|$ is a locally infinite tree of infinite diameter. Due to the fact that $|RI(S(T))|$ is contractible and 1-dimensional, by the same reason, $|I(X(T))|$ is also a locally infinite tree of infinite diameter. 
Moreover, both $I(X(T))$ and $I(X(P_4))$ have a kind of bipartite structure; for a vertex $\mathbf{u}$ whose assigned group is pairwise isomorphic to $\mathbb{Z}\times\mathbb{F}_n$, the assigned group of each vertex adjacent to $\mathbf{u}$ is pairwise isomorphic to $\mathbb{F}_m\times\mathbb{Z}$ for $n,m>1$ and vice-versa.
Therefore, $I(X(P_4))$ is semi-isomorphic to $I(X(T))$.
\end{proof}

\begin{theorem}[$\mathit{cf}$.\cite{BN}]\label{Tree}
Let $T$ be a tree of diameter $\geq 3$. For a graph $\Lambda$, if $S(\Lambda)$ is 2-dimensional, then $I(X(\Lambda))$ is semi-isomorphic to $I(X(T))$ if and only if $\Lambda$ is a tree of diameter $\geq 3$.
\end{theorem}
\begin{proof}
Suppose that $\Lambda$ is a tree. By Lemma \ref{CoreofT}, if $\Lambda$ has diameter $\geq 3$, then $I(X(\Lambda))$ is semi-isomorphic to $I(X(T))$.
If $\Lambda$ is a tree of diameter $<3$, then $|I(X(\Lambda))|$ is a vertex so that $I(X(\Lambda))$ is not semi-isomorphic to $I(X(T))$.

Suppose that $\Lambda$ has an induced $n$-cycle ($n>3$).
If $n=4$, then $I(X(\Lambda))$ has a vertex whose assigned group is quasi-isometric to $\mathbb{F}\times\mathbb{F}$. Since the assigned group of every vertex in $I(X(T))$ is quasi-isometric to $\mathbb{Z}\times\mathbb{F}$, $I(X(\Lambda))$ cannot be semi-isomorphic to $I(X(T))$.
If $n\geq 5$, by Lemma \ref{Ncycle}, $|RI(\Lambda)|$ and thus $|I(X(\Lambda))|$ have either an induced $n$-cycle or a $k$-simplex for $k\geq 2$. In particular, $I(X(\Lambda))$ is not semi-isomorphic to $I(X(T))$. 
\end{proof}

In \cite{BN}, Behrstock and Neumann studied about the quasi-isometric rigidity in a class of graphs of groups. Their main result can be reformulated in our language as follows:
\begin{theorem}[\cite{BN}]\label{GraphofGroups}
Let $X_i$ be a graph of groups for $i=1,2$ such that
\begin{enumerate}
\item
The assigned group $G_{\mathbf{u}}$ of every vertex $\mathbf{u}$ is isomorphic to $\mathbb{Z}\times\mathbb{F}$, and the assigned group $G_\bold{E}$ of every edge $\bold{E}$ is isomorphic to $\mathbb{Z}\times\mathbb{Z}$.
\item(Flipping)
For two endpoints $\mathbf{u}_1$ and $\mathbf{u}_2$ of each edge $\bold{E}$, let $G_{\mathbf{u}_1}=\langle a_1\rangle\times\langle a_2,\cdots,a_n\rangle$, $G_{\mathbf{u}_2}=\langle b_1\rangle\times\langle b_2,\cdots,b_m\rangle$ for $n,m\geq 3$ and $G_\bold{E}=\langle x\rangle\times\langle y\rangle$. If the embedding $G_\bold{E}\hookrightarrow G_{\mathbf{u}_1}$ maps $x$ to $a_1$ and $y$ maps to $a_i$ for $2\leq i\leq n$, then $G_\bold{E}\hookrightarrow G_{\mathbf{u}_2}$ maps $x$ to $b_j$ for $2\leq j\leq m$ and $y$ maps to $b_1$ and vice-versa.
\end{enumerate}
Then, the fundamental groups of $X_1$ and $X_2$ (considered as graph of groups) are quasi-isometric.
\end{theorem}

If $T$ is a tree of diameter $\geq 3$, then $RI(S(T))$ is a graph of groups decomposition of $A(T)$ by Lemma \ref{Decomposition} and satisfies the conditions in Theorem \ref{GraphofGroups}. Therefore, if $T$ and $T'$ are trees of diameter $\geq 3$, not only $I(X(T))$ and $I(X(T'))$ are semi-isomorphic, but also $A(T)$ and $A(T')$ are quasi-isometric.

\begin{remark}
If $\Lambda$ is square-free, then $RI(S(\Lambda))$ is a graph of groups. However, if $|RI(S(\Lambda))|$ is not simply connected, by Proposition \ref{ContractibleRI} and Theorem \ref{GOG}, $I(X(\Lambda))$ is not the Bass-Serre tree associted to $RI(S(\Lambda))$. By Remark \ref{RmkofS}, moreover, $A(\Lambda)$ is not isomorphic to $\pi_1(RI(S(\Lambda)))$.
\end{remark}

The second example is the class of RAAGs whose outer automorphism groups are finite. 
By the results in \cite{Ser89} and \cite{Lau95}, the outer automorphism group $\Out(A(\Lambda))$ of $A(\Lambda)$ is generated by the following four types of elements with respect to the standard generating set of $A(\Lambda)$ which is the collection of vertices in $\Lambda$:
\begin{enumerate}
\item Automorphisms induced by graph automorphisms of $\Lambda$,
\item Inversions: $v_i\mapsto v^{-1}_i$ for a vertex $v_i\in\Lambda$ and fix all the other vertices,
\item Transvections: For two vertices $v,w\in\Lambda$, $w\mapsto wv$ if $lk(w)\subset st(v)$, and fix all the other vertices,
\item Partial conjugations: Suppose that for a vertex $v\in\Lambda$, $\Lambda-st(v)$ is disconnected and $C$ is one of its components. Then $w\mapsto vwv^{-1}$ for vertices $w\in C$ and fix all the other vertices.
\end{enumerate}
In particular, $\Out(A(\Lambda))$ is finite if it does not have transvections and partial conjugations, i.e. $\Lambda$ has no closed separating stars and there are no two distinct vertices $v,w\in\Lambda$ such that $lk(w)\subset st(v)$.

Recall that $\{R_x\}_{x\in X(\Lambda)}$ is the cover of $I(X(\Lambda))$ where $R_x$ is a copy of $RI(S(\Lambda))$ such that $\rho_{S(\Lambda)}(R_x)=RI(S(\Lambda))$.
Suppose that $\Out(A(\Lambda))$ has no transvections. Then the star of any vertex in $\Lambda$ is a maximal join subgraph of $\Lambda$ and there is an embedding $\tau:\Lambda\rightarrow |RI(S(\Lambda))|^{(1)}$ which sends each vertex $v\in\Lambda$ to a vertex in $RI(S(\Lambda))$ labelled by $st(v)$. 
It means that the image of $\Lambda$ in $|RI(S(\Lambda))|$ under $\tau$ is the subcomplex of $|RI(S(\Lambda))|$ spanned by vertices whose assigned groups are quasi-isometric to $\mathbb{Z}\times\mathbb{F}$.
For convenience, let $\tau(\Lambda)$ be the subcomplex of $RI(S(\Lambda))$ whose underlying complex is the image of $\Lambda$ under $\tau$ (note that $|\tau(\Lambda)|$ is isometric to $\Lambda$ and $\tau(\Lambda)$ may not be a reduced intersection complex).
Any vertex whose assigned group is quasi-isometric to $\mathbb{F}\times\mathbb{F}$ is adjacent to some vertices in $\tau(\Lambda)$.  
Let $\overline{\tau(\Lambda)}$ be the preimage of $\tau(\Lambda)$ in $I(X(\Lambda))$ under $\rho_{S(\Lambda)}$. 
Then $\overline{\tau(\Lambda)}$ is the subcomplex of $I(X(\Lambda))$ spanned by the vertices whose assigned groups are quasi-isometric to $\mathbb{Z}\times\mathbb{F}$ and $|\overline{\tau(\Lambda)}|$ is exactly the extension complex of $\Lambda$ defined in \cite{Hua(a)}. From the result of \cite{Hua(a)}, we obtain the following fact.

\begin{theorem}\label{FiniteOut}
Let $\Lambda$ and $\Lambda_*$ be triangle-free graphs such that $\Out(A(\Lambda))$ and $\Out(A(\Lambda_*))$ are finite.
Then $\Lambda$ and $\Lambda_*$ are isometric if and only if $RI(S(\Lambda))$ and $RI(S(\Lambda_*))$ are semi-isomorphic if and only if $I(X(\Lambda))$ and $I(X(\Lambda_*))$ are semi-isomorphic.
\end{theorem}
\begin{proof}
Obviously, if $\Lambda$ and $\Lambda_*$ are isometric, then $RI(S(\Lambda))$ and $RI(S(\Lambda_*))$ are semi-isomorphic, and $I(X(\Lambda))$ and $I(X(\Lambda_*))$ are semi-isomorphic.

Suppose that there is a semi-isomorphism $\Phi:RI(S(\Lambda))\rightarrow RI(S(\Lambda_*))$. 
Let $\tau:\Lambda\rightarrow |RI(S(\Lambda))|$ and $\tau_*:\Lambda_*\rightarrow |RI(S(\Lambda_*))|$ be the embeddings given as before with subcomplexes $\tau(\Lambda)\subset RI(S(\Lambda))$ and $\tau_*(\Lambda_*)\subset RI(S(\Lambda_*))$, respectively.
Then the image of $\tau(\Lambda)\subset RI(S(\Lambda))$ under $\Phi$ is $\tau_*(\Lambda_*)\subset RI(S(\Lambda_*))$ since the semi-isomorphism $\Phi$ preserves the quasi-isometric type of labels of vertices. Therefore, $\Lambda$ and $\Lambda_*$ are isometric.

Suppose that there is a semi-isomorphism $\Phi:I(X(\Lambda))\rightarrow I(X(\Lambda_*))$.
Then the restriction of $\Phi$ to $\overline{\tau(\Lambda)}\subset I(X(\Lambda))$ induces an isometry between $|\overline{\tau(\Lambda)}|$ and $|\overline{\tau_*(\Lambda_*)}|\subset |I(X(\Lambda_*))|$ since $\Phi$ preserves the quasi-isometric type of labels of vertices.
Theorem 1.1 in \cite{Hua(a)} asserts that under our assumption on $\Lambda$ and $\Lambda_*$, if the extension complexes of $\Lambda$ and $\Lambda_*$ are isometric, then $\Lambda$ and $\Lambda_*$ are isometric. Therefore, $\Lambda$ and $\Lambda_*$ are isometric. 
\end{proof}

\begin{corollary}\label{FiniteOut2}
Suppose that $\Lambda$ and $\Lambda_*$ are given as in Theorem \ref{FiniteOut}. 
Then $\Lambda$ and $\Lambda_*$ are isometric if and onyl if $A(\Lambda)$ and $A(\Lambda_*)$ are quasi-isometric.
\end{corollary}
\begin{proof}
The forward direction is obvious. 
If $\Lambda_1$ and $\Lambda_2$ are not isometric, by Theorem \ref{FiniteOut}, $I(X(\Lambda_1))$ and $I(X(\Lambda_2))$ are not semi-isomorphic. By Theorem \ref{QIimpliesIso2}, $X(\Lambda_1)$ and $X(\Lambda_2)$ are not quasi-isometric. 
\end{proof}

By using (reduced) intersection complexes, we can show that adding leaves to the defining graph may change the quasi-isometric type of the RAAG. 

\begin{lemma}\label{FiniteOutNoSV}
If $\Out(A(\Lambda))$ is finite, then every vertex in $RI(S(\Lambda))$ is type-1. In particular, $|I(X(\Lambda))|$ has no separating vertices.
\end{lemma}
\begin{proof}
Since there is no leaf in $\Lambda$, it is obvious that for any vertex $\mathbf{u}\in RI(S(\Lambda))$, $G_\mathbf{u}$ is contained in the subgroup of $A(\Lambda)$ generated by $G_{\mathbf{u}'}$ for all the vertices $\mathbf{u}'$ adjacent to $\mathbf{u}$.
The vertex in $RI(S(\Lambda))$ whose label is $st(v)$ for a vertex $v\in\Lambda$ is not a separating; this vertex is in $\tau(\Lambda)\subset RI(S(\Lambda))$ and $|\tau(\Lambda)|$ has no separating vertices. 
Let $\mathbf{u}$ be a vertex whose assigned group is quasi-isometric to $\mathbb{F}\times\mathbb{F}$. Let $\{u_i\}_{i\in I}\circ\{v_j\}_{j\in J}$ be the label of $\mathbf{u}$. 
Then $\mathbf{u}$ is adjacent to the vertices labelled by $st(u_i)$'s and $st(v_j)$'s. Moreover, these vertices span a complete bipartite graph in $|RI(S(\Lambda))|$ which is isometric to $\{u_i\}_{i\in I}\circ\{v_j\}_{j\in J}$. So, $\mathbf{u}$ is not a separating vertex. 
Therefore, every vertex in $RI(S(\Lambda))$ is type-1 and by Lemma \ref{Type2Vertex}, $I(X(\Lambda))$ has no separating vertices.
\end{proof}

\begin{proposition}\label{ValencyOne}
Let $\Lambda$ be a graph with finite $\Out(A(\Lambda))$ and let $\Lambda'$ be the graph obtained from $\Lambda$ by attaching an edge to a vertex $v\in\Lambda$.
Then $|I(X(\Lambda'))|$ is the union of copies of $|I(X(\Lambda))|$ such that the intersection of two copies is either a vertex or an empty set.
In particular, $I(X(\Lambda))$ and $I(X(\Lambda'))$ are not semi-isomorphic. 
\end{proposition}
\begin{proof}
Consider $\Lambda$ as a subgraph of $\Lambda'$ and let $M'_{\mathbf{v}}\subset S(\Lambda')$ be the maximal product subcomplex whose label is $st_{\Lambda'}(v)$. 
Then $X(\Lambda')$ is the union of copies of $X(\Lambda)$ and p-lifts of $M'_{\mathbf{v}}$; the intersection of any two copies of $X(\Lambda)$ is empty.
Thus, $|I(X(\Lambda'))|$ is the union of copies of $|I(X(\Lambda))|$ such that the intersection of two copies is either the vertex corresponding to a p-lift of $M'_{\mathbf{v}}$ or an empty set. 
Note that except the vertices in $I(X(\Lambda'))$ which correspond to p-lifts of $M'_{\mathbf{v}}$, the labels of vertices are the same with the labels of vertices in the copy of $|I(X(\Lambda))|$.
Since Lemma \ref{FiniteOutNoSV} implies that $I(X(\Lambda))$ has no separating vertices, the proposition holds. 
\end{proof}

In the above proposition, $|RI(S(\Lambda))|$ and $|RI(S(\Lambda'))|$ are isometric. Actually, there is a canonical semi-morphism $RI(S(\Lambda))\rightarrow RI(S(\Lambda'))$ but its inverse is not a semi-morphism.
This means that the semi-isomorphic type of $RI(S(\Lambda))$ tells us more than the isometric type of $|RI(S(\Lambda))|$.
There is a specific example that two reduced intersection complexes are not isomorphic but their underlying complexes are isometric; the same holds for intersection complexes. This is the proof of Proposition \ref{1.5'} with an explicit example.

\begin{proposition}[Proposition \ref{1.5'}]\label{1.5}
Let $\Lambda$ be a graph with finite $\Out(A(\Lambda))$ and let $\Lambda'$ be the graph obtained from the disjoint union of $\Lambda$ and a vertex $v'$ by attaching an edge between $v'$ and each vertex in $lk_{\Lambda}(v)$ for some vertex $v\in\Lambda$.
Then $|I(X(\Lambda))|$ and $|I(X(\Lambda'))|$ are isometric but $I(X(\Lambda))$ and $I(X(\Lambda'))$ are not semi-isomorphic.
\end{proposition}
\begin{proof}
Consider $\Lambda$ as a subgraph of $\Lambda'$. Then there is a canonical semi-morphism $RI(S(\Lambda))\rightarrow RI(S(\Lambda'))$ such that each simplex $\triangle\subset RI(S(\Lambda))$ maps to a simplex $\triangle'\subset RI(S(\Lambda'))$ where the label of $\triangle$ is contained in the label of $\triangle'$. More precisely, 
\begin{itemize}
\item 
The label of a simplex $\triangle\subset RI(S(\Lambda))$ which does not contain $v$ is the same as the label of the image of $\triangle$ in $RI(S(\Lambda'))$ under the isometry. 
\item
Otherwise, the label of the image of $\triangle$ is obtained by adding $v'$ to the label of $\triangle$.  
\end{itemize}
However, $RI(S(\Lambda))$ and $RI(S(\Lambda'))$ are not semi-isomorphic.
In $RI(S(\Lambda))$, the number of vertices whose assigned groups are quasi-isometric to $\mathbb{Z}\times\mathbb{F}$ is the same as the number of vertices of $\Lambda$.
In $RI(S(\Lambda'))$, the number of vertices whose assigned groups are quasi-isometric to $\mathbb{Z}\times\mathbb{F}$ is smaller than the number of vertices of $\Lambda$ since there are no vertices labelled by $st(v)$ and $st(v')$.

The construction of $I(X(\Lambda))$ from $RI(S(\Lambda))$ is the same as the construction of $I(X(\Lambda'))$ from $RI(S(\Lambda'))$ since the role of $v$ in the construction of $I(X(\Lambda))$ and the roles of $v$ and $v'$ in the construction of $I(X(\Lambda'))$ are the same.
This implies that $|I(X(\Lambda))|$ is isometric to $|I(X(\Lambda'))|$.

Since $\Out(A(\Lambda))$ is finite, for any vertex $\mathbf{u}\in RI(S(\Lambda))$ whose assigned group is quasi-isometric to $\mathbb{F}\times\mathbb{F}$ (if exists), there are edges $\bold{E}_i\subset RI(S(\Lambda))$ containing $\mathbf{u}$ such that the assigned groups of $\bold{E}_i$'s are quasi-isometric to $\mathbb{Z}\times\mathbb{Z}$.
However, for the vertex $\mathbf{u}'\in RI(S(\Lambda'))$ whose label is $\{v,v'\}\circ lk(v)$, the assigned group of every edge of $RI(S(\Lambda'))$ containing $\mathbf{u}'$ is quasi-isometric to $\mathbb{Z}\times\mathbb{F}$.
Therefore, $I(X(\Lambda'))$ are not semi-isomorphic to $I(X(\Lambda))$.
\end{proof}

Proposition \ref{ValencyOne} and \ref{1.5} can also be deduced from the result in \cite{Hua(b)} that if $\Out(A(\Lambda))$ is finite, then $A(\Lambda')$ is quasi-isometric to $A(\Lambda)$ if and only if $\Lambda'$ is obtained from $\Lambda$ by generalized star extensions.
\vspace{1mm}

Lastly, the following fact shows that subgraphs $\Lambda_i\leq\Lambda$ whose outer autormophism groups are finite can be considered as pieces which quasi-isometries preserve up to finite Hausdorff distance.
\begin{reptheorem}{1.7'}
Let $\Lambda$ ($\Lambda'$, respectively) be the union of subgraphs $\Lambda_i$ ($\Lambda'_{i'}$, respectively) with finite $\Out(A(\Lambda_i))$ ($\Out(A(\Lambda'_{i'}))$, respectively). Suppose that $\Lambda$ and $\Lambda'$ satisfy the following two common properties:
\begin{enumerate}
\item
The intersection of two of the subgraphs is either a vertex or empty and the intersection of at least three of the subgraphs is empty.
\item
There is no sequence of the subgraphs such that the first and last subgraphs are the same and the intersection of two consecutive subgraphs is a vertex.
\end{enumerate}
Let $\mathcal{I}$ ($\mathcal{I}'$, resp.) be the collection of isometry classes of $\Lambda_i$'s ($\Lambda'_{i'}$'s, resp.). If $\mathcal{I}$ and $\mathcal{I}'$ are different, then $A(\Lambda)$ and $A(\Lambda')$ are not quasi-isometric. 
\end{reptheorem}
\begin{proof}
Let $L_i$ be the subgraph of $\Lambda$ which is the union of the stars of vertices in $\Lambda_i$ and let $L'_{i'}$ be the subgraph of $\Lambda'$ which is defined by the same way from $\Lambda'_{i'}$.
By the assumption of $\Lambda$ and Proposition \ref{ValencyOne}, $|I(X(L_i))|$ ($|I(X(L'_{i'}))|$, resp.) is the union of copies of $|I(X(\Lambda_i))|$ ($|I(X(\Lambda'_{i'}))|$, resp.) such that the intersection of two copies is either a vertex or an empty set. 

The condition of $\Lambda$ implies that $RI(S(\Lambda))$ is the union of the subcomplexes $RI_i\subset RI(S(\Lambda))$ such that there is an injective semi-morphism $\rho_i:RI(S(L_i))\rightarrow RI_i\subset RI(S(\Lambda))$ and $RI_i\cap RI_{j}$ is either a separating vertex or an empty set for $i\neq j$. More precisely, if $\Lambda_i\cap\Lambda_{j}=\{v\}$, then $L_i\cap L_{j}=st(v)$ so that $RI_i\cap RI_{j}$ consists of the vertex labelled by $st(v)$. Otherwise, $RI_i\cap RI_{j}$ is empty.
By Lemma \ref{SeparatingVertex}, then, $I(X(\Lambda))$ is the union of copies of $I(X(L_i))$'s where the intersection of any two copies is either a separating vertex or an empty set. Thus, $|I(X(\Lambda))|$ is the union of copies of $|I(X(\Lambda_i))|$'s such that the intersection of any copies is either an empty set or a separating vertex. 
By the same reason, $|I(X(\Lambda'))|$ is the union of copies of $|I(X(\Lambda'_{i'}))|$'s such that the intersection of any two copies is either a separating vertex or an empty set.

If there is a quasi-isometry $\phi:A(\Lambda)\rightarrow A(\Lambda')$, then $\phi$ induces a semi-isomorphism $\Phi:I(X(\Lambda))\rightarrow I(X(\Lambda'))$ which sends separating vertices to separating vertices. Moreover, the image of a copy of $|I(X(\Lambda_i))|$ in $|I(X(\Lambda))|$ under the isometry $|\Phi|:|I(X(\Lambda))|\rightarrow |I(X(\Lambda'))|$ must be a copy of $|I(X(\Lambda'_j))|$ in $|I(X(\Lambda'))|$. By changing labels a little bit, the restriction of $\Phi$ to the subcomplex of $I(X(\Lambda))$ whose underlying complex is a copy of $I(X(L_i))$ induces a semi-isomorphism from $I(X(\Lambda_i))$ to $I(X(\Lambda'_{i'}))$ and by Theorem \ref{FiniteOut}, $\Lambda_i$ is isometric to $\Lambda'_{i'}$. Therefore, $\mathcal{I}$ is equal to $\mathcal{I}'$.
\end{proof}

%
% 5.2 Applications to Graph 2-braid groups
%

\subsection{Applications to Graph 2-braid groups}

Suppose that $\Gamma$ is a simplest cactus of type-S so that the boundary cycles are labelled by $a_1,\cdots,a_n$ in this order. Let $\Lambda_\Gamma$ be a complete graph with $n$ vertices such that vertices are labelled by $a_i$'s. Then there is a semi-isomorphism $RI(D_2(\Gamma))\rightarrow RI(S(\Lambda_\Gamma))$ such that the assigned group of each simplex $\triangle\subset RI(D_2(\Gamma))$ is isomorphic to the assigned group of the image of $\triangle$ in $RI(S(\Lambda_\Gamma))$ under this semi-isomorphism. Moreover, there is also a semi-isomorphism between $I(X(\Lambda_\Gamma))$ and a component of $I(\overline{D_2(\Gamma)})$.

\begin{proposition}\label{ScomponentGraph}
Let $\Gamma$ be a simplest cactus of type-S and $\Lambda_\Gamma$ the graph obtained from $\Gamma$ as above. Then a component of $I(\overline{D_2(\Gamma)})$ is semi-isomorphic to $I(X(\Lambda_\Gamma))$. 
\end{proposition}
\begin{proof}
By Proposition \ref{Scomponent}, $|RI(D_2(\Gamma))|$ is contractible so that $|RI(S(\Lambda_\Gamma))|$ is also contractible. By Lemma \ref{Decomposition}, therefore, the proposition holds.
\end{proof}

When $\Gamma$ is a simplest cactus of type-M, there is no obvious way to find $\Lambda_\Gamma$ such that $RI(D_2(\Gamma))$ is semi-isomorphic to $RI(S(\Lambda_\Gamma))$.
Even when such a graph $\Lambda_\Gamma$ exists, $I(X(\Lambda_\Gamma))$ may not be semi-isomorphic to a component of $I(\overline{D_2(\Gamma)})$. The underlying complex of a lift of $RI(D_2(\Gamma))$ is simply connected since $RI(D_2(\Gamma))$ is the developable complex of groups so that its lift is the development (Corollary \ref{Development}). On the other hand, $|I(X(\Lambda_\Gamma))|$ is not simply connected since $|RI(S(\Lambda_\Gamma))|$ is not simply connected (Proposition \ref{ContractibleRI}). 
However, if $\Gamma$ is in relatively simple cases, then there is an alternative graph $\Lambda'_\Gamma$ (not isometric to $\Lambda_\Gamma$) such that $I(X(\Lambda'_\Gamma))$ is semi-isomorphic to a component of $I(\overline{D_2(\Gamma)})$. 

\begin{proposition}\label{StarN}
Consider the star tree with central vertex of valency $k\geq 3$ and let $\mathcal{O}_k$ be the graph obtained from this star tree by attaching a 3-cycle to each leaf ($\mathcal{O}_3$ and $\mathcal{O}_4$ are given in Figure \ref{Ex1} and \ref{Ex2}, respectively).
Then $RI(D_2(\mathcal{O}_k))$ is a graph of groups satisfying the conditions in Theorem \ref{GraphofGroups}.
\end{proposition}
\begin{proof}
Label boundary cycles in $\mathcal{O}_k$ by $a_1,\cdots,a_k$. Then there are $2k$ vertices in $RI(D_2(\mathcal{O}_k))$; vertices $\mathbf{u}_i\in RI(D_2(\mathcal{O}_k))$ labelled by $\{a_i\}\times (\{a_1,\cdots,a_k\}-\{a_i\})$ and vertices $\mathbf{u}^s_i$ labelled by $(\{a_1,\cdots,a_k\}-\{a_i\}) \times \{a_i\}$ for $i=1,\cdots,k$.
As in Example \ref{T_3Ex} and \ref{T_4Ex}, $(k-1)$ edges are attached to each $\mathbf{u}_i$ such that their other endpoints are $\mathbf{u}^s_j$ for $j\in \{1,\cdots,k\}-\{i\}$.
Since $\mathcal{O}_k$ is a cactus, $RI(D_2(\mathcal{O}_k))$ can be considered as the graph of groups decomposition of $SPB_2(\mathcal{O}_k)$ by Theorem \ref{pi_1-injective}.
\end{proof}

\begin{corollary}\label{QItoAT}
Suppose that $\mathcal{O}_k$ is given as in Proposition \ref{StarN}. Then $SPB_2(\mathcal{O}_k)$ is quasi-isometric to $A(T)$ for a tree $T$ of diameter $\geq 3$.
\end{corollary}
\begin{proof}
Since $SPB_2(\mathcal{O}_k)$ is the fundamental group of $RI(D_2(\mathcal{O}_k))$ (as a graph of groups), by Theorem \ref{GraphofGroups} and Proposition \ref{StarN}, the corollary holds.
\end{proof}

Based on the fact in \cite{PW02}, we can know that $\mathcal{R}(D_2(\Gamma))$ is the main object when we study the quasi-isometry rigidity of $\overline{D_2(\Gamma)}$.  

\begin{theorem}[\cite{PW02}]\label{PW02}
Let $G$ be any finitely generated group of order at least three. Then, $G*G$, $G*\mathbb{Z}$, $G*\mathbb{F}_n$ are all quasi-isometric.
\end{theorem}

\begin{proposition}\label{QItoProd}
Suppose that $\Gamma$ is a simplest cactus. If $\Gamma$ is of type-M, Then $PB_2(\Gamma)$ is quasi-isometric to the free product of $SPB_2(\Gamma)$ and $\mathbb{Z}$.
\end{proposition}
\begin{proof}
In the proof of Corollary \ref{Development}, we showed that if $\Gamma$ is of type-M (type-S, resp.) $PB_2(\Gamma)$ ($B_2(\Gamma)$, resp.) is isomorphic to $SPB_2(\Gamma)*\mathbb{F}_n$ where $\mathbb{F}_n$ is the free group of rank $n\ge 0$ ($\mathbb{F}_0$ means that it is trivial). 
In order to prove the proposition, we need to show that $n$ is $\geq 1$.

Let $a$ be an outermost boundary cycle in $\Gamma$, i.e. $v=\Psi(a)$ is a leaf in $\mathcal{K}_\Gamma$. Since $v$ is a vertex of valency $\geq 3$ in $\Gamma$, there are three vertices $v_1,v_2,v_3\in \Gamma$ adjacent to $v$ such that $v_1$ and $v_2$ are in $a$ but $v_3$ is not. 
Then the movement corresponding to an Y-exchange $(v_1,v_2,v_3;v)$ (see Example \ref{Tripod}) is not contained in $\mathcal{R}(D_2(\Gamma))$ so that any conjugation of the Y-exchange $(v_1,v_2,v_3;v)$ is not contained in $SPB_2(\Gamma)$. This means that $SPB_2(\Gamma)$ is a proper subgroup of $PB_2(\Gamma)$ ($B_2(\Gamma)$, resp.) when $\Gamma$ is of type-M (type-S, resp.), and thus, $n$ must be $\ge 1$. By Theorem \ref{PW02}, therefore, $PB_2(\Gamma)$ is quasi-isometric to $SPB_2(\Gamma)*\mathbb{Z}$.
\end{proof}

\begin{proposition}\label{QItoRAAG}
Let $\mathcal{O}_k$ be the graph in Proposition \ref{StarN}. Let $T$ be a tree of diameter $\geq3$. Then, $PB_2(\mathcal{O}_k)$ is quasi-isometric to $A(T)*\mathbb{Z}$.
\end{proposition}
\begin{proof}
By Proposition \ref{QItoProd}, $PB_2(\mathcal{O}_k)$ is quasi-isometric to the free product of $SPB_2(\mathcal{O}_k)$ and $\mathbb{Z}$.
Since $SPB_2(\mathcal{O}_k)$ is quasi-isometric to $A(T)$ by Corollary \ref{QItoAT}, $PB_2(\mathcal{O}_k)$ is quasi-isometric to $A(T)*\mathbb{Z}$ by Theorem \ref{PW02}.
\end{proof}

However, there also exists a simplest cactus of type-M whose graph braid group is not quasi-isometric to any RAAG. 
To show the existence of such a cactus, we recall the following two observations:
\begin{enumerate}
\item
Suppose that $\Lambda$ consists of finitely many components $\Lambda_1,\cdots,\Lambda_n$ each of which is triangle-free. Then $I(X(\Lambda))$ consists of infinitely copies of $I(X(\Lambda_i))$ for $i=1,\cdots,n$.
\item If $\overline{D_2(\Gamma)}$ is quasi-isometric to $X(\Lambda)$, by Theorem \ref{FlattoFlat}, $\Lambda$ must be triangle-free and contain at least one edge since the dimension of a top dimensional flat in $\overline{D_2(\Gamma)}$ is at most two.
\end{enumerate}
By these observations, we will find a cactus $\Gamma$ such that $I(\overline{D_2(\Gamma)})$ has a component which is not semi-isomorphic to $I(X(\Lambda))$ for any triangle-free connected graph $\Lambda$. 

\begin{proposition}\label{NotQI}
Let $\mathcal{O}'_4$ be the graph in Example \ref{T'_4Ex}. Then there exists no triangle-free graph $\Lambda$ such that $I(X(\Lambda))$ is semi-isomorphic to a component $I_0$ of $I(\overline{D_2(\mathcal{O}'_4)})$. In particular, $PB_2(\mathcal{O}'_4)$ is not quasi-isometric to any RAAGs.
\end{proposition}
\begin{proof}
Assume that there exists a triangle-free graph $\Lambda$ such that $I(X(\Lambda))$ is semi-isomorphic to $I_0$ with a semi-isomorphism $\Phi:I_0\rightarrow I(X(\Lambda))$.
Since $\mathcal{O}'_4$ is a cactus of type-M, by corollary \ref{Development}, $|I_0|$ is simply connected. It means that $|I(X(\Lambda))|$ is simply connected so that by Proposition \ref{ContractibleRI}, $|RI(S(\Lambda))|$ is also simply connected. 

Let $\rho_S:I(X(\Lambda))\rightarrow RI(S(\Lambda))$ and $\rho_D:I_0\rightarrow RI(D_2(\mathcal{O}'_4))$ be the canonical quotient morphisms.
Let $\alpha=(\bold{E}_1,\triangle_1,\bold{E}_2,\triangle_2)$ be a sequence of simplices of $RI(D_2(\mathcal{O}'_4))$ where $\bold{E}_1$ and $\bold{E}_2$ are labelled by $\{a_1\}\times\{a_2\}$ and $\{a_4\}\times\{a_3\}$, respectively, and $\triangle_1$ and $\triangle_2$ are labelled by $\{a_1\}\times\{a_3\}$ and $\{a_4\}\times\{a_2\}$, respectively.
Let $$\overline{\alpha}=(\cdots,\overline{\bold{E}}_1,\overline{\triangle}_1,\overline{\bold{E}}_2,\overline{\triangle}_2,\cdots)$$ be an infinite sequence of consecutive simplices in $I_0$ such that $\rho_D(\overline{\bold{E}}_{2k+1})=\bold{E}_1$, $\rho_D(\overline{\bold{E}}_{2k})=\bold{E}_2$, $\rho_D(\overline{\triangle}_{2k+1})=\triangle_1$ and $\rho_D(\overline{\triangle}_{2k})=\triangle_2$. 
Let $\mathbf{v}_n=\overline{\bold{E}}_n\cap\overline{\triangle}_n$ and $\mathbf{w}_n=\overline{\triangle}_n\cap\overline{\bold{E}}_{n+1}$. 
Then $\mathbf{v}_n$, $\mathbf{w}_n$ are separating vertices in $I_0$ which have the following property: For any maximal edge $\overline{\bold{E}}\subset I_0$ and any 2-simplex $\overline{\triangle}\subset I_0$ both of which contain $\mathbf{v}_n$ (or $\mathbf{w}_n$), there is no path joining the other vertex in $\overline{\bold{E}}$ and any other vertex in $\overline{\triangle}$ without passing through $\mathbf{v}_n$ (or $\mathbf{w}_n$).
Since $\Phi(\mathbf{v}_n)$ and $\Phi(\mathbf{w}_n)$ are separating vertices in $I(X(\Lambda))$ which have this property, $\rho_S(\Phi(\mathbf{v}_n))$ and $\rho_S(\Phi(\mathbf{w}_n))$ are separating vertices in $RI(S(\Lambda))$.
For if $\rho_S(\Phi(\mathbf{v}_n))$ (or $\rho_S(\Phi(\mathbf{w}_n))$) was not a separating vertex in $RI(S(\Lambda))$, then $\Phi(\mathbf{v}_n)$ (or $\Phi(\mathbf{w}_n)$) would not have the above property from the fact that $I(X(\Lambda))$ is covered by copies of $RI(S(\Lambda))$.

Let $\overline{\alpha}'=\Phi(\overline{\alpha})$ and consider a half $\overline{\alpha}'_{+}$ of $\overline{\alpha}'$, 
$$\overline{\alpha}'_+=(\Phi(\overline{\bold{E}}_1),\Phi(\overline{\triangle}_1),\Phi(\overline{\bold{E}}_2),\Phi(\overline{\triangle}_2),\cdots).$$
Then $\overline{\alpha}'_{+}$ induces the sequence of seperating vertices in $RI(S(\Lambda))$, 
$$(\rho_{S}(\Phi(\mathbf{v}_1)),\rho_{S}(\Phi(\mathbf{w}_1)),\rho_{S}(\Phi(\mathbf{v}_2)),\rho_{S}(\Phi(\mathbf{w}_2)),\cdots).$$ 
Since $|RI(S(\Lambda))|$ is finite, there exists $m>0$ such that $\rho_{S}(\Phi(\mathbf{v}_m))=\rho_{S}(\Phi(\mathbf{v}_{m+1}))$ or $\rho_{S}(\Phi(\mathbf{w}_{m}))=\rho_{S}(\Phi(\mathbf{w}_{m+1}))$. 
But it is impossible since $\rho_S(\Phi(\overline{\bold{E}}_i))$ and $\rho_S(\Phi(\overline{\triangle}_i))$ are simplices of $RI(S(\Lambda))$ and $|RI(S(\Lambda))|$ is a simplicial complex, a contradiction.
Therefore, there is no such $\Lambda$.
\end{proof}

Let $\mathcal{O}'_{4,n}$ be the graph obtained from $\mathcal{O}'_4$ as follows: after adding $n$ vertices on the interior of the edge of $\mathcal{O}'_4$ not meeting boundary cycles, one 3-cycle is attached to each added vertex and the added 3-cycles are labelled by $a_5,\cdots,a_{n+4}$ (Figure \ref{4,n}). 
In $RI(D_2(\mathcal{O}'_{4,n}))$, there are eight maximal $(n+2)$-simplices and four maximal edges. The $(n+2)$-simplices are labelled by $\{a_1\}\times \{a_3\}$, $\{a_1\}\times \{a_4\}$, $\{a_2\}\times \{a_3\}$, $\{a_2\}\times \{a_4\}$ and their switched ones; the first four of them share the $n$-simplex labelled by $\{a_1,a_2\}\times\{a_3,a_4\}$ and the last four share the $n$-simplex labelled by $\{a_3,a_4\}\times\{a_1,a_2\}$.
The maximal edges are labelled by $\{a_1\}\times \{a_2\}$, $\{a_3\}\times \{a_4\}$, $\{a_2\}\times \{a_1\}$, $\{a_4\}\times \{a_3\}$. 
Indeed, $RI(D_2(\mathcal{O}'_{4,n}))$ is obtained from $RI(D_2(\mathcal{O}'_4))$ by replacing some simplices: the vertices in $RI(D_2(\mathcal{O}'_4))$ labelled by $\{a_1,a_2\}\times\{a_3,a_4\}$ and $\{a_3,a_4\}\times\{a_1,a_2\}$ are replaced by $n$-simplices and the maximal simplices of $RI(D_2(\mathcal{O}'_4))$ labelled by $\{a_i\}\times \{a_j\}$ or $\{a_j\}\times \{a_i\}$ are replaced by $(n+2)$-simplices for $i\in\{1,2\}$, $j\in\{3,4\}$. 
See Figure \ref{4,1} for $RI(D_2(\mathcal{O}'_{4,1}))$.
Similarly to how we prove the previous proposition, we can show that $PB_2(\mathcal{O}'_{4,n})$ is not quasi-isometric to RAAGs.

\begin{figure}[h]
\begin{tikzpicture}
\tikzstyle{every node}=[draw,circle,fill=black,minimum size=4pt,inner sep=0pt]
  \node[anchor=center] (x1) at (0,0){};
  \node[anchor=center] (y1) at (6,0){};
  \node[anchor=center] (a1) at (-0.5,0.73){};
  \node[anchor=center] (a2) at (-1.5,0.73){};  
  \node[anchor=center] (a3) at (-0.5,1.73){};  
  \node[anchor=center] (b1) at (-0.5,-0.73){};
  \node[anchor=center] (b2) at (-1.5,-0.73){};  
  \node[anchor=center] (b3) at (-0.5,-1.73){};  
  \node[anchor=center] (c1) at (6.5,0.73){};
  \node[anchor=center] (c2) at (7.5,0.73){};  
  \node[anchor=center] (c3) at (6.5,1.73){};  
  \node[anchor=center] (d1) at (6.5,-0.73){};
  \node[anchor=center] (d2) at (6.5,-1.73){};  
  \node[anchor=center] (d3) at (7.5,-0.73){};  
  \node[anchor=center] (e1) at (1.5,0){};
  \node[anchor=center] (e2) at (1,0.73){};  
  \node[anchor=center] (e3) at (2,0.73){};  
  \node[anchor=center] (f1) at (4.5,0){};
  \node[anchor=center] (f2) at (5,0.73){};  
  \node[anchor=center] (f3) at (4,0.73){};  
  \node[minimum size=2pt] (g1) at (2.7,0.63){};
  \node[minimum size=2pt] (g2) at (3,0.63){};  
  \node[minimum size=2pt] (g3) at (3.3,0.63){};

  \draw (a1.center) -- (a2.center) -- (a3.center) -- cycle;  
  \draw (b1.center) -- (b2.center) -- (b3.center) -- cycle;
  \draw (c1.center) -- (c2.center) -- (c3.center) -- cycle;
  \draw (d1.center) -- (d2.center) -- (d3.center) -- cycle;  
  \draw (e1.center) -- (e2.center) -- (e3.center) -- cycle;  
  \draw (f1.center) -- (f2.center) -- (f3.center) -- cycle;  

  \foreach \from/\to in {x1/a1,x1/b1,y1/c1,y1/d1,x1/y1}
  \draw (\from) -- (\to);
  
\end{tikzpicture}
\caption{$\mathcal{O}'_{4,n}$ with $(n+4)$ boundary cycles.}
\label{4,n}
\end{figure}
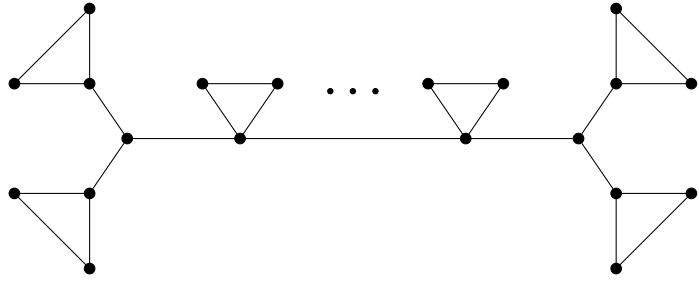

\begin{figure}[h]
\begin{tikzpicture}
\tikzstyle{point}=[circle,thick,draw=black,fill=black,inner sep=0pt,minimum width=3.5pt,minimum height=4pt]
\draw [gray, -triangle 60  ] (2.5,1.4) -- (0.3,0.4) node (x2) [minimum size=0pt]{};
\draw [gray, -triangle 60  ] (-2.9,1.1) -- (-0.3,-0.2) node (x3) [minimum size=0pt]{};

\node (x)[point,label={[shift={(2.8,0.9)}]:$\{a_1,a_2,a_5\}\times\{a_3,a_4\}$}] at (0.2,0.3) {};
\node (x1)[point,label={[shift={(-3.4,1.2)}]:$\{a_1,a_2\}\times\{a_3,a_4,a_5\}$}] at (-0.2,-0.3) {};
\node (a)[point,label={[label distance=-0.1cm]above:$\{a_1\}\times\{a_2,a_3,a_4,a_5\}$}] at (-0.7,1.6) {};
\node (b)[point,label={[label distance=0cm]left:$\{a_1,a_2,a_4,a_5\}\times\{a_3\}$}] at (-3.5,0) {};
\node (c)[point,label={[shift={(1.7,-0.55)}]:$\{a_2\}\times\{a_1,a_3,a_4,a_5\}$}] at (0.7,-1.6) {};
\node (d)[point,label={[label distance=-0.1cm]right:$\{a_1,a_2,a_3,a_5\}\times\{a_4\}$}] at (3.5,0) {};

    \begin{scope}[yshift=-4.5cm]
    \node (y)[point] at (0.2,0.3) {};
    \node (y1)[point] at (-0.2,-0.3) {};
    \node (e)[point] at (-0.7,1.6) {};
    \node (f)[point] at (-3.5,0) {};
    \node (g)[point] at (0.7,-1.6) {};
    \node (h)[point] at (3.5,0) {};
    \end{scope}

    \draw[pattern=dots] (a.center) -- (x.center) -- (b.center) -- cycle;
    \draw[pattern=dots] (a.center) -- (x1.center) -- (b.center) -- cycle;
    \draw[pattern=dots] (a.center) -- (x1.center) -- (x.center) -- cycle;
    \draw[pattern=dots] (b.center) -- (x1.center) -- (x.center) -- cycle;
    \draw[pattern=mydots] (b.center) -- (x.center) -- (c.center) -- cycle;
    \draw[pattern=mydots] (b.center) -- (x1.center) -- (c.center) -- cycle;
    \draw[pattern=mynewdots] (c.center) -- (x.center) -- (d.center) -- cycle;
    \draw[pattern=mynewdots] (c.center) -- (x1.center) -- (d.center) -- cycle;
    \draw[pattern=mynewdots] (c.center) -- (x.center) -- (x1.center) -- cycle;
    \draw[pattern=mynewdots] (d.center) -- (x1.center) -- (x.center) -- cycle;
    \draw[pattern=verynewdots] (d.center) -- (x.center) -- (a.center) -- cycle;
    \draw[pattern=verynewdots] (d.center) -- (x1.center) -- (a.center) -- cycle;
    \foreach \from/\to in {x/x1,y/y1}
  \draw (\from) -- (\to);
	\draw [red] plot [smooth] coordinates {(-0.7,1.6) (-1.1,-0.65) (-0.7,-2.9)};
	\draw [blue] plot [smooth] coordinates {(-3.5,0) (-4.3,-2.25) (-3.5,-4.5)};
	\draw [red] plot [smooth] coordinates {(0.7,-1.6) (1.1,-3.85) (0.7,-6.1)};
	\draw [blue] plot [smooth] coordinates {(3.5,0) (4.3,-2.25) (3.5,-4.5)};

    \draw[pattern=mynewdots] (e.center) -- (y.center) -- (f.center) -- cycle;
    \draw[pattern=mynewdots] (e.center) -- (y1.center) -- (f.center) -- cycle;
    \draw[pattern=mynewdots] (e.center) -- (y.center) -- (y1.center) -- cycle;
    \draw[pattern=mynewdots] (f.center) -- (y1.center) -- (y.center) -- cycle;
    \draw[pattern=verynewdots] (f.center) -- (y.center) -- (g.center) -- cycle;
    \draw[pattern=verynewdots] (f.center) -- (y1.center) -- (g.center) -- cycle;
    \draw[pattern=dots] (g.center) -- (y.center) -- (h.center) -- cycle;
    \draw[pattern=dots] (g.center) -- (y1.center) -- (h.center) -- cycle;
    \draw[pattern=dots] (g.center) -- (y.center) -- (y1.center) -- cycle;
    \draw[pattern=dots] (h.center) -- (y1.center) -- (y.center) -- cycle;
    \draw[pattern=mydots] (h.center) -- (y.center) -- (e.center) -- cycle;
    \draw[pattern=mydots] (h.center) -- (y1.center) -- (e.center) -- cycle;

%    \draw[pattern=dots] (d.center) -- (e.center) -- (f.center) -- cycle;
%    \draw (p.center) -- (d.center);
\end{tikzpicture}
\caption{$RI(D_2(\mathcal{O}'_{4,1}))$.}
\label{4,1}
\end{figure}
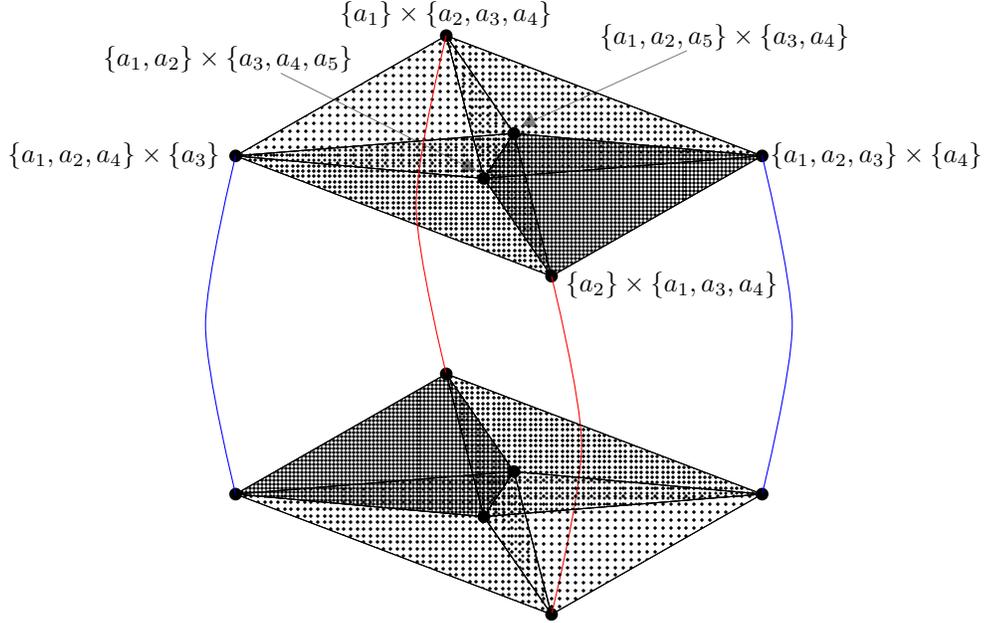

\begin{corollary}\label{NotQItoRAAG}
Let $\mathcal{O}'_{4,n}$ be given as above. Then there is no triangle-free graph $\Lambda$ such that $I(X(\Lambda))$ is semi-isomorphic to a component $I_0$ of $I(\overline{D_2(\mathcal{O}'_{4,n})})$. In particular, $PB_2(\mathcal{O}'_{4,n})$ is not quasi-isometric to any RAAGs.
\end{corollary}
\begin{proof}
Assume that there exists a triangle-free graph $\Lambda$ such that $I(X(\Lambda))$ is semi-isomorphic to $I_0$ with a semi-isomorphism $\Phi:I_0\rightarrow I(X(\Lambda))$.
As in Proposition \ref{NotQI}, $|RI(S(\Lambda))|$ is simply connected since $\mathcal{O}'_{4,n}$ is a cactus of type-M. 

Let $\rho_S:I(X(\Lambda))\rightarrow RI(S(\Lambda))$ and $\rho_D:I_0\rightarrow RI(D_2(\mathcal{O}'_{4,n}))$ be the canonical quotient morphisms.
Let $\alpha=(\bold{E}_1,\triangle_1,\bold{E}_2,\triangle_2)$ be a sequence of simplices of $RI(D_2(\mathcal{O}'_{4,n}))$ where $\bold{E}_1$ and $\bold{E}_2$ are edges labelled by $\{a_1\}\times\{a_2\}$ and $\{a_4\}\times\{a_3\}$, respectively, and $\triangle_1$ and $\triangle_2$ are $(n+2)$-simplices labelled by $\{a_1\}\times\{a_3\}$ and $\{a_4\}\times\{a_2\}$, respectively.
Let $$\overline{\alpha}=(\cdots,\overline{\bold{E}}_1,\overline{\triangle}_1,\overline{\bold{E}}_2,\overline{\triangle}_2,\cdots)$$ be an infinite sequence of consecutive simplices in $I_0$ such that $\rho_D(\overline{\bold{E}}_{2k+1})=\bold{E}_1$, $\rho_D(\overline{\bold{E}}_{2k})=\bold{E}_2$, $\rho_D(\overline{\triangle}_{2k+1})=\triangle_1$ and $\rho_D(\overline{\triangle}_{2k})=\triangle_2$. Let $\mathbf{v}_n=\overline{\bold{E}}_n\cap\overline{\triangle}_n$ and $\mathbf{w}_n=\overline{\triangle}_n\cap\overline{E}_{n+1}$. 
Then $\mathbf{v}_n$, $\mathbf{w}_n$ are separating vertices in $I_0$ which have the following property: For any maximal edge $\overline{\bold{E}}\subset I_0$ and any $(n+2)$-simplex $\overline{\triangle}\subset I_0$ both of which contain $\mathbf{v}_n$ (or $\mathbf{w}_n$), there is no path joining the other vertex of $\overline{\bold{E}}$ and any other vertex of $\overline{\triangle}$ without passing through $\mathbf{v}_n$ (or $\mathbf{w}_n$).
Since $\Phi(\mathbf{v}_n)$ and $\Phi(\mathbf{w}_n)$ must have this property, $\rho_S(\Phi(\mathbf{v}_n))$ and $\rho_S(\Phi(\mathbf{w}_n))$ are separating vertices in $RI(S(\Lambda))$.
For if $\rho_S(\Phi(\mathbf{v}_n))$ (or $\rho_S(\Phi(\mathbf{w}_n))$) was not a separating vertex in $RI(S(\Lambda))$, then $\Phi(\mathbf{v}_n)$ (or $\Phi(\mathbf{w}_n)$) would not have the above property from the fact that $I(X(\Lambda))$ is covered by copies of $RI(S(\Lambda))$.

As we did in the proof of Proposition \ref{NotQI}, from the fact that $|RI(S(\Lambda))|$ is a finite simplicial complex, the sequence of seperating vertices in $RI(S(\Lambda))$, $$(\rho_{S}(\Phi(\mathbf{v}_1)),\rho_{S}(\Phi(\mathbf{w}_1)),\rho_{S}(\Phi(\mathbf{v}_2)),\rho_{S}(\Phi(\mathbf{w}_2)),\cdots)$$
yields a contradiction. Therefore, there is no such $\Lambda$.
\end{proof}

By combining Proposition \ref{QItoRAAG} and Corollary \ref{NotQItoRAAG}, the main theorem about planar graph 2-braid groups is proved.

\begin{reptheorem}{MainThm}
There are infinitely many cacti whose graph 2-braid groups are quasi-isometric to RAAGs.
There are also infinitely many cacti whose graph 2-braid groups are NOT quasi-isometric to RAAGs.
\end{reptheorem}

\medskip
\bibliographystyle{alpha} 
\bibliography{Ref.bib}

\end{document}